\documentclass[12pt,a4paper]{amsart}
\makeatletter
\renewcommand\normalsize{%
    \@setfontsize\normalsize{11.7}{14pt plus .3pt minus .3pt}%
    \abovedisplayskip 10\p@ \@plus4\p@ \@minus4\p@
    \abovedisplayshortskip 6\p@ \@plus2\p@
    \belowdisplayshortskip 6\p@ \@plus2\p@
    \belowdisplayskip \abovedisplayskip}
\renewcommand\small{%
    \@setfontsize\small{9.5}{12\p@ plus .2\p@ minus .2\p@}%
    \abovedisplayskip 8.5\p@ \@plus4\p@ \@minus1\p@
    \belowdisplayskip \abovedisplayskip
    \abovedisplayshortskip \abovedisplayskip
    \belowdisplayshortskip \abovedisplayskip}
\renewcommand\footnotesize{%
    \@setfontsize\footnotesize{8.5}{9.25\p@ plus .1pt minus .1pt}
    \abovedisplayskip 6\p@ \@plus4\p@ \@minus1\p@
    \belowdisplayskip \abovedisplayskip
    \abovedisplayshortskip \abovedisplayskip
    \belowdisplayshortskip \abovedisplayskip}
\ifdefined\pdfpagewidth
\else
\fi
\calclayout
\makeatother

\usepackage{epsfig,color}

\makeatother

\theoremstyle{definition}

\def\fnum{equation} 
\newtheorem{Thm}[\fnum]{Theorem}
\newtheorem{Cor}[\fnum]{Corollary}

\newtheorem{Lem}[\fnum]{Lemma}

\newtheorem{Pro}[\fnum]{Proposition}

\numberwithin{equation}{section}

\newcommand{\Vol}{{\text{Vol}}}

\newcommand{\R}{{\text{R}}}

\newcommand{\nn}{{\bf{n}}}
\newcommand{\Ric}{{\text{Ric}}}

\newcommand{\Tr}{{\text{Tr}}}

\newcommand{\diam}{{\text {diam}}}
\newcommand{\dist}{{\text {dist}}}

\newcommand{\cK}{{\mathcal{K}}}

\newcommand{\Hess}{{\text {Hess}}}

\def\RR{{\bold R}}
\def\RP{{\bold{RP}}}
\def\SS{{\bold S}}

\newcommand{\dv}{{\text {div}}}
\newcommand{\e}{{\text {e}}}

\newcommand{\bg}{{\bar{g}}}

\newcommand{\bG}{{\bar{\Gamma}}}

\newcommand{\bb}{{\hat{h}}}

\newcommand{\cC}{{\mathcal{C}}}
\newcommand{\cE}{{\mathcal{E}}}

\newcommand{\cI}{{\mathcal{I}}}

\newcommand{\cB}{{\mathcal{B}}}

\newcommand{\cL}{{\mathcal{L}}}

\newcommand{\cH}{{\mathcal{H}}}

\newcommand{\cP}{{\mathcal{P}}}

\newcommand{\cR}{{\mathcal{R}}}

\newcommand{\cQ}{{\mathcal{Q}}}

\newcommand{\eqr}[1]{(\ref{#1})}

\title{Singularities of Ricci flow and diffeomorphisms}

\author{Tobias Holck Colding}%
\address{MIT, Dept. of Math.\\
77 Massachusetts Avenue, Cambridge, MA 02139-4307.}
\author{William P. Minicozzi II}%

\thanks{The  authors
were partially supported by NSF  DMS Grants   2405393 and 2304684.}

\email{colding@math.mit.edu and minicozz@math.mit.edu}

\begin{document}

\maketitle

  \begin{abstract}
    We solve a well-known open problem in Ricci flow: Strong rigidity
   of cylinders.  Strong rigidity is an illustration of a {\it shrinker principle} that uniqueness radiates out from a compact set.
   It implies that if one tangent flow at a future singular point is a   cylinder, then all tangent flows  are.
   
   At the heart of this problem in Ricci flow is 
    comparing and recognizing metrics.  This can be rather complicated
because of the group of diffeomorphisms.  Two metrics, that could even be the same, could look completely different in different coordinates.  This is the gauge problem.   Often it can be avoided if one uses some additional structure of the particular situation.  
The gauge problem is subtle for non-compact spaces without additional structure.

We solve this gauge problem by solving a nonlinear system of PDEs.  The PDE produces a diffeomorphism that fixes an appropriate gauge in the spirit of the slice theorem for group actions.  We then show optimal bounds for the displacement function of the diffeomorphism.   Strong rigidity relies on gauge fixing and several other new ideas.  One of these is ``propagation of almost splitting''\!,
 another is quadratic rigidity in the right gauge, and a third is an optimal polynomial growth bound for PDEs that holds in great generality.
  \end{abstract}

\section{Introduction}

To prove strong rigidity of cylinders, we need to be able
 to recognize a metric  from rough information on a compact set without any canonical coordinate system.
Recognizing a metric without canonical coordinates is a common problem in many questions.  

 Suppose we have two weighted manifolds $(M_i,g_i,f_i)$ for $i=1$, $2$ satisfying some PDE.
 Assume that on a large, but compact set, the manifolds $M_i$, metrics $g_i$ and weights $\e^{-f_i}$ almost agree after identification by a diffeomorphism.  
\begin{itemize}
\item Is there a diffeomorphism so that  the metrics and weights are the same everywhere?   
\end{itemize}
A major obstacle for understanding this 
  is the infinite dimensional gauge group
 \footnote{``By fixing a gauge (thus breaking or spending the gauge symmetry), the model becomes something easier to analyse mathematically....  Deciding exactly how to fix a gauge (or whether one should spend the gauge symmetry at all) is a key question in the analysis of gauge theories, and one that often requires the input of geometric ideas and intuition into that analysis.'', \cite{Tt}.}
   of diffeomorphisms:
\begin{itemize}
\item Two metrics, that could even be the same, could look very different in different coordinates.
\end{itemize}

In some situations the gauge problem can be avoided if there is some additional structure.  A classical example 
is the  Killing-Hopf theorem that classifies constant curvature metrics.   This classification uses that the curvature tensor is constant to construct a ``canonical'' isometry between the two spaces.
In general, the gauge problem can be solved when there is strong asymptotic decay and circumvented when the space is characterized in a coordinate-free way, such as a large symmetry group,  the vanishing of a special tensor, or  a strong curvature condition.

In the problems we will be interested in,  the manifold will be non-compact and we will not have any special structure.
Thus, we will be forced to deal with the gauge problem head on.  We do this by solving a nonlinear PDE to get a diffeomorphism that fixes the gauge in the spirit of the slice theorem for group actions.  Since the manifold is non-compact, we need strong bounds for the 
displacement function of the diffeomorphism.

\subsection{Where do  questions like these arise?}  

Problems about identifying spaces occur in many different situations.  The one we are interested in 
is from Ricci flow.  
A one parameter family $(M,g(t))$ of manifolds flows by the Ricci flow  if 
 $g_t=-2\,\Ric_{g(t)}\,  ,$ 
where $\Ric_{g(t)}$ is the Ricci curvature of the evolving metric $g(t)$ and $g_t$ is the time derivative of the metric,{\footnote{The gauge group is known to cause difficulties in Ricci flow.  The invariance under the group makes the equation degenerate so standard parabolic techniques do not apply.  The Ricci-DeTurck flow deals with this by fixing an arbitrary initial gauge and then solving coupled equations for evolving metrics and gauges to get a parabolic PDE.  The arbitrary initial choice of gauge makes this unsuitable for the problems we are interested in since the gauge has to be right to compare two solutions.}}
 \cite{H}.

The key to understand Ricci flow is the singularities that form. The simplest singularity is a homothetically shrinking sphere that becomes extinct at a point.  The product of a sphere with $\RR$ gives a shrinking cylinder.  This singularity is called a neck pinch.  It is more complicated than the spherical extinction.  In dimension three, spherical extinctions and neck pinches are essentially the only singularities.  
Adding another $\RR$ factor gives a cylinder with a two-dimensional Euclidean factor; this singularity is the so-called bubble sheet  that is only recently partially understood.  With each additional $\RR$ factor, the singularities become more complicated and the sets where they occur are larger.

 A triple $(M,g,f)$ of a manifold $M$, metric $g$ and function $f$  
is a gradient shrinking Ricci soliton (or {\it shrinker}) if 
\begin{align}
\Ric + \Hess_f = \frac{1}{2} \, g \, .
\end{align}
  Shrinkers give special solutions of the Ricci flow that evolve by rescaling up to diffeomorphism and are  singularity models.  They arise as time-slices of limits of rescalings (magnifications) of the flow around a fixed future singular point in space-time.   Such limits are said to be  tangent flows at the singularity.  Even when the evolving manifold 
   is compact, the shrinker is typically non-compact and the convergence is on compact subsets.
  Shrinkers also arise in other important ways, such as  blowdowns from $-\infty$ for ancient flows.   Ancient flows are flows that have existed for all prior times.  All blowups are ancient flows, but not every blow up gives a shrinker.  
  Shrinkers are   key singularities in Ricci flow and will be our focus.

  Among shrinkers,  cylinders are particularly important.    Indeed, the
   Almgren-Federer-White dimension reduction,  cf.~\cite{W,KL2,BaK1,BaK2},  divides the singular set into strata whose dimension is the dimension of translation-invariance of the blowup.  Thus, the top strata is the largest part of the singular set.
   For Ricci flow, this
   suggests:
\begin{itemize}
\item Top strata of the singular set corresponds to points where the blowup is $\RR^{n-2} \times N^2$.
\item The next strata consists of points where the blowup is $\RR^{n-3} \times N^3$.
\end{itemize}
The $N$'s  are themselves shrinkers and have been classified in low dimensions by Cao-Chen-Zhu, Hamilton, Ivey, 
Naber, Ni-Wallach, Perelman, \cite{CaCZ,CaC}.    
In dimensions two and three, they are $N^2 = \SS^2$ or $\RP^2$ and $N^3 = \SS^3$ or $\SS^2 \times \RR$ (plus quotients).
 The classification in dimension three
relies on an equation for the $2$-tensor $\frac{\Ric}{S}$ that fails in higher dimensions where there is no similar classification.  In fact, there are
large families of shrinkers in higher dimensions.
Combining  dimension reduction with the classification in low dimensions suggests  that
the most prevalent singularities are:\\
\centerline{$\SS^2\times \RR^{n-2}$ followed by $\SS^3\times \RR^{n-3}$ (and quotients).}

As one approaches a singularity in the flow and magnifies, one would like to know which singularity it is.  Since most singularities are non-compact yet the evolving manifolds are closed, one only sees a compact piece of the singularity at each time as one approaches it.  The next theorem recognizes   singularities from just a compact piece (see Theorem \ref{t:rigidP} for the precise statement).

\begin{Thm}	\label{t:main}
Cylindrical shrinkers  $\SS^{\ell} \times \RR^{n-\ell}$ are strongly rigid for any $\ell$.
\end{Thm}

{\it Strong rigidity means} that if another shrinker is close enough on a large compact set, then it must agree.  
The theorem holds for products of $\RR^{n-\ell}$ with quotients of $\SS^{\ell}$ and a large class of other 
positive Einstein manifolds; see Section \ref{s:jacobi} for details.    An important difficulty is that  there are nontrivial infinitesimal  variations, i.e.,   in the kernel of the linearized operator (and not generated by diffeomorphisms).  One   consequence  of Theorem \ref{t:main} is that the infinitesimal variations are not integrable.

  Uniqueness is important in many areas of geometry,  PDE, and general relativity.  Unlike here, one typically makes global assumptions - e.g., symmetries, curvature conditions, or  asymptotics at infinity.\footnote{In GR uniqueness and stability of solutions to Einstein's equations are fundamental problems and the gauge group causes well-known difficulties.   Unlike here,
    in GR space-time is assumed to have strong asymptotic decay.  The two central difficulties in stability of black holes are determining the final state (uniqueness) and proving convergence. Convergence can only be established relative to a
coordinate system which cannot be a priori fixed but has instead to be
constructed dynamically. This is often referred to as ``the famous problem of gauge determination''.  For the uniqueness of the final state, the gauge group can be circumvented when the space can be characterized by the vanishing of a special tensor like the Mars-Simon tensor.}
  In most problems in geometric PDEs, it would be impossible
    to control an entire solution from just knowing roughly how it looks on a compact set.    If one knew exactly how it looked like on a compact set, it would be much less surprising and essentially follow from unique continuation.  The surprising thing here is that we only assume closeness and only on a compact set and this is enough to characterize the shrinker.  This is an illustration of a 
{\it{shrinker principle}}  which roughly says that  ``uniqueness radiates outwards''.  Nothing like this is true for   Einstein manifolds (or steady solitons), where gravitational instantons contain arbitrarily large arbitrarily Euclidean regions.
 The  shrinker principle was originally discovered in mean curvature flow 
  \cite{CIM,CM2}. It has been conjectured since that something similar  holds for Ricci flow, but the gauge group has been one of the major obstacles.  In mean curvature flow,   the  gauge is circumvented using extrinsic coordinates.

Tangent flows are  limits of a subsequence of rescalings at the singularity.  A priori different subsequences 
might give different limits.
Using Theorem \ref{t:main}, we get the following uniqueness:

\begin{Thm}	\label{t:uni}
For a Ricci flow, if  one tangent flow at a point in space-time is a   cylinder, then all other tangent flows at that point are also cylinders.
\end{Thm}

Unlike most results in Ricci flow, these results hold  for every $n$ and $\ell$. Increasing the dimension of the Euclidean factor is a subtle problem (e.g.~surgery, cylindrical estimates, and  $k$-convexity estimates only allow small Euclidean factors).
For general $n$ and $\ell$, cylinders do not have a coordinate-free characterization.
  This is a major part of the difficulty.

At singularities where the tangent flows are compact shrinkers, the singularities are isolated in space-time.  For compact shrinkers, rigidity was proven in dimension three by Hamilton  and by Huisken  for higher dimensional spheres.
  Even in the compact case,   rigidity fails in general; see  \cite{Bs,Bo,BGK,Ca} and \cite{Kr,SZ}. 

Rigidity for necks  $\SS^{n-1} \times \RR$ was proven independently by Li-Wang \cite{LW2}.  They
are able to
  circumvent the gauge problem   using
 that  their Euclidean factor is a line.  They do that, in part,   using 
 tensors with  special properties on the product of a sphere with a  line to prove asymptotic structure and approximate symmetry.  Once they have this, 
 they are able to use again that their Euclidean factor is a line to 
 adapt   Brendle's symmetry improvement \cite{Br1,Br2,Br3} to get $O(n)$ symmetry and, finally,
  Kotschwar's classification of rotationally symmetric shrinkers \cite{K}.

\subsection{What is needed for rigidity?}
We need to show that if two shrinkers are close on a large but compact set, then there is a global diffeomorphism between them
that preserves the metric and weight.  
The two shrinkers are   not assumed to be globally diffeomorphic, so we   must build the global map starting from the map between compact pieces.
   This is done in stages, first building the initial map out to a larger scale  so that it still roughly preserves the metric and weight
(Theorem \ref{t:extend}).  This comes at the cost of a loss in the estimates: the metrics and weights will not be as close on the larger set as they were initially.  This loss means that this process cannot be repeated indefinitely.  To overcome this, we  make a change of gauge to recover the loss and   get even better estimates on the larger scale  (Theorem \ref{t:improve}).  Together, Theorems  \ref{t:extend} and \ref{t:improve} can be iterated to get better and better estimates on larger and larger scales, eventually giving the strong rigidity.
Estimates proving  polynomial losses will be played off against exponential gains.

There are four key ingredients in the proof of strong rigidity.  All of them are new.  The first two hold on any non-compact shrinker.
\begin{enumerate}
\item   Gauge fixing.
\item New polynomial growth estimates for PDEs.
\item Propagation of almost splitting.
\item Quadratic rigidity in the right gauge. 
\end{enumerate}
We will use the new polynomial growth estimates as ingredients in both (1) and (3).

\subsection{Gauge fixing}
Fix $(M,g,f)$. 
   We are given a diffeomorphism from a large compact set in $M$ to a second weighted space. 
\begin{itemize}
\item  The pull-back metric and weight are $g+h$ and $\e^{-f-k}$.
\item $h$ and $k$ are small on the compact set.
\end{itemize}
 Composing with a diffeomorphism on $M$ gives a different $h$ and $k$.   We want to mod out by this group action, by choosing a diffeomorphism so that the new $h$ is orthogonal to the group action.
  This is gauge fixing.

 One of the most interesting results of transformation groups is the existence of slices.
  A slice for the action of a group on a manifold  is  a submanifold  which is transverse to the orbits.{\footnote{If the group is compact and Lie and the space is completely regular, Mostow proved, as a generalization of works of Gleason, Koszul, Montgomery, Yang and others, that there is a slice through every  point.    If the group is not compact but Lie and if the space is a Cartan space, then Palais proves the same result.}} 
    Ebin and Palais proved the existence of a slice for the   diffeomorphism group of a {\emph{compact}} manifold acting on the space of all Riemannian metrics.  The slice can be thought of as the gauge fixing on the compact manifold.

   In our setting, $M$ is noncompact and gauge fixing is choosing a diffeomorphism $\Phi$ on $M$ so $\tilde{h} = \Phi^*\,(g+h)-g$  is orthogonal to the group action.  
   Orthogonality corresponds to  
   \begin{align}
\dv_f\, \tilde{h} =0 \, ,	\label{e:nonlinear}
\end{align}
where
  $\dv_f\,(h) =\dv\,(h)-h(\nabla f,\cdot)$.  The equation  \eqr{e:nonlinear}  is a nonlinear PDE for $\Phi$.
  Terms involving $\dv_f $ come up again and again, so many quantities simplify in this gauge and having them drop out as they do when $\dv_f\, \tilde{h}=0$ makes things possible to analyze.

 We construct the diffeomorphism $\Phi$ that solves \eqr{e:nonlinear} using an iteration scheme  for the linearized operator $\cP$ on vector fields $Y$.    
 Using optimal polynomial  bounds on $\cP$, we show sharp polynomial bounds for the displacement function of $\Phi$
$$
x\to \text{dist}_g(x,\Phi (x))\,  .
$$
For applications, it is crucial that we only assume closeness on a compact set and, in particular, a priori 
the two shrinkers do not need to be diffeomorphic.  This means that we cannot fix the gauge at the outset.  Instead we need to apply our gauge fixing procedure iteratively to fix the gauge on larger and larger scales as we move outward and show closeness on larger and larger scales.  To pull this off requires very strong estimates for the displacement.  Our optimal estimates show that the displacement of the gauge fixing diffeomorphism grows at a sharp polynomial rate.  
These results are very general and apply to all shrinkers.

On a shrinker $(M,g,f)$, the  natural gaussian  $L^2 = L^2 (\e^{-f})$ norm  is given by  $\| u \|_{L^2(\e^{-f})}^2 = \int_M u^2 \, \e^{-f}$. 
Diffeomorphisms near the identity are infinitesimally generated by integrating a vector field $X$.  The infinitesimal change of the metric is given by the Lie derivative of the metric with respect to $X$.  This is equal to $- \frac{1}{2} \, \dv_f^* X$, where $\dv_f^*$ is the operator adjoint of $\dv_f$ with respect the to gaussian inner product. Thus,
if we define the  operator $\cP$ by
\begin{align}	
\cP\,X=\dv_f\,\circ\dv_f^*\,X\,  ,
\end{align}
then the linearization of \eqr{e:nonlinear} is to find a vector field $Y$ with 
\begin{align}	\label{e:cPYdv}
	\cP\,Y=\frac{1}{2}\,\dv_f\,h\,  .
\end{align}
A detailed analysis of $\cP$ and its properties plays an important role in the gauge fixing.

Solutions of \eqr{e:cPYdv} are unique once we require that $Y$ is orthogonal to the kernel of $\cP$.  The kernel is the Killing fields.  We will solve \eqr{e:cPYdv} on any shrinker (Theorem \ref{c:elliptic}) 
and show via $L^2$ methods that $\|Y\|_{W^{1,2}(\e^{-f})}\leq \|\dv_f\,h\|_{L^2(\e^{-f})}$.  Given the non-compactness, $L^2$ estimates are not sufficient to implement the iteration scheme and we need stronger polynomial 
estimates.{\footnote{ The $L^2$ theory for $\cP$ shares formal similarities with H\"ormander's  influential $L^2$ $\bar\partial$ method in several complex variables.  In the $L^2$ $\bar\partial$ method, one solves the Poisson equation $\bar\partial u=F$, with estimates, where $\bar\partial F=0$.  To do so, one introduces the adjoint of $\bar\partial$ with respect to a weight.  H\"ormander's idea for the weight came from Carleman's method for proving unique continuation of a PDE.  Here we solve $\cP\,Y= F$, where $F=\frac{1}{2}\,\dv_f\,h$ is orthogonal to the kernel of $\dv_f^*$.  H\"ormander's method gives weighted $L^2$ bounds for $\bar\partial$ similar to our weighted bounds for $\cP$.   To introduce a second weight to capture the growth \`a la Carleman and H\"ormander is less natural here.   Instead, we go a different route to prove stronger bounds.}}
  The problems are magnified by that initial closeness is only on a given compact set.  As one builds out  to get closeness on larger   sets, one needs at each step to adjust the entire diffeomorphism so the normalization is zero on larger and larger sets.

The operator $\cP$ is related to the  generalized Ornstein-Uhlenbeck operator  
$\cL = \Delta  - \nabla_{\nabla f} $.
  Given a vector field $X$ on a 
shrinker, the operators $\cL$ and $\cP$ commute and are related by    
\begin{align}
-2\,\cP\,X=\nabla \dv_f\,X+\cL\,X+\frac{1}{2}\,X
\end{align}
 (Proposition \ref{p:commPL} and Lemma \ref{l:dvastY}).
The unweighted version of $\cP$  was used implicitly by Bochner to show that closed manifolds with negative Ricci curvature have no Killing fields and
later   by Bochner and Yano  to show that the isometry group   is finite.  
 The unweighted operator also arises in general relativity and fluid dynamics.    The weighted operator $\cP$ appears to have been largely overlooked.
The relationship between $\cP$ and the unweighted version  mirrors the relationship between the Ornstein-Uhlenbeck operator and the Laplacian.

\subsection{Optimal growth bounds}
Laplace discovered that on the line eigenfunctions of $\cL\,u=u''-\frac{x}{2}\,u'$ in the gaussian $L^2$ space are polynomials whose degree is exactly twice the eigenvalue.   These polynomials were later rediscovered twice.  First by Chebyshev and a few years later by Hermite.   They are known as the Hermite polynomials and the eigenvalue equation as the Hermite equation.  They play an important role in diverse fields.  

On the line, the   space  $L^2(\e^{ - \frac{x^2}{4}})$ allows extremely rapid growth, so it is  surprising that the  $L^2(\e^{ - \frac{x^2}{4}})$ eigenfunctions grow just polynomially.  The standard proofs of this use the special structure of Euclidean space that
do not extend to manifolds without making very strong assumptions.  
However, 
we will prove that this polynomial growth holds for a wide class of manifolds, metrics and weights.  
   In many settings one has an $n$-dimensional Riemannian manifold $(M,g)$  with two nonnegative functions $f$ and $S$ that satisfy
\begin{align}
\Delta\,f+S&=\frac{n}{2}\, ,\label{ee:aronson1}\\
|\nabla f|^2+S&=f\, ,\label{ee:aronson2}
\end{align}
and where $f$ is proper and $C^n$.   Two important examples  are  shrinkers in both Ricci flow and mean curvature flow (MCF).
In Ricci flow, $S$ is scalar curvature, while
 $f=\frac{|x|^2}{4}$ and
$S=|{{\bf{H}}}|^2$ in MCF, where ${\bf{H}}$ is the mean curvature   (see, e.g., \cite{Hu}, \cite{CM1}, \cite{CM9}).  

\begin{Thm}	\label{t:PGmeta}
If \eqr{ee:aronson1} and \eqr{ee:aronson2} hold and a tensor $u \in L^2 (\e^{-f})$ satisfies $\cL \, u = - \lambda \, u$, then $u$ grows polynomially of degree at most $2\, \lambda$.
\end{Thm}

This and a corresponding Poisson version give powerful new tools with many applications, 
 including in the proofs of propagation of almost splitting and gauge fixing.

Combining Theorem \ref{t:PGmeta} with the following gives optimal growth bound for eigenvector fields of $\cP$ on any Ricci shrinker
(note that $\cP$ and $\cL$ have opposite signs):

\begin{Thm}	\label{t:cPandcL}
On any shrinker, any eigenvector field $Y$ for $\cP$ with eigenvalue $-\lambda$ can be written as the $L^2(\e^{-f})$-orthogonal sum of two eigenvector fields for $\cL$.  One is $\dv_f$-free with eigenvalue $2\lambda + \frac{1}{2} $ and the other is $ \frac{2}{2\, \lambda + 1} \,\nabla \dv_f\,Y$ and has eigenvalue $\lambda$.  
\end{Thm}


 These growth estimates   hold
   in remarkable generality and without any assumptions on asymptotic decay.   
   This is surprising and in contrast to most other situations, like unique continuation, that require very strong geometric assumptions on the space.   A typical starting point for growth estimates is a Pohozaev identity or commutator estimate that comes from a dilation, or approximate dilation, structure.   We have none of these here in this general setting.   In contrast, we rely on a miraculous cancellation for just the right quantity.

\subsection{Propagation of almost splitting}
  One of the important new 
ingredients  is that a
   Ricci shrinker close to a product $N \times \RR^{n-\ell}$ on a large scale remains close on a fixed larger scale.  
The idea is that the initial closeness will imply that $\cL$ has  eigenvalues that are exponentially close to $\frac{1}{2}$. 
The drift Bochner formula on a shrinker implies that every eigenvalue is at least $\frac{1}{2}$ with equality only when it splits.  We show that being close to $\frac{1}{2}$ gives that the hessian is almost zero in $L^2(\e^{-f})$, which is very strong when the weight $\e^{-f}$ is 
close to one but says almost nothing further out.  The crucial point is  that our polynomial growth estimates imply that the hessian can grow only polynomially, so the very small initial bound
gives bounds much further out.  Thus, the gradients of these eigenfunctions give the desired almost parallel vector fields and almost splitting.
This is very much a Ricci flow fact that does not have an   analog in MCF where there is no corresponding description of the bottom of the spectrum.

Once we have this metric almost splitting, we show that it also almost splits as shrinker on the larger scale.  Namely, the cross-sections are close to $N$
and the potential  $f$ is well-approximated by $ \frac{|x|^2}{4}$.     However, 
there is a loss in the estimates - it may look less 
cylindrical on the larger scale - that makes this impossible to iterate on its own.
   
\subsection{Quadratic rigidity}

The propagation of almost splitting and gauge fixing give that the shrinker is  close to a cylinder on a large set via a diffeomorphism  that fixes the gauge.  
The last of the four key ingredients is an estimate for the difference in metrics that is small enough to be iterated.
For this, it is essential   that the gauge be right, or else it just isn't true.  The closeness cannot be seen via
 linear analysis.  However, we show that there is a second order rigidity that gives the estimate; we call this quadratic rigidity.

To explain the estimate,  let $(M,g,f)$ be the cylinder and  $(M,g+h, f+k)$ the shrinker that is close on a large compact set.  We need
 bounds on $h$ and $k$ that can be iterated.
The linearization of the shrinker equation is
\begin{align}	\label{e:linea}
	\frac{1}{2} \, L \, h   + \Hess_{ \frac{1}{2} \, \Tr (h) - k} + \dv_f^* \, \dv_f \, h \, .
\end{align}
This linearization was derived by Cao-Hamilton-Ilmanen in their calculation of 
 the second variation operator for Perelman's  entropy.
The operator $L$ acts on $2$-tensors by 
	$L \, h = \cL \, h + 2 \, \R (h)$ and $\R(h)$ is the natural action of the Riemann tensor, cf. \cite{CM1} for mean curvature flow.

Since $(M,g+h, f+k)$ is also a shrinker,  \eqr{e:linea} must be at least quadratic in $(h,k)$.  
The last two terms in \eqr{e:linea} are gauge terms -  i.e., in the image of $\dv_f^*$ and there is no reason for these - or $h$ -  to be small if not in right gauge.
 In the {\it right gauge}, $h$ satisfies the Jacobi equation $L\, h = 0$ up to higher order terms.  
 This does not  force $h$ to be small since
 cylinders have non-trivial Jacobi fields that could potentially integrate to give nearby shrinkers.  However,
it will give that $h$ is a Jacobi field to first order.  The Jacobi field is  described by a  quadratic Hermite polynomial $u$, so
 $|h|$ is  $|u|$ up to higher order. 
The second variation of the {\bf{shrinker equation}} in the direction of the Jacobi field is given by the tensor
    \begin{align}	\label{e:quad}
    	-2\, |\nabla u|^2 \, \Ric - 2\, S \, u \, \Hess_u - S \, \nabla u \otimes \nabla u
	\, ,
\end{align}
where  $S$ is scalar curvature.  
The second order Taylor expansion
 will imply that 
\eqr{e:quad} 
vanishes to at least third order in $h$, so the quadratic expression \eqr{e:quad} is at least
 cubic in $u$.  When $u$ is small, this implies that $u$ and $h$ vanish; we will have extra error terms  so 
 will get that $h$ is exponentially small, giving the improvement that we needed to iterate.

 \subsection{Connections with other work}  Rigidity and uniqueness of blowups are fundamental questions in regularity theory  
with many  applications.    In mean curvature flow, they play a major role in understanding the singular set, proving optimal regularity, 
 understanding solitons, 
 classifying ancient solutions, and understanding low entropy flows.
 In MCF, cylinders are rigid by \cite{CIM,CM10} and cylindrical blowups unique by \cite{CM2,CM9}.  These results have important applications, \cite{CM4,CM5,CM6}.

 One of the central problems in many areas of dynamical systems, ergodic theory, PDEs and geometry is to understand the dynamics of a flow near singularities.   
Such as classifying nearby singularities, determining whether flows  have unique limits  or oscillate,
and identifying dynamically stable solutions that attract nearby flows.  
These questions are more complicated in the presence of a gauge group.
 The techniques   introduced here open a door for understanding dynamical properties for Ricci flows nearby. By further developing these techniques, we show  uniqueness of blowups for Ricci flow.    See also the survey \cite{CM15}.

\vskip2mm
We would like to thank the referee for carefully reading through the manuscript and helpful comments.
We would also like to thank Yi Lai, Yu Li and Bing Wang for their interest and comments.

 \section{Elliptic systems on tensor bundles and their commutators}

 In this section, the triple $(M,g,f)$ is a manifold with Riemannian metric $g$ and a function $f$.  Given a constant $\kappa$, define the symmetric $2$-tensor  
\begin{align} 
	\phi &= \kappa \, g -  \Ric - \Hess_f   \, .  \label{e:phi}
\end{align}
  The triple $(M,g,f)$ is a gradient Ricci soliton when 
  $\phi =0$; it is shrinking for $\kappa = \frac{1}{2}$, steady for $\kappa = 0$, and expanding for $\kappa = - \frac{1}{2}$;
  see \cite{H,Cn,Ca,ChL,ChLN,CN,CRF,KL1,P,T}.
    Later, we will take $\kappa = \frac{1}{2}$
  and focus on shrinking solitons. For now, we leave $\kappa$ as a variable as the results here apply to all three cases.

We recall some basic properties of $\cL$.  First,  $\cL$ is self-adjoint for the weighted $L^2 = L^2 (\e^{-f})$ norm
$
	\int_M \left( \cdot \right)  \, \e^{-f} $ and $\cL = -\nabla^* \, \nabla$ where $\nabla^*$ is the adjoint of $\nabla$ with respect to the weighted $L^2$ norm.
	When   $V$ is a vector field and $u$ is a function with compact support, then  integration by parts gives
\begin{align}
	\int \langle \nabla u , V \rangle \, \e^{-f} = - \int u \, \dv \, \left( V \, \e^{-f} \right) 
	=- \int u \, \left( \dv \, V - \langle V , \nabla f \rangle \right) \, \e^{-f} \, .
\end{align}
Motivated by this, define  $\dv_f$ on vector fields by  $\dv_f \, V =  -\nabla^* \, V = \dv \, V - \langle V , \nabla f \rangle$.

Let $\R_{ijk\ell}$ be the full Riemann curvature tensor in an orthornomal frame, so 
     \begin{align}
     	\R_{ijk\ell} = \langle \R(e_i , e_j)\, e_k , e_{\ell} \rangle = \langle \nabla_{e_j} \nabla_{e_i }e_k - \nabla_{e_i} \nabla_{e_j} e_k + \nabla_{[e_i , e_j]} e_k , e_{\ell} \rangle  \, .
     \end{align}
             The sign convention  is that 
    $\Ric_{ij} = \R_{kikj}$, where, by convention, we sum over the repeated index $k$.     
	Define the operator $L$ on a $2$-tensor $B$ in an orthonormal frame by 
\begin{align}	\label{e:LBij}
	L\,B_{ij}=\cL\,B_{ij}+ 2\, \R_{\ell i k j}\,B_{\ell k}   \, .
\end{align}
Since $\cL\, g = 0$ (as the metric is parallel), we see that $L \, g_{ij} = 2 \, \Ric_{ij}$.  

\vskip1mm
The next result gives Simons-type differential equations for the Ricci and scalar curvature, $\Ric$ and  $S$,  in terms of the drift operators $L$ and $\cL$ and the tensor $\phi$.

 \begin{Thm}	\label{t:Lricci}
 We have
 \begin{align}
 	(L \, \Ric)_{ij} &=  2\, \kappa \,  \Ric_{ij} + 2\, \R_{kjin}\, \phi_{nk} -
	  \phi_{ij,kk}  - \phi_{kk,ji} + \phi_{jk,ki} + \phi_{ik,kj} \, ,  \label{e:withinA} \\
	  \cL \, S &=  2\, \kappa \, S  -2\, | \Ric |^2  -2\, \Ric_{kn}\, \phi_{nk} -
	 2\, \Delta\, \phi_{kk} + 2\, \phi_{ik,ki}  \, .    \label{e:cLS}
 \end{align}
 \end{Thm}

When $\phi = 0$, Theorem \ref{t:Lricci} recovers well-known identities for gradient Ricci solitons (cf. \cite{CaZ},  \cite{H}, \cite{T}).  However, the theorem applies to
{\emph{any}} metric $g$ and weight $\e^{-f}$.   Allowing $\phi \ne 0$ is important in analyzing Ricci flow near a singularity.  Furthermore, even for solitons, it is useful to allow $\phi \ne 0$ when  ``cutting off'' a non-compact solution.

 \subsection{Bochner formulas and commutators}

 To keep notation short,  let $f_{i_1 \dots i_k}$ denote the $(k-1)$-st covariant derivative of $\nabla f$ evaluated on $(e_{i_1}, \dots , e_{i_k})$, where $e_{i_k}$ goes into the slot for the last derivative.  In the calculations below, we work at a point $p$ in an orthonormal frame $e_i$ with $\nabla_{e_i} e_j = 0$ at $p$.  
We will use subscripts on a bracket to denote the ordinary directional derivative.  For example, $f_{ijk} = (\nabla_{e_k} \Hess_f )(e_i , e_j)$ and
 \begin{align}
	(f_{ij})_k &= \nabla_{e_k} \left( \Hess_f (e_i , e_j)  \right) = f_{ijk} + \Hess_f ( \nabla_{e_k} e_i , e_j) + \Hess_f (e_i , \nabla_{e_k} e_j) \, ,	\label{e:leib3}
 \end{align}
 where the last equality is the Leibniz rule.  Thus, at $p$ we have $f_{ijk} = (f_{ij})_k$.  We use corresponding notation for tensors, with a comma to separate the derivatives from the original indices.  Thus, if $Y$ is a vector field, then 
 $Y_i = \langle Y , e_i \rangle$ and 
 	$(Y_i)_j = Y_{i,j} + \langle Y , \nabla_{e_j} e_i \rangle$.
 The next lemma computes the commutator of $\nabla$ and $\cL$, i.e., the drift Bochner formula, \cite{BE,L}:
 
 \begin{Lem}	\label{l:boch}
 If $Y$ is a vector field, then $Y_{i,jk} - Y_{i,kj} = \R_{kjin} \, Y_n$.  In particular, $u_{ijk}   - u_{ikj}  =  \R_{kjin}u_n$ and
 we get the drift Bochner formulas
 \begin{align}	\label{e:commuta}
 	\cL \, \nabla u &= \nabla \, \cL u + \left( \Ric + \Hess_f \right)  (\nabla u , \cdot ) = \nabla \, \cL u +\kappa \, \nabla u - \phi (\nabla u , \cdot ) \, , \\
	\cL \, \dv_f \, Y &=     \dv_f \, \cL \, Y -\kappa \, \dv_f \, Y 
	+    \dv_f \, ( \phi (Y , \cdot ))     \, .
	\label{e:commutaB}
  \end{align}
 \end{Lem}
 
 \begin{proof}
The first claim is essentially the definition of $\R$.  The second claim follows immediately with $Y = \nabla u$.  Next, using the second claim, we have
\begin{align}
	(\cL \, \nabla u )_i &= u_{ijj} - u_{ij} \, f_j = u_{jji} + \R_{jijk} \, u_k- u_{ij} \, f_j = u_{jji} + \Ric_{ik} \, u_k- (u_{k} \, f_k)_i + u_k \, f_{ik} \notag \\
	&= ( \nabla \, \cL \, u )_i + \left( \Ric_{ik} + f_{ik} \right) \, u_k =  ( \nabla \, \cL \, u )_i  + \kappa \, u_i - \phi_{ik} \, u_k \, .
\end{align}
Finally, \eqr{e:commutaB}
 follows by taking the adjoint of \eqr{e:commuta} since $(\cL \, \nabla)^* = - \dv_f \, \cL$ and we have
 $(\nabla \, (\cL+ \kappa))^* = -( \cL + \kappa) \, \dv_f$.
  \end{proof}

 We compute the gradient and Hessian of  $S$ (cf. \cite{ChLN}, \cite{PW} for solitons):  

\begin{Lem}	\label{l:gradS}
The gradient and Hessian of $S$ are given by
\begin{align}
	\frac{1}{2} \, S_i & =\nabla^j \Ric_{ij} =   \Ric_{ik} f_k - \phi_{kk,i} + \phi_{ik,k} \, , 	\label{e:gradSa} \\
	\frac{1}{2} \, S_{ij} &=- \phi_{ik,kj} - f_{kikj}  =   \Ric_{ik,j} f_k + \Ric_{ik} f_{kj} - \phi_{kk,ij} + \phi_{ik,kj} \, .  	\label{e:gradSb}
\end{align}
\end{Lem}

\begin{proof}
The first equality in \eqr{e:gradSa} is known as the Schur lemma and  is a standard consequence of the contracted second Bianchi identity
  \begin{align}	\label{e:2bia}
 	 \Ric_{kn, i} + \R_{kijn, j} - \Ric_{in, k} = 0 \, .
 \end{align}
 Use the first equality, 
take the divergence of \eqr{e:phi}  and use Lemma \ref{l:boch} to get
\begin{align}	\label{e:one}
	-\Ric_{ij,j} = \phi_{ij,j} + f_{jij}  =  \phi_{ij,j} + f_{jji} +\Ric_{ij}\, f_j =
	 \phi_{ij,j}   + \Ric_{ij}\, f_j + \left(\kappa \, n -\phi_{jj} - S  \right)_i \, .
\end{align}
This gives the second equality in \eqr{e:gradSa}.
The first equality in \eqr{e:gradSb}  follows from taking the derivative of \eqr{e:one}. Taking the derivative of  \eqr{e:gradSa}
gives the last claim.  
\end{proof}

\begin{Cor}	\label{c:gradS}
We have $\frac{1}{2} \, \left( S + |\nabla f|^2 - 2\, \kappa \, f \right)_i = - \phi_{ik} \, f_k - \phi_{kk,i} + \phi_{ik,k} $ and
\begin{align}
	\frac{1}{2} \, \left( \cL  \, f + 2\, \kappa \, f \right)_i &=  \phi_{ik} \, f_k + \frac{1}{2} \,  \phi_{kk,i} - \phi_{ik,k}
	\, .
\end{align}
\end{Cor}

\begin{proof}
Substituting the definition of $\phi$ in the first claim in Lemma \ref{l:gradS} gives
\begin{align}
	\frac{1}{2} \, S_i &=   \Ric_{ik}\, f_k - \phi_{kk,i} + \phi_{ik,k} = \kappa \, f_i  - f_{ik}\,   f_k - \phi_{ik} \, f_k - \phi_{kk,i} + \phi_{ik,k} 
	\, .
\end{align}
This gives the first claim.  The second claim follows from the first and 
$\Delta \, f = n \, \kappa - S - \Tr \, \phi $.
\end{proof}

  We next compute the Laplacian of $\Hess_f$.
 
 \begin{Lem}	\label{l:deltahess}
 We have the following formula
  \begin{align}	 
 	f_{ijkk} &=  -\Ric_{ji,k}\, f_k   + 2\,  \R_{kjin}\, f_{nk} + \phi_{kk,ji} - \phi_{jk,ki} - \phi_{ik,kj} \, .
 \end{align}
 \end{Lem}
 
 \begin{proof}
 Working at $p$, Lemma \ref{l:boch} gives that
 \begin{align}	\label{e:s1}
 	f_{ijkk} &= (f_{ijk})_k = \left(f_{ikj} + \R_{kjin}\, f_n  \right)_k = f_{ikjk} + \R_{kjin,k}\, f_n +  \R_{kjin}\, f_{nk}  \, .
 \end{align}
The Ricci identity gives
 \begin{align}
 	f_{kijk} = f_{kikj} + \R_{kjkn}\,f_{ni} + \R_{kjin}\, f_{kn} =  f_{kikj} + \Ric_{jn}\,f_{ni} + \R_{kjin}\, f_{kn}  \, .
 \end{align}
 Using this in \eqr{e:s1} gives
  \begin{align}	 \label{e:cfunc}
 	f_{ijkk} &=  f_{kikj} + \Ric_{jn}\,f_{ni} + \R_{kjin}\, f_{kn}  + \R_{kjin,k}\, f_n +  \R_{kjin}\, f_{nk}  \, .
 \end{align}
 Using the last two claims in Lemma \ref{l:gradS} (and  symmetry of $S_{ij}$) gives
\begin{align}
	\Ric_{jk,i}\, f_k + \Ric_{jk}\, f_{ki} - \phi_{kk,ji} + \phi_{jk,ki} = -\phi_{ik,kj} - f_{kikj} \, .
\end{align}
 Applying this to the first two terms on the right in \eqr{e:cfunc}, we get
  \begin{align}	 
 	f_{ijkk} &=  -\Ric_{jk,i} f_k + \phi_{kk,ji} - \phi_{jk,ki} - \phi_{ik,kj} +    \R_{kjin,k}\, f_n + 2\,  \R_{kjin}\, f_{nk} \notag \\
	&=  \left( -\Ric_{jk,i} +  \R_{mjik,m} \right) \, f_k  + \phi_{kk,ji} - \phi_{jk,ki} - \phi_{ik,kj}   + 2\,  \R_{kjin}\, f_{nk}  \, .
 \end{align}
 The lemma follows from this and using the  trace \eqr{e:2bia} of the second Bianchi identity  to rewrite the first term on the right 
 as $
 -\Ric_{jk,i} +  \R_{mjik,m} = - \Ric_{ij,k}$.
 \end{proof}

 \begin{proof}[Proof of Theorem \ref{t:Lricci}]
  Use the definition of $\phi$ to write
 $	\cL\, \Ric =   -\Delta\, \phi - \Delta \, \Hess_f - \nabla_{\nabla f} \Ric$.   Using Lemma \ref{l:deltahess}, we get
 \begin{align}
 	\cL \, \Ric =    -\phi_{ij,kk} - f_{ijkk}- \Ric_{ij,k}\, f_k  = -2\, \R_{kjin}\, f_{nk} -
	  \phi_{ij,kk}  - \phi_{kk,ji} + \phi_{jk,ki} + \phi_{ik,kj}\, .
 \end{align}
 The definition of $\phi$ gives $f_{nk} =- \phi_{nk}+ \kappa \, g_{nk} - \Ric_{nk}$.  Thus, we have
 \begin{align}
 	 \R_{kjin}\, f_{nk} =  -\R_{kjin}\, \phi_{nk} -  \R_{kjin}\, \Ric_{nk} -\kappa \, \Ric_{ij} \, .
 \end{align}
 Substituting this gives
  \begin{align}	\label{e:within}
 	(\cL \, \Ric)_{ij} =  2\, \kappa \, \Ric_{ij}  +2\, \R_{kjin} \Ric_{nk} +2\, \R_{kjin}\, \phi_{nk} -
	  \phi_{ij,kk}  - \phi_{kk,ji} + \phi_{jk,ki} + \phi_{ik,kj} \, ,
 \end{align}
 and \eqr{e:withinA} follows from substituting the definition of $L$.   The second claim follows from taking
 the trace over $i=j$ in \eqr{e:withinA}.
 \end{proof}

  \subsection{The $f$-divergence and its adjoint}
 
   As in   \cite{CaZ}, define the $f$-divergence of a symmetric $2$-tensor $h$  to be the vector field  
 	$(\dv_f \, h) (e_i) = \e^f \, \left( \e^{-f} \, h_{ij} \right)_j  = h_{ij,j} - f_j h_{ij}$.
The second equality in  Lemma \ref{l:gradS} gives
 $(\dv_f \, \Ric )(e_i) = \Ric_{ik,k} -   \Ric_{ik} f_k  = - \phi_{kk,i} + \phi_{ik,k} $, 
so that $\dv_f \, \Ric = 0$  on a soliton.
 The adjoint $\dv_f^{*}$ of $\dv_f$ is given on a vector field $Y$ by
  \begin{align}	\label{e:adjoint}
 	(\dv_f^{*}\, Y) (e_i , e_j) =  - \frac{1}{2} \, \left( \nabla_i Y_j + \nabla_j Y_i 
	\right) \, .
 \end{align}
 Namely, if $\int \left( |Y|^2 + |\nabla Y|^2 + |h|^2 + |\nabla h|^2 \right) \, \e^{-f} < \infty$, then
 $\int \langle h , \dv_f^{*}\, Y   \rangle \, \e^{-f} = \int \langle Y , \dv_f \, h \rangle \, \e^{-f} $.
Note that  $\dv_f^{*}$ applied to a gradient gives $\dv_f^{*}\, \nabla v = - \Hess_v$.
 Thus, if $\dv_f \, h =0$, then $h$ is orthogonal to any Hessian and,  more generally,     to   variations coming from diffeomorphisms
 since $-2\, \dv_f^*\,Y$ is the Lie derivative of the metric in the direction of $Y$.  
  
The next theorem computes the commutator of $\cL$ with $\dv_f$ and $\dv_f^*$.  As a consequence, 
 $L$ preserves the image of $\dv_f^{*}$ when $(M,g,f)$ is a gradient Ricci soliton.

    \begin{Thm}	\label{c:BHa}
  If  $V$ is a vector field and $h$ is a symmetric two-tensor,
  then
\begin{align}
	L \, \dv_f^{*} (V) &= \dv_f^{*} \left( \cL \, V +\kappa \, V \right) 
	+ \frac{1}{2} \left(  \phi_{jn} V_{i,n} + \phi_{in}  V_{j,n} 	\right)      
	            - \frac{V_n}{2}(  2\, \phi_{ ji,n } - \phi_{jn,i} - \phi_{in,j}  ) \, , \notag \\
   	 \dv_f \, L \, h &=  \left( \cL + \kappa \right) \,\dv_f  \, h - h_{nj} \, (\dv_f \, \phi)_j - h_{in,j} \, \phi_{ij}
	    - \frac{h_{ij}}{2}  (  2\, \phi_{ ji,n } - \phi_{jn,i} - \phi_{in,j}  )
	\, .  \label{e:114}
\end{align}
\end{Thm}

 \begin{Cor}	\label{t:bochHess}
   If $(M,g,f)$ is a gradient Ricci soliton, then
   	$L \, \Hess_u =  \Hess_{(2\, \kappa \, u + \cL u)}$.
   \end{Cor}
   
   \begin{proof}
   Theorem \ref{c:BHa} with $V=\nabla u$ and $\phi = 0$ gives that 
   	$- L \, \Hess_u =  \dv_f^{*} \left( \cL \, \nabla u + \kappa \, \nabla u \right)$.  
   This and the drift Bochner formula $\cL \, \nabla u = \nabla \cL \, u +\kappa \, \nabla u$, see  \eqr{e:commuta}, gives the claim.
   \end{proof}

    The next lemma will be used in the proof of Theorem 
    \ref{c:BHa}.

  \begin{Lem}	\label{l:bochHessA}
If $V$ is vector field, then
    \begin{align}	 
 	V_{i,jkk}  &=  V_{i,kkj}  + \Ric_{jn}V_{i,n}       +\R_{nijk   } f_k \,  V_n  + 2\,  \R_{kjin} V_{n,k} 
	+ V_n (  \phi_{ ji,n } - \phi_{jn,i}   )  \, .
 \end{align}
 \end{Lem}
 
 \begin{proof}
 The trace of the second Bianchi identity gives that
	 $\R_{kjin,k} = \Ric_{jn,i} - \Ric_{ ji,n }$.
	 Using that $\Ric = \kappa \, g - \phi - \Hess_f $,  $g$ is parallel, and then using
 Lemma \ref{l:boch} gives 
\begin{align}	\label{e:e252}
	  \R_{kjin,k} = \Ric_{jn,i} - \Ric_{ ji,n }  =  \phi_{ ji,n } - \phi_{jn,i}  + f_{ ji,n } - f_{jn,i}      =   \phi_{ ji,n } - \phi_{jn,i}   + \R_{nijk   } f_k \, .
\end{align}
 Lemma \ref{l:boch} gives 
$  	V_{i,jk} = V_{i,kj} + \R_{kjin} V_n$.
 Working at a point where $\nabla_i e_j (p) = 0$ for the orthonormal frame $e_i$,
differentiating gives
 \begin{align}	 
 	V_{i,jkk} &=   \left( V_{i,kj} + \R_{kjin} V_n \right)_k = V_{i,kjk} + \R_{kjin,k} V_n  + \R_{kjin} V_{n,k}  \, .
 \end{align}
The Ricci identity gives
 \begin{align}
 	V_{i,kjk}  =V_{i,kkj}  + \R_{kjkn}V_{i,n} + \R_{kjin} V_{n,k}   = V_{i,kkj}  + \Ric_{jn}V_{i,n} + \R_{kjin} V_{n,k}    \, .
 \end{align}
 Using this gives
$  	V_{i,jkk}  =  V_{i,kkj}  + \Ric_{jn}V_{i,n}  + \R_{kjin,k} V_n  + 2\, \R_{kjin} V_{n,k} $.
The claim follows from this and \eqr{e:e252}.
\end{proof}

\begin{Cor}	\label{c:dbochner}
 If   $V$ is a vector field, then
    \begin{align}	 
 	(\cL \, \nabla \, V)_{i,j} &=( \nabla \, \cL \, V)_{i,j}  + \kappa \, V_{i,j}    + 2\,  \R_{kjin} V_{n,k} - \phi_{jn} \, V_{i,n} 
	+ V_n (  \phi_{ ji,n } - \phi_{jn,i}   ) \, .
	 \end{align}
\end{Cor}

\begin{proof}
We will work at a point where $\nabla_i e_j (p) = 0$ and  $g_{ij} = \delta_{ij}$.
Lemma \ref{l:bochHessA} gives that 
    \begin{align}	 
 	(\Delta \nabla V)_{ij} = V_{i,jkk}  =  (\nabla \Delta \, V)_{i,j}  + \Ric_{jn}V_{i,n}       +\R_{nijk   } f_k \,  V_n  + 2\,  \R_{kjin} V_{n,k} +V_n(   \phi_{ ji,n } - \phi_{jn,i}   )  \, . \notag
 \end{align}
On the other hand, $(\nabla_{e_j} \nabla_{\nabla f}V )_i= f_{jn} \, V_{i,n} + f_n \, V_{i,nj}$ so we get
    \begin{align}	 
 	(\Delta \nabla V)_{ij}   =  (\nabla \cL \, V)_{i,j}  +  f_{jn} \, V_{i,n} + f_n \, V_{i,nj} + \Ric_{jn}V_{i,n}       +\R_{nijk   } f_k \,  V_n  + 2\,  \R_{kjin} V_{n,k}  + V_n (  \phi_{ ji,n } - \phi_{jn,i}   ) \, . \notag 
 \end{align}
Using that  $\Ric + \Hess_f =\kappa \, g -  \phi$ and, by Lemma \ref{l:boch}, that $V_{i,jn} = V_{i,nj} + \R_{njik} V_k$, we get
    \begin{align}
 	(\cL \,  \nabla V)_{ij}   =  (\nabla \cL \, V)_{i,j}  + \kappa \, V_{i,j} - \phi_{jn} \, V_{ i,n} + 2\,  \R_{kjin} V_{n,k} 
	+ V_n (  \phi_{ ji,n } - \phi_{jn,i}   )  \, .
	\, \, \, \, \, \, \, \, \, \, \, \qedhere
 \end{align}
 \end{proof}

\begin{proof}[Proof of Theorem \ref{c:BHa}]
Set $W_{ij} = \dv_f^{*} \, V$.
Since $-2\, W_{ij} = V_{i,j} + V_{j,i}$, 
 Lemma \ref{l:bochHessA}  gives
 \begin{align}	 	\label{e:fromcBHa}
 	-2\, \Delta W_{ij} &= -2\, W_{ij,kk} =    V_{i,kkj}  +V_{j,kki}  + \Ric_{jn}V_{i,n} +\R_{nijk   } f_k \,  V_n  + 2\, \R_{kjin} V_{n,k}  \notag \\
	&+ \Ric_{in}V_{j,n}      +\R_{njik   } f_k \,  V_n + 2\, \R_{kijn} V_{n,k} 
	+  V_n (  2\, \phi_{ ji,n } - \phi_{jn,i} - \phi_{in,j}  )\, .
 \end{align}
Relabeling indices, using the symmetries of $\R$, and using the definition of $W$ gives
\begin{align}
	\R_{kjin} V_{n,k}     + \R_{kijn} V_{n,k} = 	\R_{kjin} V_{n,k}     + \R_{nijk} V_{k,n} =  	\R_{kjin} V_{n,k}     + \R_{kjin} V_{k,n} 
	= -2 \, \R_{kjin} W_{kn}   \, . \notag
\end{align}
Using this and $V_{i,kkj}  +V_{j,kki} = -2 \,   \dv_f^{*} \, \Delta \, V$, we can rewrite \eqr{e:fromcBHa} as
 \begin{align}	 	 
 	     2 \,  \dv_f^{*} \, \Delta \, V & = 2\, \Delta \, W + \Ric_{jn}V_{i,n} - 4 \, \R_{kjin} W_{kn} + f_k \left( \R_{nijk   } + \R_{njik   } \right)\,  V_n   
	 + \Ric_{in}V_{j,n}  \notag     \\
	 &+  V_n (  2\, \phi_{ ji,n } - \phi_{jn,i} - \phi_{in,j}  )
   \, .   
    \end{align}
Using  that $  	V_{i,jk} = V_{i,kj} + \R_{kjin} V_n$ by Lemma \ref{l:boch}, the first derivative term is  
 \begin{align}
 	-2 \,  \dv_f^{*} ( \nabla_{\nabla f} V) &= (V_{i,k}f_k)_j +  (V_{j,k}f_k)_i = 
	 V_{i,kj}f_k  + V_{i,k}f_{kj} +  V_{j,ki}f_k  +     V_{j,k}f_{ki } \notag \\
	 &=   
	 (V_{i,jk} -  \R_{kjin} V_n)f_k  + V_{i,k}f_{kj} +  (V_{j,ik } -  \R_{kijn} V_n )f_k  +     V_{j,k}f_{ki }
	 \, .
 \end{align}
Adding the last two equations, we get
\begin{align}	 	 
 	      2 \,  \dv_f^{*} \, \cL \, V &= 2\, L\, W + ( \Ric_{jn} + f_{jn}) V_{i,n}    + ( \Ric_{in} + f_{in}) V_{j,n}   \\
	      &+ f_k \left( \R_{nijk   } + \R_{njik   } -  \R_{kjin} -  \R_{kijn}\right)\,  V_n     
	      +  V_n (  2\, \phi_{ ji,n } - \phi_{jn,i} - \phi_{in,j}  )
	          \, . \notag
 \end{align}
 The first term on the last line vanishes because of the symmetries of the curvature tensor.  Using this and 
  $\Ric + \Hess_f = \kappa \, g -\phi$,  we get
\begin{align}	 	 
 	      2 \,  \dv_f^{*}\, \cL \, V &= 2\,  L \,  W + \kappa \, \left( V_{i,j} + V_{j,i} \right)-  \phi_{jn} V_{i,n} - \phi_{in}  V_{j,n}   
	      	      +  V_n (  2\, \phi_{ ji,n } - \phi_{jn,i} - \phi_{in,j}  ) \notag \\
	      &= 2 \, (L -\kappa) \, \dv_f^* \, V    -  \phi_{jn} V_{i,n} - \phi_{in}  V_{j,n} 	      
	            +  V_n (  2\, \phi_{ ji,n } - \phi_{jn,i} - \phi_{in,j}  ) \, .
 \end{align}
 This gives the first claim.  
The second claim follows from taking the adjoint of the first claim and using that
$
  \left(L\,\dv_f^*\right)^*=\dv_f\,L$ and 
 $ \left(\dv_f^*\,\left(\cL+\kappa \right)   \right)^*=\left(\cL+\kappa \right)\,\dv_f$.
   \end{proof}

   \subsection{Solitons}

For a soliton,   $(S+|\nabla f|^2 -2\, \kappa \, f)$ is constant (cf. \cite{ChLN} or Corollary \ref{c:gradS}) and 
   it is customary (when $\kappa \ne 0$)  to subtract a constant from $f$ so that
\begin{align}	\label{e:normalizef}
	S+ |\nabla f|^2 = 2\, \kappa \, f \, .
\end{align}
Combining this with the trace of the soliton equation gives that
	$\cL \, f =  n \, \kappa - 2 \, \kappa \, f$.
If $\kappa = \frac{1}{2}$ (i.e., a shrinker), then $S\geq 0$ by \cite{Cn} and we have that 
 \begin{align}	\label{e:fromccz}
 	f- S = |\nabla f|^2\leq f \, .
\end{align}
By Cao-Zhou,  \cite{CaZd}:   There exist $c_1, c_2$, depending only on   $B_1(x_0) \subset M$ so that 
 \begin{align}	\label{e:ccz}
 	\frac{1}{4} \left( r(x) - c_1 \right)^2 &\leq f (x) \leq  	\frac{1}{4} \left( r(x) + c_2 \right)^2 \, ,
 \end{align}
 where $r(x)$ is the distance to a fixed point $x_0$.

 The simplest shrinker is the Gaussian soliton $(\RR^n, \delta_{ij} , \frac{|x|^2}{4})$, followed by  cylinders 
 $ \SS^{\ell} \times \RR^{n-\ell} $ with the product metric where the sphere has Ricci curvature $\frac{1}{2}$ and  $f = \frac{|x|^2}{4}+ \frac{\ell}{2}$.  There are also shrinkers asymptotic to cones; see, e.g.,    \cite{FIK,KW}.

The next lemma gives a concentration for vector fields on any shrinker; cf. lemma $1.5$ of \cite{CM9}.
  We think of this as a concentration  because of the asymptotics \eqr{e:ccz}
of $f$. 
 
 \begin{Lem}	\label{l:L2RPa}
 If $M$ is a shrinker and $Y$ is
  any vector field or function in $ W^{1,2}$, then
 \begin{align}
	 \int |Y|^2\, (|\nabla f|^2-n)  \,\e^{-f}  &\leq \int |Y|^2\, (f-n)  \,\e^{-f}  \leq
	   4\, \int |\nabla Y|^2\,\e^{-f}\,   . \label{e:lemmamu} 
 \end{align}
  \end{Lem}
 
 \begin{proof}
 Since $f$ is normalized so that $\cL\,f+f=\frac{n}{2}$, we have that
 \begin{align}  
\int |Y|^2\, \left( f - \frac{n}{2} \right) \,\e^{-f}&= 2\int \langle\nabla_{\nabla f} Y,Y\rangle\,\e^{-f}
\leq 2\int |\nabla Y|^2\,\e^{-f}+\frac{1}{2}\int |Y|^2\,|\nabla f|^2\,\e^{-f}\, .  \notag
\end{align}
We get \eqr{e:lemmamu}  since $f \geq |\nabla f|^2 $ by  \eqr{e:fromccz}.  
 \end{proof}

 We will  use the following elementary interpolation inequality:
 
 \begin{Lem}	\label{l:RPcL}
 Given any shrinker, if $Y , \cL \, Y \in L^2$, then $Y \in W^{1,2}$, $\dv_f \, Y \in L^2$,  and $\| \nabla Y \|_{L^2}^2 \leq 2 \, \| Y \|_{L^2} \, \| \cL \, Y \|_{L^2}$.  If in addition $\dv_f \, \cL \, Y \in L^2$, then $\dv_f \, Y \in W^{1,2}$.  Finally, if the sectional curvature is bounded, then 
 \begin{align}
 	\| \nabla^2 Y \|_{L^{2}}^2 \leq \| \cL \, Y \|_{L^2}^2 + C\, \| \nabla Y \|_{L^2}^2 \, .
\end{align}
 \end{Lem}
 
 \begin{proof}
 Let $\eta$ be a cutoff function with $|\eta | \leq 1$. The Cauchy-Schwarz inequality, integration by parts and an absorbing inequality 
give (with $\| \cdot \| = \| \cdot \|_{L^2}$)
\begin{align}
	\| \cL \, Y \|    \| Y \| & \geq - \int \langle \cL\, Y , \eta^2  Y \rangle \, \e^{-f} 
	= \int \left( \eta^2  |\nabla Y|^2 + \langle \nabla \eta^2 ,   \frac{\nabla | Y|^2}{2} \rangle
	\right) \, \e^{-f} \notag \\
	&\geq \frac{ \| \eta \, \nabla Y \|^2}{2}  - 2\, \| |\nabla \eta | Y \|^2 \, . \notag
\end{align}
Taking a sequence of $\eta$'s converging to one and applying  the dominated convergence theorem gives that $Y \in W^{1,2}$ 
and  
\begin{align}
	\| \nabla Y \|_{L^2}^2 \leq 2 \, \| Y \|_{L^2} \, \| \cL \, Y \|_{L^2} \, .
\end{align}
  By this and Lemma \ref{l:L2RPa}, 
$\dv_f \, Y \in L^2$.

Since Lemma \ref{l:boch} gives that
	$\cL \, \dv_f \, Y =     \dv_f \, \cL \, Y -\frac{1}{2} \, \dv_f \, Y$, it follows that $\cL \, \dv_f  \, Y \in L^2$ if  $\dv_f \, \cL \, Y \in L^2$.  The first part of the lemma
	now gives that $\dv_f \, Y \in W^{1,2}$.
	 Finally, integrating Corollary \ref{c:dbochner}  gives  a $W^{2,2}$ bound when the sectional curvature is bounded.    
 \end{proof}

 We will need a $W^{1,2}$ localization result for eigenfunctions:

\begin{Lem}	\label{c:localize}
If $v \in W^{2,2}$ satisfies $\cL \, v = - \mu \, v$ and $\| v \|_{L^2} = 1$, then
\begin{align}
	\frac{s^2}{4} \, \int_{ f \geq \frac{s^2}{4}} \, \left\{ v^2 + |\nabla v|^2 \right\} \, \e^{-f} \leq 4\, \mu^2 + (n+2)\,  \mu + n
	\, .
\end{align}
\end{Lem}

\begin{proof}
Integrating $\frac{1}{2} \, \cL \, v^2 = |\nabla v|^2 - \mu \, v^2$ and  the drift Bochner formula  $\frac{1}{2} \, \cL \, |\nabla v|^2 = 
 |\Hess_{v}|^2 + \left( \frac{1}{2} - \mu \right) \, |\nabla v|^2$ gives
\begin{align}	\label{e:e326}
	\| \nabla v \|_{L^2}^2 = \mu {\text{ and }} \| \Hess_v \|_{L^2}^2 =  \left( \mu - \frac{1}{2} \right) \, \mu \, .
\end{align}
Applying Lemma \ref{l:L2RPa} to $v$ and to $\nabla v$ and adding these inequalities gives the claim.
 \end{proof}

\section{Diffeomorphisms and the $\cP$ operator}    \label{s:Poperator}

A key tool in this paper for dealing with the infinite dimensional gauge group is a natural second order system operator $\cP$ that seems to have been largely overlooked.    This operator is defined on vector fields and given by composing $\dv_f^{*}$ with its adjoint $\dv_f$ so 
 $\cP = \dv_f \circ\, \dv_f^{*}$.    In one dimension, $\cP = - \cL$, but in higher dimensions $\cP$ and $\cL$ are very different.

  A vector field $Y$ is a Killing field if the Lie derivative of the metric with respect to   $Y$ is zero, i.e., 
   $ \dv_f^{*} \, Y = 0$.  Since a Killing field is determined by its value and first derivative at a point, the space  of Killing fields is finite dimensional.     Integration by parts shows that if $Y \in W^{1,2}$ and $\cP \, Y \in L^2$, then
	\begin{align}
		\int \langle Y , \cP \, Y \rangle \, \e^{-f} = \| \dv_f^{*} (Y) \|_{L^2}^2 \, .
	\end{align}
Thus, the $L^2$ kernel of $\cP$ is the space
 $\cK_{\cP}$ of $L^2$ Killing fields.

\subsection{Basic properties of $\cP$}
 In this subsection, we prove the basic properties of $\cP$.  Many of the results are valid on any manifold and for any function $f$.  The results are strongest for gradient Ricci solitons - shrinking, steady or expanding.
   The next lemma relates $\cP$ and $\cL$. 
   
 \begin{Lem}	\label{l:dvastY}
Given a vector field $Y$, we have
 \begin{align}	\label{e:emy29}
 	-2\,  \cP \, Y   = \nabla \, \dv_f\, Y + \cL \, Y + \Hess_f(\cdot,Y)+\Ric (\cdot,Y) \, .
 \end{align}
 Thus, for a Ricci soliton, 
 	$-2\,  \cP \, Y   = \nabla \, \dv_f\, Y + \cL \, Y + \kappa \, Y$.
 \end{Lem}
 
 \begin{proof}
 Fix a point $p$ and let $e_i$ be an orthonormal frame  with $\nabla_{e_i} e_j (p) = 0$.  Set $h = \dv_f^{*} \, Y$, so that $-2\,h (e_i , e_j) = \langle \nabla_{e_i} Y , e_j \rangle + 
 \langle \nabla_{e_j} Y , e_i \rangle$. Working at $p$, we have
 \begin{align}
 	-2 \, \left( \dv_f \, h  \right) (e_i) &=   
	\nabla_{e_j}\langle \nabla_{e_i} Y , e_j \rangle + 
\nabla_{e_j} \langle \nabla_{e_j} Y , e_i \rangle - \langle \nabla_{e_i} Y , \nabla f \rangle -
 \langle \nabla_{\nabla f} Y , e_i \rangle \\
 &=   
	\langle \nabla_{e_j} \nabla_{e_i} Y , e_j \rangle + 
  \langle \cL\, Y , e_i \rangle - e_i \langle  Y , \nabla f \rangle  + \langle Y ,  \nabla_{e_i} \nabla f \rangle \notag \, .
 \end{align}
 Commuting the covariant derivatives introduces a curvature term, giving
  \begin{align}
 	-2 \, \left( \dv_f \, h  \right) (e_i) & =   
	\langle \nabla_{e_i} \nabla_{e_j} Y , e_j \rangle + \Ric (e_i ,  Y) +
  \langle \cL\, Y , e_i \rangle - e_i \langle  Y , \nabla f \rangle + \Hess_f (e_i , Y) \, .\notag
  \, \, \, \, \, \, \, \, \, \qedhere
 \end{align}
 \end{proof}

   We will next show that on any gradient Ricci soliton  $\cL$ and $\cP$ commute.  
 
  \begin{Pro}  \label{p:commPL}
    For a gradient Ricci soliton and any vector field $V$, 
   $\cL\,\cP\,V=\cP\,\cL\,V$ and 
 $ \cP\, \nabla\,\dv_f\,(V)=\nabla\,\dv_f\,(\cP\,V)$.
   \end{Pro}
   
   \begin{proof}
   By Theorem \ref{c:BHa}, 
$   \dv_f\,L \, \dv_f^{*} (V) = \dv_f\,\dv_f^{*} \left( \cL \, V + \kappa \, V \right) =\cP\,\left(\cL\,V+\kappa\,V\right)$.
    Moreover, Theorem \ref{c:BHa} with $h=\dv_f^*\,V$ gives
    \begin{align}
    \dv_f\,L \, \dv_f^{*} (V) =\left(\cL+\kappa\right)\,\dv_f\,\dv_f^*\,V=\left(\cL+\kappa\right)\,\cP\,V\, .
    \end{align}
   Combining these two equations and cancelling terms gives the first claim.  The second follows from the first together with Lemma \ref{l:dvastY}.
      \end{proof}

The next result characterizes $\cP$  locally on all vector fields.  
      
      \begin{Pro}	\label{p:ongrads}
      The operators $\cL$ and $\cP$ are self-adjoint.  Moreover, 
      \[ -\cP\,V =\left\{ \begin{array}{ll}
         \cL\,V =\nabla\,\cL\,u+\Hess_f(\cdot,\nabla u)+\Ric (\cdot,\nabla u)& \mbox{if $V=\nabla u$}\, ;\\
         \frac{1}{2}\,\left[\cL\,V+\Hess_f(\cdot,V)+\Ric (\cdot,V)\right]& \mbox{if $\dv_f\,(V)=0$}\, .
         \end{array} \right. \]  
         For a Ricci soliton, $\cL$ and $\cP$ preserve this orthogonal decomposition.
               \end{Pro}
      
      \begin{proof}
      We have already seen that both $\cL$ and $\cP$ are self-adjoint and that there is an orthogonal decomposition of vector fields into gradient of functions and those with $\dv_f=0$.  To compute $\cP \, \nabla u$, use
       Lemma \ref{l:dvastY} and Lemma \ref{l:boch} to get
\begin{align}
-2\,  \cP \, \nabla u   &= \nabla \, \dv_f\, \nabla u + \cL \, \nabla u + \Hess_f(\cdot,\nabla u) +\Ric (\cdot,\nabla u)\notag\\
&=2\,\cL\,\nabla u=2\,\nabla\,\cL\,u+2\,\left[\Hess_f(\cdot,\nabla u) +\Ric (\cdot,\nabla u)\right]\, .  \notag
\end{align}
 In particular, for a Ricci soliton $\cL$ and $\cP$ preserve the subspace of vector fields that are gradients of functions.    Next, if $\dv_f\,(V)=0$, then  Lemma \ref{l:dvastY} gives that $-2\,\cP\,V=\cL\,V+\Hess_f(\cdot,V)+\Ric (\cdot,\nabla u)$, implying that $-\cP\,V$ is as claimed.  Finally, for a Ricci soliton, if $\dv_f\,(V)=0$, then it follows from \eqr{e:commutaB} that  $\dv_f\,(\cL\,V)=0$, and thus $\dv_f\,(\cP\,V)=0$.     This shows that for a Ricci soliton both $\cL$ and $\cP$ preserve the orthogonal splitting.  
      \end{proof}

  \begin{Lem}  \label{l:HessianLem}
On any gradient Ricci soliton for any vector field $Y$
\begin{align}
\left( \cL + \kappa \right) \, \dv_f \, Y &= - \dv_f \, (\cP\,Y)\, ,\label{e:claim1}\\
\cL\,\nabla\,\dv_f\,(Y)&=-\nabla\,\dv_f\,(\cP\,Y)\, ,\label{e:claim2}\\
   	\left(L -\kappa \right)\, \Hess_{\dv_f \, (Y)}  &=   -\Hess_{  \dv_f \, (\cP\,Y)} \,  , \label{e:claim3} \\
	L\, \dv_f^*\,(Y)&= \Hess_{\dv_f\,(Y)}-2\, \dv_f^*\,(\cP\,Y) \, .
	\label{e:firstc}
\end{align}
\end{Lem}

\begin{proof}
Lemma \ref{l:boch} together with Lemma \ref{l:dvastY} gives that
 \begin{align}	 
	\cL \, \dv_f \, Y &=     \dv_f \, \left( \cL   - \kappa \right) \, Y = 
	-2\,  \dv_f \,  ( \cP\,Y)   - \cL  \, \dv_f\, (Y)  - 2\, \kappa \, \dv_f\,(Y)
	   \, .
\end{align}
Thus, \eqr{e:claim1}.   By Lemma \ref{l:boch},
$\cL \, \nabla u = \nabla \, \cL u +\kappa \,\nabla u$, so
  \eqr{e:claim1} gives \eqr{e:claim2}.  
Combining \eqr{e:claim1} with Corollary \ref{t:bochHess} gives \eqr{e:claim3}.  
Applying $\dv_f^*$ to Lemma \ref{l:dvastY} gives
\begin{align}
-2\,\dv_f^*\,(\cP\,Y)&=\dv_f^*\,(\cL\,Y)- \Hess_{\dv_f\,(Y)}+\kappa \,\dv_f^*\,(Y)\, .
\end{align}
  Theorem \ref{c:BHa}  gives that $L \, \dv_f^* \, Y = \dv_f^* \, \left( \cL + \kappa \right) \, Y$ and, thus, \eqr{e:firstc}.  
\end{proof}

This lemma is used in \cite{CM12} to show that if $Y$ is an $L^2$ Killing field on a gradient shrinking Ricci soliton, then either $Y$ preserves $f$ or the soliton splits off a line.

  \begin{Lem}	\label{l:interpcP}
For any gradient Ricci soliton if $Y$, $\cP\,Y\in L^2$, then $\dv_f\,(Y)$, $\nabla\, Y\in L^2$ and
\begin{align}
	\| \nabla Y \|_{L^2}^2 + \| \dv_f \, Y \|_{L^2}^2 \leq 2 \, \| Y \|_{L^2} \, \| (2\, \cP + \kappa) \, Y \|_{L^2} \, .
\end{align}
\end{Lem}

\begin{proof}
 Since  $Y$, $\cP\,Y\in L^2$, so is  $\left( 2\, \cP + \kappa \right) \, Y = - \cL \, Y - \nabla \dv_f \, Y$ (by Lemma \ref{l:dvastY}).  If $0 \leq \eta \leq 1$ has  compact support, then the Cauchy-Schwarz inequality and integration by parts   give
\begin{align}
& \| \left( 2\, \cP + \kappa \right) \, Y \|_{L^2} \, \| Y \|_{L^2} \geq  -\int \eta^2\,\langle \cL\,Y,Y\rangle\,\e^{-f}-\int \eta^2\,\langle \nabla\,\dv_f\,(Y),Y\rangle\,\e^{-f}\notag\\
&\quad \qquad = \| \eta \, \nabla Y \|_{L^2}^2 + \| \eta \, \dv_f \, (Y) \|_{L^2}^2 
 +2\int \eta\,\dv_f\,(Y)\,\langle \nabla \eta,Y\rangle\,\e^{-f} +2\int \eta\,\langle \nabla_{\nabla \eta}\,Y,Y\rangle\,\e^{-f}\,  . \notag
\end{align}
Using an absorbing inequality on the last two terms, then taking $\eta \to 1$ and applying the monotone convergence theorem gives the lemma.
\end{proof} 

\begin{Lem}	\label{l:elliptic}
For any gradient Ricci soliton if $Y$ is a weak solution of
  $(\cP - \lambda) \, Y = V_0$, where $V_0$ is smooth and $Y , \dv_f \, (Y) \in L^2_{loc}$, then $Y$ is smooth.
\end{Lem}

\begin{proof}
Given any smooth $X$ with compact support, we have 
	$\int \langle X , V_0 \rangle \, \e^{-f} = \int \langle (\cP - \lambda) \, X  , Y \rangle \, \e^{-f}$. If $X = \nabla u$ where $u \in C^{\infty}_c$, then $\cP \, \nabla u = \nabla \, \left( \cL + \kappa \right) \, u$ by 
	Proposition \ref{p:ongrads} so
	\begin{align}
		- \int   u  \, \dv_f \, (V_0)  \, \e^{-f} & = 
		\int \langle \nabla u  , V_0 \rangle \, \e^{-f} = \int \langle (\cP - \lambda) \, \nabla u  , Y \rangle \, \e^{-f} = 
		 \int \langle \nabla \, \left( \cL + \kappa - \lambda \right) \, u  , Y \rangle \, \e^{-f} \notag \\
		 & = - 
		 \int \left( \cL + \kappa - \lambda \right) \, u  \, \dv_f \, ( Y) \, \e^{-f} \, .
	\end{align}
The last equality used that $\dv_f \, (Y) \in L^2_{loc}$.  It follows that $\dv_f \, (Y)$ is an $L^2_{loc}$ weak solution to 
$\left( \cL + \kappa - \lambda \right) \, \dv_f \, (Y) = \dv_f \, (V_0)$.  Since $V_0$ is smooth, elliptic regularity gives that $\dv_f \, (Y)$ is also smooth.  Since $-2\, \cP = \left( \cL + \kappa \right) + \nabla \, \dv_f$ by
Lemma \ref{l:dvastY}, we have
\begin{align}
	2\, \int \langle X , V_0 \rangle \, \e^{-f}   = - \int \left\{   \langle ( \cL + \kappa + 2\, \lambda ) \, X  , Y \rangle + \langle X , \nabla \, \dv_f \, Y \rangle \right\} \, \e^{-f}
\, . 
\end{align}
It follows that $Y$ is an $L_{loc}^2$ weak solution to $( \cL + \kappa + 2\, \lambda ) \, Y = -\nabla \, \dv_f \, (Y) - 2\, V_0$.  Since the right-hand side is smooth, elliptic regularity gives that so is $Y$.
\end{proof}

\section{Optimal growth bounds}

In this section, we will prove the optimal growth bound Theorem \ref{t:PGmeta}.    Throughout this section $(M,g,f)$ will be assumed to satisfy
  \eqr{ee:aronson1} and   \eqr{ee:aronson2}.  This applies to all shrinkers in both Ricci flow and MCF, but is much more general than that.

Since $|\nabla \sqrt{f}| \leq \frac{1}{2}$ by  \eqr{ee:aronson2},  the function $b = 2\,\sqrt{f}$
satisfies  $|\nabla b|\leq 1$ as in \cite{CaZd}, cf. \cite{CxZh1}.
Throughout, $\lambda > 0$ is a constant and $u$ is a tensor.  We will often assume that
 \begin{align}
 	\langle \cL \, u  , u \rangle  \geq  - \lambda \, |u|^2 \, ;    \label{ee:herelambda}
\end{align}
 this includes eigentensors with $\cL \, u = - \lambda \, u$.  
 To understand the growth of $u$,  we will study a weighted average of $|u|^2$ on level sets of  $b$
  \begin{align}
 	I(r) &= r^{1-n} \, \int_{b = r} |u|^2 \, |\nabla b| \, . \label{ee:Iofr1} 
\end{align}
This is  defined at regular values of $b$, but   extends continuously to all values to be differentiable a.e.~and absolutely continuous.  The  weight $|\nabla b|$ will play a crucial role (cf. \cite{CM7, CM12, CM3, C, AFM,BS,GiV}).   The  growth of $I$   will be bounded above in terms of  the solid integral
 \begin{align}
 	D(r) &= r^{2-n} \, \e^{ \frac{r^2}{4}} \, \int_{b < r} \left( |\nabla u|^2 + \langle \cL \, u , u \rangle	\right)  \, \e^{-f} \, .
	\label{ee:integralformD}
 \end{align}
 The frequency $U = \frac{D}{I}$ is defined when $I$ is positive and will measure the    growth of $\log I$.

 The next theorem is the precise version of Theorem \ref{t:PGmeta}.  It shows that an $L^2$ tensor satisfying \eqr{ee:herelambda}
  has frequency bounded by $2\, \lambda$ and, accordingly, it grows at most polynomially at this rate.  This may seem surprising since the weight $\e^{-f}$ decays rapidly, so the $L^2$ condition a priori allows extremely rapid growth.   The theorem holds very generally and does not assume any cone or dilation structure.
 
 \begin{Thm}   \label{et:main}
 Suppose $u , \cL \, u \in L^2$, \eqr{ee:aronson1}, \eqr{ee:aronson2}, \eqr{ee:herelambda} hold, and $u$ does not vanish identically outside a compact set.
 Given $\epsilon > 0$, there exists $R=R(n,\lambda,\epsilon)$ so if $r>R$, then
\begin{align}
U(r)\leq 2\,\lambda\,\left(1+\frac{\bar{\mu}+\epsilon}{r^2}\right)\, ,	\label{ee:0p9}
\end{align}
where $\bar{\mu} = 2\, n + 4\, \max \{ \lambda -1 , 0\}$.  Moreover, 
  for all $r_2>r_1>R$  
 \begin{align}
 I(r_2)\leq I(r_1)\,\left(\frac{r_2}{r_1}\right)^{4\,\lambda} \, \e^{(\bar{\mu}+\epsilon)\, \left( r_1^{-2} -r^{-2}_2 \right)}\, .	\label{ee:IU08}
 \end{align}
 \end{Thm}
 
 This is sharp for the Ornstein-Uhlenbeck operator on $\RR^n$  where the $L^2$ eigenfunctions are Hermite polynomials with degree twice the eigenvalue.  The upper bound
  \eqr{ee:0p9} is sharp not just in the $2\, \lambda$ in front, but in all the other constants as  well as can be seen from the Hermite polynomials.
The  $R$ in Theorem \ref{et:main} does not depend  on $f$, $M$ or $S$. The  theorem still holds if \eqr{ee:aronson1}, \eqr{ee:aronson2}, and 
 \eqr{ee:herelambda} hold outside of a compact set.  Moreover, it holds with obvious changes when the constant $n$ in \eqr{ee:aronson1} is replaced by any other constant.    
 Finally, note that $u$ cannot vanish on an open set if $u$ has unique continuation, e.g. if $\cL \, u = - \lambda \, u$.

There is a long history of studying the growth of solutions to differential equations, inequalities, and systems.  At a very rough level, there are two main techniques.  The first, exemplified in the work of Carleman and H\"ormander, is to consider weighted $L^2$ norms with growing weights.  The second, seen for instance in the work of Hadamard and Almgren, is to study the growth of spherical maxima or averages.  
  Almgren's frequency has  been used to show unique continuation  and structure of the nodal sets; prior to this, the main tool in unique continuation was Carleman estimates that still is the primary technique.
Almgren's frequency bounds relied on scaling for $\RR^n$; cf. \cite{CM12,CM3}.  

As an application,  polynomially growing ``special functions''  are dense in $L^2$. This  gives   manifold versions of some very classical problems in analysis.
Whereas Weierstrass's approximation theorem shows  that polynomials are dense among continuous functions on any compact interval, the classical Bernstein problem, \cite{Lu}, dating back to 1924, asks if polynomials are dense on $\RR$ in the weighted $L^p(\e^{-f}\,dx)$ space if $f$ is assumed to grow sufficiently fast at infinity.    On the line, the Hermite polynomials are dense in $L^2(\e^{-\frac{|x|^2}{4}}\,dx)$ and Carleson (and implicitly Izumi-Kawata) showed that polynomials are dense in $L^p(\e^{-|x|^{\alpha}}\,dx)$ if and only if $\alpha\geq 1$.    A similar problem in several complex variables is the {\it{completeness problem}}, going back to Carleman in 1923, about density of polynomials in weighted $L^2$ spaces of holomorphic functions; \cite{BFW}.

For the applications to $\cP$,
we will need a more general Poisson version  where $u$ satisfies  
\begin{align}
	\langle \cL \, u , u \rangle  \geq - \lambda \, |u|^2 - \psi  \, ,  \label{ee:poisson}
\end{align}
where   $\psi$ is a nonnegative function.  
 Define the quantity $J$ by
 \begin{align}    \label{ee:J}
 	J(r) = \int_{b < r} b^{2-n} \, \psi  \, .
 \end{align}
  The next theorem gives  polynomial growth in terms of $\lambda$ and $J$.

 \begin{Thm}   \label{et:main2A}
If $u , \cL \, u \in L^2$,  \eqr{ee:aronson1}, \eqr{ee:aronson2}, \eqr{ee:poisson} hold, $\delta\in (0,2)$ and $r_2 > r_1 \geq R(\lambda,n,\delta)$,
then
 \begin{align}
	  I(r_2) \leq \left( \frac{r_2}{r_1} \right)^{ 4 \, \lambda +  \delta} \, \left\{ I(r_1) +  \frac{20 \, \sup J}{4 \, \lambda + \delta} \right\} \, .		\label{ee:main2A}
\end{align}
  \end{Thm}

One application will be to  gradient shrinking Ricci solitons.
  The standard drift Bochner formula gives that if $\cL \, v = -\left( \frac{1}{2} + \lambda \right) \, v$, then $\cL \, \nabla v = - \lambda \, \nabla v$ and
  \eqr{ee:herelambda} applies to $u = \nabla v$:

\begin{Cor}	\label{ec:gsrs}
If $(M,g,f)$ is a gradient shrinking soliton, then \eqr{ee:0p9} and \eqr{ee:IU08} hold if  $u = \nabla v$ where $v$ is an eigenfunction with eigenvalue $\lambda + \frac{1}{2}$.
\end{Cor}

The papers \cite{Be,CM11}  developed frequencies   for conical and cylindrical MCF shrinkers (cf. \cite{Wa}).
These results were perturbative in that they assumed the existence of an exhaustion
 function $r$ that behaves like Euclidean distance up to higher order.
 For instance,  in \cite{Be}, it was assumed that
  $\left| |\nabla r| - 1 \right| = O(r^{-4})$ and
 $\left| \Hess_{r^2} - 2 \, g \right| = O(r^{-2})$.
Theorems \ref{et:main}, \ref{et:main2A}, in contrast, hold very generally, including for all shrinkers in both Ricci flow and MCF and make no use of
any approximate conical structure.   
A much weaker version of Theorem \ref{et:main}, that was not relative, was proven in \cite{CM9} in the special case of MCF.  

\subsection{The level sets of $b$ and the properties of $I$ and $D$}

We will define   $D(r)$ and $I(r)$  as solid integrals over sub-level sets $\{ b < r\}$ of a proper $C^n$ function $b$.
For these functions to be continuous, we must show that level sets of $b$ have measure zero. This is (2) in the next lemma; (1) will be used to prove absolute continuity,  while (3) will be used to show that $I>0$.  Since $b$ is $C^n$, Sard's theorem gives that almost every level set is regular.

\begin{Lem}   \label{el:preparation}
Suppose $f:M\to \RR$ is a proper function with $\cL\,f=\frac{n}{2}-f$.  Let $\cC$ denote the set of critical points of $f$
and $\cH_r$ the boundary of $\{f>\frac{r^2}{4}\}$.  We get for $r>\sqrt{2\,n}$ that:
\begin{enumerate}
\item The critical set $\cC$ in $\{ f> \frac{n}{2} \}$  is locally contained in  a smooth $(n-1)$-manifold.
\item Each level set $\{ f = c \}$ for $c \geq \frac{n}{2}$ has $\cH^n ( \{ f = c \}) = 0$.
\item The regular set $\cR_r=\cH_r\setminus\cC$ is dense in $\cH_r$.     
\end{enumerate}
\end{Lem}

\vskip1mm
The nodal sets of eigenfunctions have a great deal of structure, but the value zero is special and many properties do not hold for non-zero values.  In fact, it is possible  to have a level set that is entirely critical, as  occurs at the local extrema for the radial eigenfunction $J_0 (|x|)$ on $\RR^2$ where $J_0$ is the Bessel function of the first kind.   However, by (3), this does not  occur for the subset $\cH_r$ of $\{ f= r\}$ that is the boundary of $\{ f > r \}$.

\begin{proof}[Proof of Lemma \ref{el:preparation}]
Note first that  $\cL\,f<0$ on $\{f>\frac{n}{2}\}$ and, thus,  $\Delta\,f<0$ on $\cC\cap \{f>\frac{n}{2}\}$.    Working in a neighborhood of a critical point we can therefore choose a coordinate system $\{x_i\}$ so that  $\partial_{x_1}^2f<-1$.  If $x\in \cC$,  then $\partial_{x_1}f(x)=0$ and thus by the implicit function theorem we can choose a new coordinate system in a neighborhood of $x$ so that in those coordinates $\{\partial_{x_1}f=0\}\subset \{y_1=0\}$ and so that $\partial_{x_1}$ is transverse to $\{y_1=0\}$.  We therefore have that (nearby) $\cC\subset \{\partial_{x_1}f=0\}\subset \{y_1=0\}$.     This gives (1). 

 For $c > \frac{n}{2}$, claim (2) follows from (1) since $\{ f = c \} \setminus \cC$ is a countable union of $(n-1)$-manifolds.  The borderline case $c = \frac{n}{2}$ in (2) follows from \cite{HHL}.
 
 We turn next to (3).  Note first that at $x=(x_1,\cdots,x_n)\in \cC$ if we let $h(s)=f(s,x_2,\cdots,x_n)$, then $h'(x_1)=0$ and $h''(x_1)<0$ so $h$ has a strict local maximum at $x_1$.  In particular, any neighborhood of any $x\in \cC\cap \{f>\frac{n}{2}\}$ intersects $\{f<f(x)\}$.  
Suppose now  that the conclusion (3) fails; so suppose that there exists $x\in \cH_r$ and a neighborhood $O$ so that $O\cap \cH_r\subset \cC$.  It follows that $O\cap \cH_r\subset \{y_1=0\}$.  Since $O\cap \cH_r$ separates the two non-empty sets $O\cap \{f>\frac{r^2}{4}\}$ and $O\cap \{\frac{r^2}{4}>f\}$ and $O\cap \cH_r$ is contained in $\{y_1=0\}$ it follows that $O\cap \cH_r=O\cap \{y_1=0\}$ and after possibly changing the orientation of $y_1$ we may assume that $O\cap \{y_1>0\}\subset \{f>\frac{r^2}{4}\}$ and $O\cap \{y_1<0\}\subset \{f<\frac{r^2}{4}\}$.     This, however, contradicts that at $x$ we have that $\partial_{x_1}^2f<0$ and $\partial_{x_1}$ is transverse to the level set $\{y_1=0\}$ so both $O\cap \{y_1>0\}$ and $O\cap \{y_1<0\}$ contains points where $f<f(x)=\frac{r^2}{4}$.  
\end{proof}

The functions $I(r)$, $D(r)$ and $U(r)$ may not be differentiable everywhere, but they will be absolutely continuous and differentiable a.e.
A function $Q(r)$ is  {\it{absolutely continuous}} on an interval $\cI$ if for every $\epsilon > 0$, there exists $\delta > 0$ so that if $\cup_{\alpha} (r_{\alpha} , R_{\alpha})$ is a finite disjoint union of intervals  in $\cI$ with $\sum (R_{\alpha} - r_{\alpha}) < \delta$, 
then we have $\sum \left| Q(R_{\alpha}) - Q(r_{\alpha}) \right| < \epsilon$.  
Absolutely continuous functions are precisely the ones where the  fundamental theorem of calculus holds (\cite{F}, page $165$): $Q$ is absolutely continuous if and only if 
  it is continuous, differentiable a.e., the derivative is in $L^1$, and  for every $r_1 < r_2$
\begin{align}	\label{ee:AC}
	Q(r_2) - Q(r_1) = \int_{r_1}^{r_2} Q'(t) \, dt \, .
\end{align}
We will use   the following standard fact:
If $Q_1$ and $Q_2$ are absolutely continuous and $W:\RR^2 \to \RR$ is Lipschitz on the range of $(Q_1 , Q_2)$, then $W(Q_1 , Q_2)$ is  absolutely continuous.

\begin{Lem}	\label{el:coareaapp}
Suppose that $b$ is a proper $C^n$ function and $\cH^n (|\nabla b| = 0) = 0$ in $\{ b \geq r_0\}$ for some fixed $r_0$.  If $g$ is a bounded function and $Q(r) = \int_{r_0 < b < r} g$, 
then $Q$ is absolutely continuous and  $Q'(r) = \int_{b =r} \frac{g}{|\nabla b|}$ a.e.
\end{Lem}

\begin{proof}
By separately considering the positive and negative parts of $g$, it suffices to assume that $g\geq 0$ is bounded.
Define a sequence of functions $Q_i$ by
\begin{align}
	Q_i (r) = \int_{r_0 < b < r}  \frac{g\, |\nabla b|}{|\nabla b| + i^{-1}} \, .
\end{align}
The functions $ \frac{g\, |\nabla b|}{|\nabla b| + i^{-1}}$ are bounded above by $g$ everywhere and converge to the bounded function $g$ a.e. (since  $\cH^n (|\nabla b| = 0) = 0$), so $\lim_{i \to \infty} \, Q_i (r) = Q(r)$ by
the dominated convergence theorem. 
Define functions $q_i(t)$ and $q(t)$ at regular values $t$ of $b$ by
\begin{align}
	q_i (t) = \int_{b=t}  \frac{g}{|\nabla b| + i^{-1}}  {\text{ and }}
	q(t)  = \int_{b=t}  \frac{g}{|\nabla b|} \, .
\end{align}
Since $b$ is $C^n$, Sard's theorem ($3.4.3$ in \cite{F}) gives that a.e.~$t$ is a regular value of $b$ and, thus, these
 functions are defined a.e.
The co-area formula (\cite{F}, page $243$) gives that
\begin{align}	\label{ee:fromco}
	Q_i(r) = \int_{r_0}^r \, q_i (t) \, dt \, .
\end{align}
The sequence $q_i$ is monotonically increasing with $q_i \leq q_{i+1} \leq \dots < q$.
Moreover, $q_i$ converges to $q$ a.e.  The monotone convergence theorem gives that
\begin{align}
	\lim_{i\to \infty} \, \int_{r_0}^r \, q_i (t) \, dt = \int_{r_0}^r \, q(t) \, dt \, .
\end{align}
Combining this with \eqr{ee:fromco} and $\lim_{i \to \infty} \, Q_i (r) = Q(r)$  gives the lemma.
\end{proof}

\subsubsection{Absolute continuity of $I$ and $D$}

In the remainder of this section, we specialize to $M$ non-compact and $f$ satisfying \eqr{ee:aronson1} and \eqr{ee:aronson2} and $b = 2\, \sqrt{f}$.  It follows that
   \begin{align}	\label{ee:rho1}
 	|\nabla b|^2 &= 1 - \frac{4\, S}{b^2} \leq 1 \, , \\
	b \, \Delta\, b &=  n - |\nabla b|^2 - 2 \, S 
	\, .  \label{ee:rho2}
 \end{align}
 Since $f$ is nonnegative and proper, then so is $b$ and, thus, the level sets of $b$ are compact.    Furthermore, Lemma 
 \ref{el:preparation} applies and, thus, so does Lemma \ref{el:coareaapp}.
 
  The definition \eqr{ee:Iofr1} of $I(r)$ at regular values of $b$ will be extended continuously  to all values next.  To do this, choose a regular value $r_0 < 2\, \sqrt{2n}$ of $b$ and set
  \begin{align}
 	 I(r) 	 &=
 \int_{r_0 < b < r} b^{1-n} \, \left\{  \, \langle \nabla |u|^2 , \nabla b \rangle + \frac{|u|^2}{b^3} \, 2\, S \left( 2\,n - b^2	\right)
	\right\}
	+ r_0^{1-n} \, \int_{b= r_0} |u|^2 \, |\nabla b|  \, . \label{ee:Iofr1A}
 \end{align}
 The reason for stopping the integral at $b=r_0$ is that $b^{1-n}$ and $S\, b^{-2-n}$ might not be integrable in the interior if $\min b = 0$.

 \begin{Lem}	\label{el:IBDR}
  At regular values $r$ of $b$, 
the definitions \eqr{ee:Iofr1} and \eqr{ee:Iofr1A}  of $I(r)$ agree and 
   \begin{align}
   	D(r) &= \frac{r^{2-n}}{2} \, \int_{b = r} \langle \nabla |u|^2 , \frac{\nabla b}{|\nabla b|} \rangle \, . \label{ee:firstclaim}
\end{align}
  \end{Lem}
  
  \begin{proof}
  To see that \eqr{ee:Iofr1} and \eqr{ee:Iofr1A}  agree at regular values, observe that
 the unit normal to the level set $b = r$ is given, at regular points, by $\nn = \frac{\nabla b}{|\nabla b|}$, so we can rewrite \eqr{ee:Iofr1}
  \begin{align}
 	  r^{1-n} \, \int_{b=r} |u|^2 |\nabla b| &- r_0^{1-n} \, \int_{b= r_0} |u|^2 \, |\nabla b|  =  \int_{r_0 < b < r} \dv \, (|u|^2 \, b^{1-n} \, \nabla b) \notag \\
	 &=
	\int_{r_0 < b < r} b^{1-n} \, \left\{  \, \langle \nabla |u|^2 , \nabla b \rangle + |u|^2 \, \left( \Delta \, b - \frac{(n-1) \, |\nabla b|^2}{b}
	\right)
	\right\} \, .   \label {ee:Iofr2} 
 \end{align}
  By \eqr{ee:rho1} and \eqr{ee:rho2}, we have that
 $b \, \Delta\, b =  n - |\nabla b|^2 - 2 \, S$ and $|\nabla b|^2 = 1 - \frac{4\, S}{b^2} $ and, thus,
 \begin{align}
 	b\, \left( \Delta \, b - \frac{(n-1) \, |\nabla b|^2}{b} \right) & = n\,(1-|\nabla b|^2) - 2 \, S 
	= \frac{2\, S}{b^2} \, (2\,n - b^2) \, .
 \end{align}
 Substituting this into \eqr{ee:Iofr2} gives \eqr{ee:Iofr1A}.
 The divergence theorem gives
  \begin{align}
\int_{b = r} \langle \nabla |u|^2 , \frac{\nabla b}{|\nabla b|} \rangle =  \e^{ \frac{r^2}{4}} \, \int_{b < r}
	\dv \, \left( \nabla |u|^2 \, \e^{-f} 
	\right) =  \e^{ \frac{r^2}{4}} \, \int_{b < r}
	   \cL \, |u|^2	   \, \e^{-f} 
	\, .
 \end{align}
    Multiplying this by $ 	\frac{r^{2-n}}{2} $ gives \eqr{ee:firstclaim}.  
  \end{proof}

 \begin{Lem}	\label{el:Idiff}
  Both $I(r)$ and $D(r)$ are absolutely continuous with derivatives given a.e. by
   \begin{align}
 	 	 I'(r) &= 	r^{1-n} \, \int_{b = r}  \, \langle \nabla |u|^2 , \frac{\nabla b}{|\nabla b|} \rangle + \left( 2\,n\,r^{-2} - 1	\right)\,r^{1-n}\int_{b=r}\frac{2\, S \, |u|^2}{r\, |\nabla b|} \, ,  \label{ee:Iofr3} \\
		 D'(r) &= \frac{2-n}{r} \, D + \frac{r}{2} \, D + \frac{r^{2-n}}{2} \, \int_{b = r}
	\frac{  \cL \, |u|^2 }{|\nabla b|} \, . \label{ee:Dpri}
\end{align}
Where $I$ is positive $\log I$ is absolutely continuous and the derivative is given a.e. by
\begin{align}
	r\, (\log I)'(r) = 2\, U + (2\,n\,r^{-2}-1)\,\frac{2\, r^{1-n}}{I}  \int_{b = r} \frac{S\,|u|^2}{ |\nabla b|}  \, . \label{ee:secondclaim}
\end{align}
  Furthermore,   $(\log I)' \leq 2 \, U/r$ a.e. when
 $r \geq \sqrt{2\,n}$. 
  \end{Lem}

 \begin{proof}
   Lemma \ref{el:coareaapp} applies to both $I$ and $D$ and, thus, both are absolutely continuous   and 
  $I
'$ is given a.e. by  \eqr{ee:Iofr3} and $D'$ is given a.e. by \eqr{ee:Dpri}.  
Equation \eqr{ee:secondclaim} follows from \eqr{ee:firstclaim} and  \eqr{ee:Iofr3}.  
  Since $S\geq 0$, we see that $ \left(\log I\right)'=\frac{I'}{I}\leq \frac{2\,U}{r}$
 for $r\geq \sqrt{2\,n}$.  
 \end{proof}

 \subsection{Positivity of $I(r)$}	\label{es:Ipo}

We show next that $I(r) > 0$ when $r$ is sufficiently large:
 
  \begin{Pro}	\label{ep:Ipositive}
 If $ u , \cL u  \in L^2$ and \eqr{ee:herelambda} holds, then either
 \begin{itemize}
 \item[(A)] $I(r) > 0$ for every $r > 2 \, \sqrt{n+ 4\, \lambda}$, or
 \item[(B)] $u$ vanishes identically outside of a compact set.
 \end{itemize}
  \end{Pro}
  
  An immediate consequence of (A) in Proposition \ref{ep:Ipositive} is that 
 $U(r)$ is well-defined  and absolutely continuous for  $r > 2 \, \sqrt{n+ 4\, \lambda}$, and $U'$ is given a.e. by
 \begin{align}
 	U'(r) = \frac{D'}{I} - \frac{D \, I'}{I^2} \, .
 \end{align}

 The next elementary lemma shows that $|u| \in W^{1,2}$ and $|u| \, |\nabla f| \in L^2$ if $u , \cL \, u \in L^2$ (cf. \cite{CxZh2,CM9}).

\begin{Lem}	\label{el:W12}
 If $ u , \, \cL \, u \in L^2$, then $|\nabla |u|| $, $|\nabla u|$, $|u| \, \sqrt{f}$, and  $|u| \, |\nabla f|$ are all in $L^2$.
\end{Lem}

\begin{proof}
By the Kato inequality  and \eqr{ee:aronson2},   $|\nabla |u|| \leq |\nabla u|$ and  $|\nabla f|^2 \leq f$. Thus,  it suffices to prove that $|\nabla u| , |u| \, \sqrt{f} \in L^2$.
We show first that $|\nabla u| \in L^2$.
Let $\eta$ be   a compactly supported function with $| \eta | , |\nabla \eta| \leq 1$.  Since
	$\cL \, |u|^2 = 2 \, |\nabla u|^2 +2 \, \langle u , \cL \, u \rangle$, 
applying the divergence theorem to $\eta^2 \, \nabla |u|^2 \, \e^{-f}$ gives 
\begin{align}
	 \int \eta^2 \, |\nabla u|^2 \,  \e^{-f} \leq  \| u \|_{L^2} \, \| \cL \, u \|_{L^2} +2 \, \int |\eta| \, |\nabla \eta|  \, |u| \, |\nabla |u|| \, \e^{-f} \, .
\end{align}
Using $|\nabla |u|| \leq |\nabla u|$ and the absorbing inequality 
$2 \, |\eta| \, |u| \, |\nabla u| \leq 2 \,  |u|^2 + \frac{1}{2} \, \eta^2 \, |\nabla u|^2$, we can absorb the $|\nabla |u||$ term and then apply the monotone convergence theorem for a sequence of $\eta$'s going to one everywhere gives that $|\nabla u| \in L^2$.
To see that $|u| \, \sqrt{f} \in L^2$, apply the divergence theorem to $\eta^2 \, |u|^2 \, \nabla f \, \e^{-f}$ and use that $\cL \, f = \frac{n}{2} - f$ to get
\begin{align}
	\int \eta^2 \, |u|^2 \, \left( f - \frac{n}{2} \right) \, \e^{-f} \leq 2 \, 
	\int \left\{ \eta^2 \, |u| \, |\nabla |u|| \, |\nabla f| + |\eta| \, |\nabla \eta| \, |u|^2 \, |\nabla f|
	\right\} \, \e^{-f} \, .
\end{align}
Using the bound $|\nabla f|^2 \leq f$, we can use absorbing inequalities on both terms on the right and then use that $|u|, |\nabla |u||$ are in $L^2$ to conclude that $|u| \, \sqrt{f} \in L^2$.
\end{proof}

    We will need a few preliminaries, including the following consequence of 
  Lemma \ref{el:preparation}:
 
\begin{Cor}     \label{ec:preparation}
If $I(r)=0$ and \eqr{ee:herelambda} holds, then $u\equiv 0$ on $\cH_r$.    
\end{Cor}

\begin{proof}
Suppose $x\in \cH_r$ with $|u|(x)>0$.  Since $u$ is continuous it follows from Lemma \ref{el:preparation} that there exists another point $y\in \cH_r\setminus \cC$ where $|u|(y)>0$.   Since $y$ is a regular point, then in a neighborhood of $y$ we have that $|u|\geq \frac{|u|(y)}{2}>0$, $|\nabla b|\geq \frac{|\nabla_y b|}{2}>0$.  It follows that there exists an $\nu>0$ such that  if $s$ be any regular value  sufficiently close to $r$, then the level set $b=s$ is a smooth hyper-surface and $I(s)\geq \nu>0$.    The claim follows.    
\end{proof}

 \begin{proof}[Proof of Proposition \ref{ep:Ipositive}]
 Suppose that (A) fails and, thus, $I(r) = 0$ for some $r > 2 \, \sqrt{n+4\, \lambda}$.    
 By Corollary \ref{ec:preparation},  we know that $|u| =0$ on $\cH_r = \partial \{ b > r \}$.      Assume (B) also fails and
 choose a connected component $\Omega $ of $\{ |u| > 0 \}$ with
 	$\Omega \subset \{ b > r \}$.  
This will lead to a contradiction.

 By Lemma \ref{el:W12}, $|u|$, $|u| \, |\nabla f|$, $|\nabla u|$ and $|\nabla |u||$ are all in $L^2$.
 For each $j$, let 
 $\eta_j: \RR \to [0, \infty)$ be a smooth function with $0 \leq \eta_j' \leq 4$ and
 \begin{align}
 	\eta_j (x) = 
	\begin{cases}
	x  & {\text{ for }} \frac{1}{j} \leq x \, , \\
	0 & {\text{ for }} x \leq \frac{1}{2j} \, .
	\end{cases}
 \end{align}
 Let $\chi$ be the characteristic function of $\Omega$, i.e, $\chi$ is one on $\Omega$ and zero otherwise, and 
  define  
  	$v_j = \eta_j (|u|) \, \chi_{\Omega} $.
Note that each $v_j$ is smooth on all of $M$ and $v_j \in W^{1,2}$ since $v$ is and $\eta_j$ is Lipschitz. 
Moreover, $v_j$ has support in $\{ b \geq r \}$ since $\Omega \subset \{ b > r \}$.

Let $V$ be a vector field with $V \in L^2$ and $v_j \, \left( \dv \, V - \langle V , \nabla f \rangle \right) \in L^1$.  Given $\eta$ with compact support and  $|\eta| , \, |\nabla \eta| \leq 1$, applying the divergence theorem to $ \eta \, v_j \, V \, \e^{-f} $ gives
\begin{align}	 
	\int \eta \,  \left( \langle \nabla v_j , V \rangle + v_j \, \left( \dv \, V - \langle V , \nabla f \rangle \right) \right) \, \e^{-f}  = - \int v_j \, \langle V , \nabla \eta \rangle \, \e^{-f}   \, .
\end{align}
Taking a sequence of $\eta$'s converging to one, the dominated convergence theorem gives
\begin{align}	\label{ee:willdom}
	\int \left( \langle \nabla v_j , V \rangle + v_j \, \left( \dv \, V - \langle V , \nabla f \rangle \right) \right) \, \e^{-f} = 0 \, .
\end{align}
By the Lipschitz bound on $\eta_j$ and the Kato inequality,
 $|v_j| \leq 4 \, |u| $ and  $ |\nabla v_j| \leq   4 \, |\nabla u|$.
Furthermore, $v_j \to |u| \, \chi$ and $\nabla v_j \to \chi \, \nabla |u|$ a.e.~(since $\nabla |u| = 0$ a.e. on  $\{ |u| = 0\}$).
Thus, applying the dominated convergence theorem to \eqr{ee:willdom} gives
\begin{align}	\label{ee:nowdom}
	\int_{\Omega} \left( \langle \nabla |u| , V \rangle + |u|  \, \left( \dv \, V - \langle V , \nabla f \rangle \right) \right) \,  \e^{-f} = 0 
	\, .
\end{align}
 First, we apply this with $V = \nabla |u|$ and then use \eqr{ee:herelambda} and $|\nabla |u|| \leq |\nabla u|$ to get
\begin{align}		\label{ee:combiii}
	0 =     \int_{\Omega} \left( |\nabla u|^2 + \langle u , \cL \, u \rangle \right) \,  \e^{-f}
	\geq    \int_{\Omega} \left( |\nabla |u||^2 - \lambda \, |u|^2 \right) \,  \e^{-f}
	\, .	 
\end{align}
 For the second application of \eqr{ee:nowdom}, take $V = |u| \, \nabla f$ and use $\cL \, f = \frac{n}{2} - f$ to get
 \begin{align}
 	0 =  \int_{\Omega} \left\{ 2\,  \langle |u| \, \nabla |u| , \nabla f \rangle 
	+|u|^2 \, \cL \, f
	  \right\} \,  \e^{-f} =   \int_{\Omega} \left\{ 2\,  \langle |u| \, \nabla |u| , \nabla f \rangle 
	+|u|^2 \, \left(  \frac{n}{2} - f \right)	  \right\} \,  \e^{-f} \, . \notag
 \end{align}
Since $|\nabla f|^2 \leq f$, the absorbing inequality
$2\, \left| \langle |u| \, \nabla |u| , \nabla f \rangle \right| \leq 2 \, |\nabla |u||^2 + \frac{1}{2} \, |u|^2 \, |\nabla f|^2$ gives
 \begin{align}
 	\int_{\Omega} |u|^2  \,  \left(  f - n \right)  \, \e^{-f} \leq 
	4\, \int_{\Omega}  |\nabla |u||^2   \, \e^{-f} \leq 4\, \lambda \, \int_{\Omega}  |u|^2   \, \e^{-f}\, ,
 \end{align}
 where the last inequality is  \eqr{ee:combiii}.
 Since $|u| > 0$ and $f = \frac{b^2}{4} > \frac{r^2}{4}$ on $\Omega$, we get that
 	 	$\left(  \frac{r^2}{4} - n - 4 \, \lambda \right)    \leq 0$.
 This is the desired contradiction since $r > 2 \, \sqrt{n+ 4\, \lambda}$.
 \end{proof}

 \subsection{Growth estimates}	\label{es:3}
 
 Assume now that $u$ satisfies \eqr{ee:herelambda}.  We will use that, by Lemma \ref{el:Idiff}, $\log I$, $\log D$, and $\log U$ 
 are absolutely continuous as long as $I, D > 0$.  One challenge for controlling the growth of $D$ and $I$ is that $D'$ and $I'$ 
 have  terms involving $S$, with the wrong sign in one case and  a variable sign in the other.  The terms will be played off each other and we will be able to control  the right combination; this miraculous cancelation makes it work.

\begin{Pro}	\label{ep:driftUmont}
If $r$ is a regular value of $b$ and $D(r) , I(r) > 0$, then 
\begin{align}
 	r\,(\log D)' &\geq 2-n  + \frac{r^2}{2}  + U -\frac{\lambda\,r^2}{U} 
	 -\frac{4\,\lambda}{U\,I} \, r^{1-n}\, \int_{b = r}
	\frac{S\,|u|^2}{|\nabla b|} 
	\, , \label{ee:thederivofD2}\\
	r\, \left(\log U\right)' (r)&\geq    2-n+\frac{r^2}{2}-U-\frac{\lambda\,r^2}{U} +
	\left(1-\frac{2\,\lambda}{U}- \frac{2\,n}{r^2}\right) \, \frac{2\,r^{1-n}}{I}\int_{b=r}\frac{S\,|u|^2}{|\nabla b|}
\, .\label{ee:logU}
\end{align}
\end{Pro}
 
 \begin{proof}
 Lemma \ref{el:Idiff} and \eqr{ee:herelambda}  give
 \begin{align}
	D'(r) &= \frac{2-n}{r} \, D + \frac{r}{2} \, D + \frac{r^{2-n}}{2} \, \int_{b = r}
	\frac{  \cL \, |u|^2 }{|\nabla b|}  \geq 
	 \frac{2-n}{r} \, D + \frac{r}{2} \, D + r^{2-n} \, \int_{b = r}
	\frac{ \left( |\nabla u|^2 - \lambda \, |u|^2	\right)}{|\nabla b|} 
	\, . \notag
 \end{align}
  Since $4\, S= b^2 - b^2 \, |\nabla b|^2$, we get that
    \begin{align}
	r\,  (\log D)'(r) &\geq 2-n +\frac{r^2}{2}-\frac{\lambda\,r^2}{U}+\frac{r^{3-n}}{D}\,\int_{b=r}\frac{|\nabla u|^2}{|\nabla b |}-\frac{4\,\lambda\,r^{1-n}}{D}\int_{b=r} \frac{S\,|u|^2}{|\nabla b|}\, .  \label{ee:thirdclaim}
 \end{align}
 Note that by the Cauchy-Schwarz inequality
 \begin{align}
 D^2(r)&=\left(\frac{r^{2-n}}{2} \, \int_{b =r} \langle \nabla |u|^2 , \frac{\nabla b}{|\nabla b|} \rangle\right)^2\leq    I(r) \, r^{3-n} \, \int_{b = r} \frac{| \nabla u|^2}{ {|\nabla b|}} \, .
 \end{align} 
Dividing this by $I(r)$ gives $U\,D\leq r^{3-n}\int_{b=r}\frac{|\nabla u|^2}{|\nabla b|}$.
Using this in \eqr{ee:thirdclaim} gives \eqr{ee:thederivofD2}.
Combining  \eqr{ee:thederivofD2}  and \eqr{ee:secondclaim} gives \eqr{ee:logU}.  
 \end{proof}
 
 An immediate consequence of the proposition is the following:

\begin{Cor}	\label{ec:driftUmont}
If $r$ is a regular value with $U(r)>2\,\lambda$ and $r>\sqrt{\frac{2\,n}{1-\frac{2\,\lambda}{U}}}$, then
\begin{align}
	\left(\log U\right)' &\geq    \frac{2-n-U}{r} +    r\,\left(\frac{1}{2}- \frac{\lambda}{U}\right)
\, .
\end{align}
\end{Cor}

We use this to show that if $U$ goes strictly above $2\, \lambda$, then it grows quadratically; this does not assume that $u \in L^2$ and, indeed, it is impossible when $u , \cL \, u \in L^2$.
 
 \begin{Thm}	\label{ec:Uquad}
 Given $\delta > 0$, there exists $R > \sqrt{2n}$ so that if 
  $U (r_0) > (2+\delta)\,\lambda$ for some $r_0 \geq R$, then $ U(r)\geq \frac{1}{2}\,r^2-r$ for every $r$ sufficiently large.
 \end{Thm}
 
 \begin{proof}
 If $U(r)> (2+\delta)\,\lambda$ for a regular value $r > \sqrt{ (4n\, (2+\delta)/\delta)}$, then Corollary \ref{ec:driftUmont} gives
 \begin{align}
 (\log U)' (r)\geq \frac{2-n-U}{r}+\frac{\delta\,r}{2\,(2+\delta)}\, .
 \end{align}
 It follows that if $ (2+\delta)\,\lambda < U < \frac{\delta\,r^2}{5\,(2+\delta)}$ and $r > \sqrt{ (4n\, (2+\delta)/\delta)}$, then
\begin{align}
	 (\log U)' (r)> \frac{\delta\,r}{2\,(2+\delta)} - \frac{n}{r} - \frac{\delta\,r}{5\,(2+\delta)} > \frac{\delta \, r}{2+\delta} 
	 \, \left( \frac{1}{2} - \frac{1}{4} - \frac{1}{5}
	 \right) = 
	 \frac{\delta\,r}{20\,(2+\delta)} \, .
\end{align}
 This implies $U$ is increasing on this interval and that there exists an $R>0$ and $c>0$ such that  $ U(r)\geq c\,r^2$
 for $r>R$.
Thus,  by Corollary \ref{ec:driftUmont}, if $ \frac{r^2}{2}-r>U$ for $r > R$, then
 \begin{align}
 (\log U)'\geq \frac{2-n}{r}+1-\frac{\lambda}{c\,r}\,  .
 \end{align}
This forces $U$ to grow exponentially to the top of this range, eventually giving the claim.   
 \end{proof}

\begin{proof}
(of Theorem \ref{et:main}).   Since $(\log I)' \leq 2 \, U/r$ for $r > \sqrt{2n}$ by Lemma \ref{el:Idiff}, 
the growth bound \eqr{ee:IU08} will follow from the  bound \eqr{ee:0p9} on $U$.
 We first show  for any $\delta > 0$ that
\begin{align}
U(r) \leq 2\,\lambda+\delta\label{ee:weakerbound}
\end{align}
for all $r$ sufficiently large. We will argue by contradiction, 
so  suppose that  \eqr{ee:weakerbound} fails for some $r$ sufficiently large.  
Theorem \ref{ec:Uquad} gives that $U \geq \frac{r^2}{2} - r  $ for all sufficiently large $r$.  It follows that $K(r) = D(r) - 4\, \lambda \, I(r)$ is positive for all large $r$.  At a regular value $r > 2 \, \sqrt{n}$, Proposition \ref{ep:driftUmont} and Lemma  \ref{el:Idiff} give
\begin{align}	\label{ee:e313}
	r\, K'   &\geq \left( 2-n  + \frac{r^2}{2}  + U   - 8 \, \lambda \right) \, D - \lambda \, r^2 \, I
	 + \left[ 8\, (1 - \frac{2\,n}{r^2}) - 4 \right] \, \lambda \, r^{1-n}  \int_{b = r} \frac{S\,|u|^2}{ |\nabla b|} \notag \\
	&\geq  \left( 2-n  + r^2 - r  - 8 \, \lambda \right) \, D - \lambda \, r^2 \, I \\
	&\geq  \left( 2-n  + r^2 - r  - 8 \, \lambda \right) \, K +  4\, \lambda \, \left( 2-n  + \frac{3\,r^2}{4} - r  - 8 \, \lambda \right) \, I
	   \, . \notag 
\end{align}
 Thus, for $r$ large, we have $(\log K)' \geq \frac{3}{4} \, r$.  Integrating this gives 
 for $t>s>R$ 
\begin{align}	\label{ee:inttocont}
	D(t)\geq K(t) \geq K(s) \,  \e^{ \frac{ 3\, (t^2-s^2)}{8}} \, .
\end{align}
This implies that 
\begin{align}
2\, \int_{b\leq t}(|\nabla u|^2 + \langle \cL \, u , u \rangle)\,\e^{-\frac{|x|^2}{4}} = \int_{b\leq t} \cL\, |u|^2
\,\e^{-\frac{|x|^2}{4}}&=2\, \e^{-\frac{t^2}{4}}\,t^{n-2}\,D(t) \to \infty  {\text{ as }} t \to \infty \, . \notag
\end{align}
This is a contradiction since $\cL \, u \in L^2$ and $u \in W^{1,2}$ by Lemma \ref{el:W12},  so   \eqr{ee:weakerbound} holds.

We turn to the sharper bound \eqr{ee:0p9}; we can assume that $\lambda > 0$ since otherwise $u$ is parallel since $u, \cL \, u \in L^2$.
The proof is by contradiction, so suppose that 
  $r \geq R$ satisfies
\begin{align}
2\,\lambda+\delta\geq U(r)\geq 2\,\lambda\,\left(1+\frac{\mu}{r^2}\right)\, , \label{ee:therange}
\end{align}
where $\mu \in \RR$ will be chosen below.
 At any $r$ satisfying \eqr{ee:therange}, we have 
\begin{align}
\frac{r^2}{2}-\frac{\lambda\,r^2}{U}=\frac{r^2}{2}\,\left(1-\frac{2\,\lambda}{U}\right)\geq \frac{\lambda\,\mu}{U} \geq 
 \frac{\lambda\,\mu}{2\,\lambda+\delta} \,  .
\end{align}
Together with \eqr{ee:logU}, this gives at regular values that 
\begin{align}   \label{ee:improvedlogU}
	r\, \left(\log U\right)' (r)
&\geq 2-n+\frac{\lambda \, \mu}{2\, \lambda + \delta} -2\,\lambda-\delta+
	\left( \frac{\lambda\,\mu}{2\,\lambda+\delta} -n\right)\, \frac{4\,r^{-1-n}}{I}\int_{b=r}\frac{S\,|u|^2}{|\nabla b|}\, .
\end{align}
Assuming that $\mu \geq\left( 2 + \frac{\delta}{\lambda} \right)\, n$ so the last term is nonnegative, we have
\begin{align}
r\, \left(\log U\right)' (r)&\geq \frac{ \lambda \, \mu - (2\, \lambda + \delta)^2 - (n-2) (2\, \lambda  + \delta)}{2\, \lambda + \delta} \, .
\end{align}
If   $\mu > 4 \, \lambda +2n - 4$, then this is strictly positive for $\delta > 0$ sufficiently small, forcing $U$ to 
  grow out of the range \eqr{ee:therange}, giving the desired contradiction if $\mu = \bar{\mu} + \epsilon$ (note that $\lambda > 0$ is fixed and $\delta > 0$ can be taken arbitrarily small)..
\end{proof}

\subsubsection{Examples}

We will next 
  consider  examples which show that Theorem \ref{et:main} is surprisingly sharp.  Not only is the threshold $2\lambda$ sharp, but even the next order term is sharp.  
If $u= b^2 - 2n$, then  $\cL \, u = - u$, so that $\lambda = 1$, and \eqr{ee:firstclaim} gives
  \begin{align}
   	D(r) &= \frac{r^{2-n}}{2} \, \int_{b = r} \langle \nabla (b^2 -2n)^2 , \frac{\nabla b}{|\nabla b|} \rangle  =
	2\, r^{3-n} \, (r^2 - 2\,n) \,  \int_{b = r}     |\nabla b|  = \frac{2\, r^2 \, I(r)}{r^2-2n}  \, . 
\end{align}
Therefore, we see that the frequency $U= \frac{D}{I}$ satisfies
\begin{align}
	U(r) &=   \frac{2\, r^2}{r^2-2n} = 2\,  \left(1 + \frac{2\,n}{r^2} + O(r^{-4}) \right) = 2\, \lambda \, \left(1 + \frac{ 4\, \lambda + 2\, n -4}{r^2} + O(r^{-4}) \right) \, .
\end{align}
Next, let $M= \RR$, 
   $f = \frac{x^2}{4}$, and $\cL$ be the Ornstein-Uhlenbeck operator.
The degree $m$ Hermite polynomial has $\lambda = \frac{m}{2}$ and is given by
	$x^m - m \, (m-1) \, x^{m-2} + O(x^{m-4})$, so that
 \begin{align}
 	I(r) = 2 \, \left( r^{2\,m} - 2 \, m \, (m-1) \, r^{2\,(m-1)} + O(r^{2\,(m-2)}) \right) \, .
 \end{align}
 It follows that
 \begin{align}
 	2\, U(r) = \frac{r\, I'}{I} = 2\,m \, \frac{ r^{2\,m} - 2 \, (m-1)^2 \, r^{ 2\,(m-1)} + O(r^{2\,(m-2)})}{r^{2\,m} - 2 \, m \, (m-1) \, r^{2\,(m-1)} + O(r^{2\,(m-2)})} \, .
 \end{align}
 Thus, we have
	$U(r) = m \, \left( 1 + 2\, (m-1) \, r^{-2}
	+ O(r^{-4}) \right) = 2\, \lambda \,   \left( 1 +  (4\, \lambda - 2) \, r^{-2}
	+ O(r^{-4}) \right) $.

\subsection{Poisson equation}

Suppose that $u$ satisfies $\langle \cL \, u , u \rangle  \geq - \lambda \, |u|^2 - \psi$,  
where $\lambda \geq 0$ is a constant and $\psi \geq 0$ is a function.    By
  Lemma \ref{el:coareaapp},  $J$ from  \eqr{ee:J} is absolutely continuous and $J'$ is given a.e. by
 \begin{align}   \label{ee:J'}
 	J' = r^{2-n} \, \int_{b=r} \frac{\psi}{|\nabla b|} \, .
 \end{align}
  We will use the following immediate analog of  Proposition \ref{ep:driftUmont} (with the additional term  in $D'$ (cf. \eqr{ee:thirdclaim}), resulting in  $J'$ terms in \eqr{ee:thederivofD2A}, \eqr{ee:logUA}).

   \begin{Lem}	\label{el:2G}
   If $r$ is a regular value of $b$ and $D(r) , I(r) > 0$, then 
  \begin{align}
	r\,(\log D)' &\geq 2-n  + \frac{r^2}{2}  + U -\frac{\lambda\,r^2}{U} 
	 -\frac{4\,\lambda}{U\,I} \, r^{1-n}\, \int_{b = r}
	\frac{S\,|u|^2}{|\nabla b|} - \frac{r}{D} \, J'
	\, , \label{ee:thederivofD2A}\\
	r\, \left(\log U\right)' &\geq    2-n+\frac{r^2}{2}-U-\frac{\lambda\,r^2}{U} +
	\left(1-\frac{2\,\lambda}{U}- \frac{2\,n}{r^2}\right) \, \frac{2\,r^{1-n}}{I}\int_{b=r}\frac{S\,|u|^2}{|\nabla b|}  - \frac{r}{D} \,J'
\, .\label{ee:logUA}
 \end{align}
   \end{Lem}

 \begin{Lem}  \label{el:Knu}
 Given $\delta \in (0,2)$, set $K =D-(2\,\lambda + \delta/2) \,I$.  There exists $r_0(\lambda , \delta , n)$, so that if $r \geq r_0$ is a regular value with $K(r) > 0$, then
  \begin{align}
 r\,K '\geq \frac{2\, \lambda \,r^2}{4\, \lambda+\delta}\,K+\left[U+2-n+\frac{\delta\,r^2}{2\,(4\lambda + \delta)}-(4\, \lambda +\delta) \right]\,D-r\,J'\,  .
 \end{align}
 \end{Lem}
 
 \begin{proof}
 By \eqr{ee:thederivofD2A} and \eqr{ee:secondclaim}, we have
  \begin{align}
  	r\, D' &\geq \left( 2-n  + \frac{r^2}{2}  + U\right) \, D - \lambda\,r^2 \, I 
	 -4\,\lambda \, r^{1-n}\, \int_{b = r}
	\frac{S\,|u|^2}{|\nabla b|} - r\, J' \, ,\\
 r\, (2\,\lambda + \delta/2)  \, I' &= (4\, \lambda + \delta)\, D + \left( \frac{2\,n}{r^2}-1 \right)\, (4\, \lambda + \delta) \, r^{1-n}  \int_{b = r} \frac{S\,|u|^2}{ |\nabla b|}  \, .
 \end{align}
 Since $S\geq 0$ and  $\left[(4\, \lambda + \delta)(1-2\,n\,r^{-2})-4\, \lambda \right]\geq 0$ for $r\geq r_0(\lambda , \delta , n)$,  it follows that 
 \begin{align}
 r\,K'& \geq \left[\left( 2-n  + \frac{r^2}{2}  + U\right)  -(4\, \lambda + \delta)\,\right]\,D-\lambda\,r^2 \, I-r\,J'\,  .
 \end{align}
Since  $[D-2\,\lambda\,I]=\frac{4\, \lambda}{4\, \lambda +\delta}\,K+\frac{\delta}{4\, \lambda +\delta}\,D$, this gives the claim.  
  \end{proof}
  
     \begin{proof}[Proof of Theorem \ref{et:main2A}]
      Set $J_0 = \sup \, J$.
 We will show that 
 \begin{align}	\label{ee:showbycdA}
 	K(r)  \leq  10 \, J_0 {\text{ for all }} r>R(\lambda,\delta,n) \, .
 \end{align}
Once we have \eqr{ee:showbycdA}, we use   \eqr{ee:secondclaim} to get that  
\begin{align}
	r\, I' &\leq 2\,D  \leq (4 \, \lambda + \delta) \, I +  20\,J_0  \, .
\end{align}
Equivalently, 
$\left(r^{-(4 \, \lambda +  \delta)}\,I\right)'\leq 20\,r^{-(4 \, \lambda +  \delta)-1}\,J_0$.  Integrating this gives
\eqr{ee:main2A}.  

We will prove \eqr{ee:showbycdA} by contradiction, so suppose instead that $K(r_0) > 10 \, J_0$ for some large  $r_0$.  At any regular value $r$ with $K(r) > 0$, we have $D(r) > 0$,  thus, also
$I(r) > 0$ by Lemma \ref{el:IBDR} and $U(r) > 2\, \lambda+ \delta/2 > 0$.  Lemma \ref{el:Knu} then implies that if $r$ is large enough and $K > 0$, then $K' \geq - J'$.  Integrating this from $r_0$ gives that $K(r) \geq 9 \, J_0$ for all $r\geq r_0$ and, thus, also that $D, I > 0$ and 
$U > ( 2\, \lambda + \delta /2) > 0$.   In particular, \eqr{ee:logUA}
gives
\begin{align}	\label{ee:UplusJa}
	\left(\log U\right)' &\geq    \frac{2-n- U}{r} +\frac{r}{2}-\frac{\lambda\,r}{U}    - \frac{J'}{D} \geq
	 \frac{2-n- U}{r} +\frac{r}{2}-\frac{\lambda\,r}{U}    - \frac{J'}{9 \, J_0}
\,  .
\end{align}
  Suppose first $U(r) < \frac{\delta \, r^2}{4(4\, \lambda + \delta)}$ for every larger $r$, then \eqr{ee:UplusJa} would give
\begin{align}	\label{ee:UplusJb}
	\left(\log U\right)' &\geq 	 \frac{2-n}{r}-  \frac{\delta \, r}{4(4\, \lambda + \delta)} +\frac{r}{2}-\frac{2\,\lambda\,r}{4\, \lambda+ \delta}    - \frac{J'}{9 \, J_0} =  	 \frac{2-n}{r}-   \frac{J'}{9 \, J_0} +  \frac{\delta \, r}{4(4\, \lambda + \delta)}
\,  .
\end{align}
Integrating this contradicts the upper bound on $U$, so we conclude that   there is a large $r$ where $U \geq \frac{\delta \, r^2}{4(4\, \lambda + \delta)}$.   Next, at any large $r$ where 
	$\frac{\delta \, r^2}{8(4\, \lambda + \delta)}  \leq U(r) \leq \frac{r^2}{2} - r$, 
then \eqr{ee:UplusJa} gives
\begin{align}	\label{ee:UplusJc}
	\left(\log U\right)' &\geq    1 +
	 \frac{2-n}{r}  -\frac{8\, \lambda \, (4\, \lambda + \delta)}{\delta \, r}    - \frac{J'}{9 \, J_0}
\,  ,
\end{align}
 forcing $U$ to grow exponentially and, thus, eventually overtake the quadratic upper bound. Thus, we get $R_1$ large so that for all $ r \geq R_1$ we have $U > \frac{r^2}{2} -r  - \frac{1}{9}$  (the last term comes from integrating $\frac{J'}{9\, J_0}$).
   Using this lower bound for $U$ in Lemma \ref{el:Knu} gives
\begin{align}
( K+ J)' &\geq \frac{2\, \lambda \,r}{4\, \lambda +\delta}\,K+\left[ \left( \frac{r^2}{2} - r - \frac{1}{9} \right) +2-n+\frac{\delta\,r^2}{2\,(4\lambda+\delta)}-(4\, \lambda+ \delta) \right]\,\frac{K}{r} \\
&= \left(r - 1 + \frac{2-n-1/9 - (4\, \lambda + \delta)}{r} \right) \, K 
\geq \frac{8\, r}{9} \, K \geq \frac{4\, r}{5} \, (K+J) \, ,	\notag
 \end{align}
 where the last inequality used $K+J \leq K + J_0 \leq \frac{10}{9} \, K$.  Integrating gives that $K+J$ grows at least like $\e^{ \frac{2\, r^2}{5}}$.  This contradicts that $u \in W^{1,2}, \cL \, u \in L^2$ as in the proof of Theorem \ref{et:main}.
 \end{proof}

 We will also prove an effective growth bound similar in spirit to Hadamard's three circles theorem, \cite{Li,N}.  Roughly, this shows that if $u$ is very small on a scale $r_1$ and bounded at larger scale $R$, then $u$ stays small out to  scale $R-1$.
  
 \begin{Pro}	\label{ep:effective}
 Given  $ \lambda >0$ and $\delta \in (0,2\, \lambda)$, there exists $r_0$ so that if $r_0 \leq r_1 < R$, $u$ satisfies \eqr{ee:herelambda} on $\{ r_1 \leq b \leq R \}$ and  $D(R) \leq \e^{ \frac{2\, R - 1}{6}} \, I(r_1)$, then for all $r \in [r_1 , R-1]$ 
\begin{align}
	I(r) \leq \left( \frac{r}{r_1} \right)^{ 4\, \lambda +2\, \delta} \, \left[1+\frac{1}{(2\, \lambda+ \delta)}  \right]\, I(r_1) \, .
\end{align}
 \end{Pro}

 \begin{proof}
 By Lemma \ref{el:Knu} with $J = 0$,  if $r\geq r_0=r_0(\lambda,\delta,n)$ and $K(r) > 0$, then $K' \geq \frac{r}{3} \, K$ and, thus, 
 $\e^{-\frac{r^2}{6}}\,K(r)$ is monotone non-decreasing.  If    $r\in [r_1,R-1]$ with $K(r) > I(r_1)$, then   $D(r) > K(r) > 0$ and, thus, also
$I(r) > 0$ by Lemma \ref{el:IBDR}.  Moreover, 
\begin{align}
	 D(R) > K(R) \geq \e^{  \frac{R^2 - r^2}{6}} \, K(r) \geq \e^{ \frac{R^2 - r^2}{6}} \, I(r_1) \geq  \e^{ \frac{2\, R - 1}{6}} \,  I(r_1) \, .
\end{align}
This contradicts $D(R) \leq \e^{ \frac{2\, R - 1}{6}} \, I(r_1)$, so 
  $K(r) \leq I(r_1)$     for all $ r \in (r_1 , R - 1)$ and, thus,
   \begin{align}
   	D(r) = K(r) + (2\, \lambda +\delta)  \, I(r) \leq I(r_1)+ (2\, \lambda+\delta )  \, I(r) \, .
\end{align}
Combining this with the bound on $I'$ from  Lemma \ref{el:Idiff} gives 
\begin{align}
	\left( r^{-(4\, \lambda +2\, \delta) } \, I(r) \right)' \leq - (4\, \lambda +2\, \delta)\, r^{-(4\, \lambda+2\, \delta)-1} \, I + 2 \,  r^{-(4\, \lambda+2\, \delta)-1} \, D
	\leq 2\, r^{-(4\lambda +2\, \delta)-1} \, I (r_1)\, .  \notag
\end{align}
Integrating from $r_1$ to $r \leq R -1$ gives the claim.  
 \end{proof}


 \section{Growth of eigenvector fields for $\cP$} 
     
        We use the relationship between $\cP$ and $\cL$ to solve the Poisson equation $\cP \, Y = \frac{1}{2} \, \dv_f \, h$ 
        in Theorem \ref{c:elliptic} and to get strong bounds for $Y$ in Theorem  \ref{t:main0a} using also the previous section.
          The next theorem proves similar growth bounds for eigenvector fields for $\cP$ which are 
        generalizations of Killing fields.

    \begin{Thm}   \label{t:main0}
 For any shrinker $(M,g,f)$, if $Y\in L^2$,
 $\cP\,Y=\lambda\, Y$ and $Z=Y+\frac{2}{2\,\lambda+1}\,\nabla\,\dv_f\,(Y)$, then $\dv_f\,(Z)=0$ and for any $\delta>0$ and 
 $r_2>r_1>R=R(\lambda,n,\delta)$
 \begin{align}
 I_{\nabla\,\dv_f\,(Y)}(r_2)&\leq \left(\frac{r_2}{r_1}\right)^{4\,\lambda+\delta}\,I_{\nabla\,\dv_f\,(Y)}(r_1) \,  ,\label{e:ee1}\\
 I_{Z}(r_2)&\leq \left(\frac{r_2}{r_1}\right)^{8\,\lambda+2+\delta}\,I_{Z}(r_1)\, . \label{e:ee2}
 \end{align}
 \end{Thm}
 
 Each of these growth bounds is sharp and so is the requirement that $Y \in L^2$.    Combining them bounds $Y$.  
 As a corollary, $L^2$ Killing fields on a shrinker grow at most linearly.

 \begin{Cor}   \label{c:main0}
On any shrinker, for any $L^2$ Killing field $Y$, $\nabla\,\dv_f\,(Y)$ is parallel and if $Z=Y+2\,\nabla\,\dv_f\,(Y)$, then $\dv_f\,(Z)=0$ and for any $\delta>0$ and $r_2>r_1>R=R(n,\delta)$
\begin{align}
I_{Z}(r_2)&\leq \left(\frac{r_2}{r_1}\right)^{2+\delta}\,I_{Z}(r_1)\, .	\label{e:0point9}
\end{align}
 \end{Cor}

It is easy to see that this is sharp; on the two dimensional Gaussian soliton $Y=x_2\,e_1-x_1\,e_2$ is a Killing field with $\dv_f\,(Y)=0$ that grows linearly.

 \subsection{Growth bounds for $\cP$}
 
 We will need bounds for vector fields given in  terms of $\cP$, but 
 the results of the previous section are for $\cL$.  The next two results use the relation between $\cP$ and $\cL$ to bridge this gap.
 The next proposition immediately implies Theorem \ref{t:cPandcL}.
        
  \begin{Pro}  \label{p:4main0}
 On any gradient Ricci soliton if $Y$ is vector field with $\cP\,Y-\lambda\,Y=V$ (where $\lambda\ne -\kappa$) and $Z=Y+\frac{1}{\lambda+\kappa}\,\nabla\,\dv_f\,(Y)$, then 
 \begin{align}
 ( \cL + \lambda) \,\nabla\,\dv_f\,(Y) &=-\nabla\,\dv_f\,(V)\,  , \label{e:e1}\\
( \cL +2\,\lambda+\kappa)\,Z&=-2\,V-\frac{1}{\lambda+\kappa}\,\nabla\,\dv_f\,(V)\, .\label{e:eigZ}
 \end{align}
 Moreover, if $Y$, $V$, $\dv_f\,(V)\in L^2$, then $\nabla\,\dv_f\,(Y)$, $Z\in L^2$.  
 
 If $V=0$, then $\dv_f\,(Z)=0$ and if also $Y\in L^2$, then $\|Y\|^2=\|Z\|^2+(\lambda+\kappa)^{-2}\,\|\nabla\,\dv_f\,(Y)\|^2$.
     \end{Pro}

\begin{proof}
We will show the proposition when $V=0$; the general case follows similarly.  By \eqr{e:claim1} 
 \begin{align}  \label{e:firid}
 \dv_f\,(\nabla\,\dv_f\,(Y))&=\cL\,\dv_f\,(Y)=-\dv_f\,(\cP\,Y)-\kappa\,\dv_f\,(Y)=-(\lambda+\kappa)\,\dv_f\,(Y)\,  .
 \end{align}
 From this, $\dv_f\,(Z)=0$ follows.   Equation \eqr{e:e1} follows from \eqr{e:claim2}.   To see \eqr{e:eigZ} let $c=\frac{1}{\lambda+\kappa}$ so   \eqr{e:claim2} and Lemma \ref{l:dvastY} give
 \begin{align}
 \cL\,Z&=\cL\,Y-c\,\lambda\,\nabla\,\dv_f\,(Y)=-2\,\cP\,Y-\kappa\,Y-\nabla\,\dv_f\,(Y)-c\,\lambda\,\nabla\,\dv_f\,(Y)\notag\\
 &=-(2\,\lambda+\kappa)\,\left(Y+\frac{c\,\lambda+1}{2\,\lambda+\kappa}\,\nabla\,\dv_f\,(Y)\right)=-\left(2\,\lambda+\kappa\right)\,Z\,  .
 \end{align}
 By Lemma \ref{l:interpcP},   $ \dv_f\, Y \in L^2$, so    \eqr{e:claim1}  gives that $\cL \, \dv_f \, Y= \kappa \, \dv_f \, Y - \dv_f (\cP \, Y)  \in L^2$.  Lemma \ref{l:RPcL} now gives that
 $\nabla \dv_f \, Y \in L^{2}$ 
   and, thus, also $Z \in L^2$.  Since $\dv_f\,(Z)=0$ and $Z   \in L^2$,
  $Z$ is automatically orthogonal to gradients  of all $W^{1,2}$ functions and thus, in particular, to $\nabla\,\dv_f\,(Y)$.  Therefore,   Pythagoras gives the last claim.
   \end{proof}

 \begin{proof}[Proof 
 of Theorem \ref{t:main0}]
Since $Y\in L^2$,  Proposition \ref{p:4main0} gives that $\nabla\,\dv_f\,(Y)$, $Z\in L^2$.  Equations \eqr{e:ee1}, \eqr{e:ee2} now follow from  \eqr{e:e1}, \eqr{e:eigZ}, respectively, and Theorem \ref{et:main2A}.  
 \end{proof}

 \begin{proof}[Proof of Corollary \ref{c:main0}]
 Since $Y \in \cK_{\cP}$, 
  $\cP\,Y=0$, $Y \in L^2$ and,thus, $\dv_f \, (Y) \in W^{1,2}$
  by Lemma \ref{l:interpcP} and  Proposition \ref{p:4main0}.
  Since  $\cL\,\nabla\,\dv_f\,(Y)=0$  by \eqr{e:e1}, 
   $\nabla\, \dv_f\,(Y)$ is parallel.  By Proposition \ref{p:4main0}, $\dv_f \, (Z) = 0$.  The bound \eqr{e:0point9} follows from Theorem \ref{t:main0}
   (with $\lambda = 0$).
 \end{proof}

   \begin{Thm}   \label{t:main0a}
 For any shrinker, if $Y\in L^2$,
 $(\cP -\lambda)\, Y = V$ and we set  $Z=Y+\frac{2}{2\,\lambda+1}\,\nabla\,\dv_f\,(Y)$, then   for any $\beta , \delta>0$ and $r_2>r_1>R=R(\lambda,n,\delta)$
 \begin{align}
 I_{\nabla\,\dv_f\,(Y)}(r_2)&\leq \left(\frac{r_2}{r_1}\right)^{4\,(\lambda+\beta)+\delta}\, \left( I_{\nabla\,\dv_f\,(Y)}(r_1) + \frac{5\, \int b^{2-n} \, | \nabla \, \dv_f \, (V)|^2}{\beta \, (4\,(\lambda+\beta)+\delta)}
 \right) \,  ,\label{e:ee1a}\\
 I_{Z}(r_2)&\leq \left(\frac{r_2}{r_1}\right)^{8\,(\lambda + \beta) +2 +\delta}\,
 \left( I_{Z}(r_1) +  \frac{\int b^{2-n} \, \left(|V|^2 + \frac{| \nabla \, \dv_f \, (V)|^2}{(2\, \lambda +1 )^2} \right) }{ \beta \,  (8\,\lambda+2+ 8\, \beta +\delta)}
 \right)\, . \label{e:ee2a}
 \end{align}
 \end{Thm}

\begin{proof}
By Proposition \ref{p:4main0}, $(\cL + \lambda) \, \nabla \, \dv_f \, (Y) = - \nabla \, \dv_f \, (V)$, so  we get
\begin{align}
	\langle \cL \, \nabla \, \dv_f \, (Y) , \nabla \, \dv_f \, (Y) \rangle \geq  - (\lambda + \beta) \, | \nabla \, \dv_f \, (Y) |^2 - \frac{| \nabla \, \dv_f \, (V)|^2}{4\, \beta} \, .
\end{align}
Thus, Theorem  \ref{et:main2A} applies with $\psi = \frac{| \nabla \, \dv_f \, (V)|^2}{4\, \beta}$  to give \eqr{e:ee1a}. Similarly, 
Proposition  \ref{p:4main0} gives $  	( \cL +2\,\lambda+\frac{1}{2})\,Z=-2\,V-\frac{1}{\lambda+\frac{1}{2}}\,\nabla\,\dv_f\,(V)$, so we have
\begin{align}
	\langle \cL  \,Z , Z \rangle \geq - \left(2\, \lambda + \frac{1}{2} + 2\, \beta  \right) \, |Z|^2 - \frac{|V|^2}{\beta} + 
	\frac{|\nabla\,\dv_f\,(V) |^2}{\beta \, (2\lambda +1)^2} \, .   \, \, \, \, \, \, \, \, \, \, \, \, \,  \qedhere
\end{align}
\end{proof}

 \subsection{Fredholm properties for $\cP$}
 
 Throughout this subsection, we assume that  $(M,g,f)$ is a shrinker
 and $\cK_{\cP}$ is the space of $L^2$ Killing fields, i.e., the $L^2$ kernel of $\cP$.

 \begin{Thm}	\label{c:elliptic}
 There exists $C_1$ so that if
 $h$  is a smooth compactly supported symmetric $2$-tensor, then there is a smooth vector field $Y \in W^{1,2}$ with
 $\dv_f \left( \frac{1}{2} \, h - \dv_f^{*}\, Y \right) = 0$  that is $L^2$-orthogonal to $\cK_{\cP}$ and satisfies
 \begin{align}	\label{e:celle}
	\| Y \|_{W^{1,2}} + \| \dv_f \, Y \|_{W^{1,2}} + \| \cL \, Y \|_{L^2}  \leq C_1 \, \| \dv_f\, h  \|_{L^2}  \, .
\end{align}
 \end{Thm}

\begin{Lem}	\label{p:fredholmL}
 If  $Y $ is a vector field, then $\| \cL \, Y \|_{L^2}^2 \leq  (2n + 8) \, \| Y \|_{W^{2,2}}^2$.  If $Y , \cL \, Y \in L^2$,  then
 \begin{align}
 	\frac{1}{4n} \, \| Y\, \sqrt{f}  \|_{L^2}^2  \leq \| Y \|_{W^{1,2}}^2  &\leq  \| \cL\, Y \|_{L^2} + 2\, \| Y \|_{L^2} \, .
	\label{l:fred1L}
 \end{align}
 The $L^2$ kernel of $\cL$ is equal to the space $\cK_{\cL}$ of parallel vector fields and
  $\cL$ has discrete eigenvalues $0 \leq \mu_0 < \mu_1 < \mu_2,\cdots \to \infty$ with finite dimensional eigenspaces $E_{\mu_i} \subset W^{1,2}$.
\end{Lem}

  \begin{proof} 
  The first claim  follows from  the squared triangle inequality and Lemma \ref{l:L2RPa}
  \begin{align}
  	\| \cL \, Y \|_{L^2}^2 \leq 2 \, \| \Delta \, Y \|_{L^2}^2 + 2 \, \| |\nabla f| \, |\nabla Y| \|_{L^2}^2
	\leq 2 \, n \, \| \nabla^2 \, Y \|_{L^2}^2 + 2 \, n \, \| \nabla Y \|_{L^2}^2 + 8 \, \| \nabla |\nabla Y | \|_{L^2}^2 \, . \notag
  \end{align}
 Suppose now that $Y , \cL \, Y \in L^2$.  Lemma \ref{l:RPcL} gives that $Y \in W^{1,2}$ and
\begin{align}	\label{e:fredRPa} 
	\| \nabla Y \|_{L^2}^2 \leq 2 \, \| Y \|_{L^2} \, \| \cL \, Y \|_{L^2} \leq  \| Y \|_{L^2}^2 + \| \cL \, Y \|_{L^2}^2 \, .
\end{align}
The first inequality gives that the $L^2$ kernel of $\cL$ is equal to the space $\cK_{\cL}$ of parallel vector fields.  
Combining \eqr{e:fredRPa}  and  Lemma \ref{l:L2RPa}  gives \eqr{l:fred1L}.
   The estimate \eqr{l:fred1L} implies that the inverse of $\cL$ is a compact symmetric operator, so the eigenvalues of $\cL$ go to infinity and the eigenspaces are finite dimensional (cf. the appendix in \cite{CxZh2} for functions, plus
   Rellich compactness for vector fields).
   \end{proof}

 Below let $\mu$ be an eigenvalue of $\cL$ and 
$E_{\mu}=\{V\in L^{2}\,|\,\cL\,V+\mu\,V=0\}$
the corresponding eigenspace.  
 Recall that the convention is that the operators $\cL$ and $\cP$ have opposite sign.

\begin{Lem}	\label{l:old1030}
We have
 \begin{enumerate}
 \item For each $\mu$, the map $\cP$ maps $E_{\mu}$ to $ E_{\mu}$, is self-adjoint, and has a basis of eigenvectors.  
 \item If $V\in E_{\mu}$ and $\cP\,V=\lambda\,V$, then  $\mu-2\,\lambda\leq \frac{1}{2} $ with equality if and only $\dv_f\,V=0$.
 \item $\cP$ has discrete eigenvalues $\lambda_i \to \infty$ and each eigenspace is finite dimensional.
 \end{enumerate}
\end{Lem}

\begin{proof}
Suppose that $V \in E_{\mu}$.  Lemma \ref{l:RPcL} gives that $V, \dv_f \, V \in W^{1,2}$ and, thus, $\cP \, V \in L^2$ by Lemma \ref{l:dvastY}.
By Proposition \ref{p:commPL},   $\cL\,\cP\,V=\cP\,\cL\,V=-\mu\,\cP\,V$.     It follows that $\cP$ maps $E_{\mu}$ to itself.  The first claim follows from this 
together with that $\cP$ is self-adjoint.

 If   $\cP\,V=\lambda\,V$  and $\cL\,V= -\mu\,V$, then $(2\,\lambda - \mu)\,V=\cL\,V+2\,\cP\,V=-\nabla \,\dv_f\,V- \frac{1}{2} \,V$ by Lemma \ref{l:dvastY}.
Since $V, \dv_f \, V \in W^{1,2}$, 
taking the inner product with $V$ and integrating gives
\begin{align}
\left(\frac{1}{2} -\mu+2\,\lambda\right)\int |V|^2\,\e^{-f}=-\int \langle \nabla \,\dv_f\,V,V\rangle\,\e^{-f}=\int |\dv_f\,V|^2\,\e^{-f} \geq 0\, .
\end{align}
This gives (2).
The third claim follows by combining (1), (2) and Lemma \ref{p:fredholmL}.
\end{proof}

 \begin{proof}[Proof of Theorem \ref{c:elliptic}]
  If $V\in \cK_{\cP}$, then $\int \langle \dv_f \, h , V \rangle \, \e^{-f} = \int \langle h , \dv_f^{*} \, V \rangle \, \e^{-f}=0$.   Therefore, by Lemma \ref{l:old1030}, there exist $a_i \in \RR$ and $L^2$-orthonormal vector fields $V_i$ so that
  $\cP \, V_i = \lambda_i \, V_i$, $0 < \lambda_1 \leq \lambda_2 \leq \dots$, $\lambda_i \to \infty$, $\cL \, V_i = -\mu_i \, V_i$, and $\dv_f \, h = \sum_{i=1}^{\infty} a_i \, V_i$.
Note that 
\begin{align}
	\| \dv_f \, h \|_{L^2}^2 = \sum_i a_i^2 < \infty \, .
\end{align}
Set $Y = \frac{1}{2} \, \sum_{i=1}^{\infty} \frac{a_i}{\lambda_i} \, V_i$, so  $\cP \, Y = \frac{1}{2} \, \dv_f \, h$ weakly and 
\begin{align}
	\| Y \|_{L^2}^2 = \frac{1}{4} \,  \sum_{i=1}^{\infty} \frac{a_i^2}{\lambda_i^2} \leq   \frac{\| \dv_f \, h\|_{L^2}^2}{4\, \lambda_1^2} \, .
\end{align}
Lemma \ref{l:interpcP} then also gives $L^2$ bounds on $\dv_f \, Y$ and $\nabla Y$.  To get the $L^2$ bounds on $\cL \, Y$ (and, thus, also $\nabla \dv_f \, Y$), observe
that (2) in Lemma \ref{l:old1030} gives  $0 \leq \mu_i \leq \frac{1}{2} + 2 \, \lambda_i$, so that
\begin{align}
	\left( \frac{\mu_i}{\lambda_i} \right)^2 \leq   \frac{\left(\frac{1}{2} + 2 \, \lambda_i \right)^2}{\lambda_i^2}  \leq  \frac{ \frac{1}{2} + 8 \, \lambda_i^2}{\lambda_i^2} \leq \frac{1}{2\, \lambda_1^2} + 8 \, .
\end{align}
Since $\cL \, Y = -\frac{1}{2} \, \sum_{i=1}^{\infty} \frac{ a_i \, \mu_i}{\lambda_i} \, V_i$, we see that
	$\| \cL \, Y \|_{L^2}^2 \leq   \left( \frac{1}{8\, \lambda_1^2} + 2 \right)  \, \| \dv_f  \, h \|_{L^2}^2$.  Finally, since $\cP \, Y =
	\frac{1}{2} \, \dv_f \, h$ weakly and $Y , \dv_f \, (Y) \in L^2$, 
	Lemma \ref{l:elliptic} gives that $Y$ is smooth.
 \end{proof}

   \subsection{Inverting the mapping $\cP$}

Let $\cK^{\perp}$ be the $L^2$ orthogonal complement of the space of Killing fields $\cK$.  Given $R>1$, let 
$C^{q,\alpha}_R$ denote $C^{q,\alpha}$  on the set $b< R$ and 
$C^{q,\alpha}_{R,0} \subset C^{q,\alpha}_R$ be the subset with support in $b \leq R$.

We have already   constructed an inverse $\cP^{-1}$ on $\cK^{\perp}$   that  maps $L^2$ to $W^{2,2}$.
These   estimates are in the weighted spaces, so they are  strong in the central region (where the weight is large), but give almost nothing 
far out.  
  The next proposition  shows that $\cP^{-1}$ has polynomially growing estimates.  The proof will be  use
   the $L^2$ estimates on a fixed scale together with   polynomial growth bounds.

\begin{Pro}	\label{p:cPi2A}
Given $q \geq 2$, there exist $C, m$ and a linear map 
 $\cP^{-1}:\cK^{\perp} \cap C^{q,\alpha}_{R,0} \to \cK^{\perp} \cap C^{q+2,\alpha}$ with  
 \begin{align}
	\| \cP^{-1} (Y) \|_{C^{q+2,\alpha}_{2R}} &\leq C \, R^m \,  \| Y \|_{C^{q,\alpha}}  \, , \\
	\int_{b > R-1} |\cP^{-1} (Y)| \, \e^{-f} &\leq C \, R^m \,  \| Y \|_{C^{2,\alpha}} \, \e^{ - \frac{(R-1)^2}{4}} \, .
\end{align}
\end{Pro}

\begin{proof}
Since $Y \in \cK^{\perp}$,  Lemma \ref{l:old1030} gives  $a_i \in \RR$ and $L^2$-orthonormal vector fields $V_i$ so that
  $\cP \, V_i = \lambda_i \, V_i$, $0 < \lambda_1 \leq \lambda_2 \leq \dots$, $\lambda_i \to \infty$, $\cL \, V_i = -\mu_i \, V_i$, and $Y = \sum_{i=1}^{\infty} a_i \, V_i$
with
\begin{align}
	\|  Y \|_{L^2}^2 = \sum_i a_i^2 < \infty \, .
\end{align}
Set $V =  \sum_{i=1}^{\infty} \frac{a_i}{\lambda_i} \, V_i$, so  $\cP \, V=  Y$ weakly and 
\begin{align}
	\| V \|_{L^2}^2 = \frac{1}{4} \,  \sum_{i=1}^{\infty} \frac{a_i^2}{\lambda_i^2} \leq   \frac{\| Y\|_{L^2}^2}{4\, \lambda_1^2} \, .
\end{align}
Lemma \ref{l:interpcP} then also gives $L^2$ bounds on $\dv_f \, V$ and $\nabla V$.  To get the $L^2$ bounds on $\cL \, V$ (and, thus, also $\nabla \dv_f \, V$), observe
that (2) in Lemma \ref{l:old1030} gives  $0 \leq \mu_i \leq \frac{1}{2} + 2 \, \lambda_i$, so that
\begin{align}
	\left( \frac{\mu_i}{\lambda_i} \right)^2 \leq   \frac{\left(\frac{1}{2} + 2 \, \lambda_i \right)^2}{\lambda_i^2}  \leq  \frac{ \frac{1}{2} + 8 \, \lambda_i^2}{\lambda_i^2} \leq \frac{1}{2\, \lambda_1^2} + 8 \, .
\end{align}
Since $\cL \, V = -\frac{1}{2} \, \sum_{i=1}^{\infty} \frac{ a_i \, \mu_i}{\lambda_i} \, V_i$, we see that
	$\| \cL \, V \|_{L^2}^2 \leq   \left( \frac{1}{8\, \lambda_1^2} + 2 \right)  \, \| Y \|_{L^2}^2$.   
	
	Given a point $x$, let $B^x$ be the ball centered at $x$ of radius $r= r_x = (1+b(x))^{-1}$.   Define scale-invariant norms by
	$\| V \|_{C^k, B^x} =  \sum_{i=0}^k \,   r^i \, \sup_{B^x} |\nabla^i V|$ and similarly for $\| V \|_{C^{k,\alpha},B^x}$.
	The operator $\cP$ is elliptic with uniform estimates on the scale $r$, so 
linear elliptic theory  on $B^x$ gives that
\begin{align}	\label{e:scaleholder}
	 \| V \|_{C^{q+2,\alpha}, \frac{1}{2} \, B^x}
	  \leq C \, \left\{   \| Y \|_{C^{q,\alpha},B^x} +  
	 \left( 
	 r^{-n} \, \int_{   B^x} |V|^2
	 \right)^{ \frac{1}{2} }
	\right\} \, .  
\end{align}
 Note that the $L^2$ norm above is the unweighted one.  In combination with the $L^2$ estimates, this gives the desired bound on unit scale (where the exponential weight has a lower bound).  The polynomial growth bounds from Theorem \ref{t:main0a} then give the desired unweighted $L^2$ bounds on the larger scale (where the weight might be very small).{\footnote{It is here where we use that $q \geq 2$, so that we have bounds on $\nabla \dv_f Y$ (see the right-hand side in \eqr{e:ee2a}).}}  This gives a polynomial bound for the second term on the right in \eqr{e:scaleholder} and the first claim follows.  The second claim follows since we have   polynomial growth for the $L^2$ norm on level sets.
 \end{proof}

\section{Jacobi fields and a spectral gap}		\label{s:jacobi}

We need to understand  Jacobi fields on a gradient shrinking soliton $\Sigma$.  A   symmetric $2$-tensor  $h$
  gives a  variation of the metric. 
 Given  $h$ with $\dv_f\, h =0$ and a function $k$, $\phi' = \frac{1}{2} \, L \, h + \Hess_{ \frac{1}{2} \, \Tr \, h - k}$ by \cite{CaHI}.   We will say that $h$ is a Jacobi field if $L \, h =0$ and $\dv_f \, h =0$.  We omit
 $\Hess_{ \frac{1}{2} \, \Tr \, h - k}$ since this will be $L^2$-orthogonal to $L \, h$ when $\dv_f \, h = 0$.

In this section, we will assume that $\Sigma$ splits as a product
\begin{align}	\label{e:assumeM}
	\Sigma = (N^{\ell}, g^1) \times \RR^{n-\ell} \, , 
\end{align}
 where $N$ is Einstein with $\Ric_N = \frac{1}{2} \, g^1$ and $f = \frac{|x|^2}{4}+ \frac{\ell}{2}$  with $x \in \RR^{n-\ell}$.    
 
 We need to understand the spectrum of $\cL$ on $\Sigma$.  The $L^2$ eigenfunctions
  of $\cL$ on $\RR^{n-\ell}$ are polynomials with eigenvalues at $\{ 0 , \frac{1}{2} , 1 , ... \}$.   Let $\cK$ be the $L^2$ kernel  of $\cL+1$ on $\RR^{n-\ell}$.  
Each $v \in \cK$   can be written   $v = a_{ij} x_i x_j - 2\, \Tr \, a$ for a matrix $a_{ij}$ (see, e.g., lemma $3.26$ in \cite{CM2}).    
    The Lichnerowicz theorem says that   $\lambda_1 (N) \geq \frac{\ell}{2(\ell -1)} > \frac{1}{2}$.  It follows 
    that
    \begin{itemize}
    \item If $\cL \, w = -w$ and $w \in L^2$ on $\Sigma$, then $w = \zeta + v$ where $v \in \cK$ is a quadratic polynomial and $\zeta$ is a $1$-eigenfunction
    on $N$.
    \end{itemize}

 There is a natural orthogonal decomposition  of symmetric $2$-tensors 
  \begin{align}	\label{e:decomph012}
  	h = u \, g^1 + h_0 + h_2 \, .
  \end{align}
  Here $u$ is a function on $\Sigma$, $h_0$ is the trace-free part of the projection of $h$ to $N$, and the remainder $h_2$  satisfies $h_2(V,W)=0$ when $V$ and $W$ are both tangent to $N$.    We will see that $L$ preserves this decomposition.  Since $g^1$ is parallel and $\R (g^1) = \Ric_N = \frac{1}{2} \, g^1$, we have
		$L (u\,g^1) = (u + \cL\, u) \, g^1$.
Since $\R$ is zero if any of the indices is Euclidean, we see that 
\begin{align}
	L \,h_2 &= \cL\, h_2 {\text{ and }}  (\cL \, h_2)(V,W) = 0 {\text{ if $V, W$ are both tangent to $N$.}} 
\end{align}
Using that $\cL\, g^1=0$ and $\langle \R (h_0) , g^1 \rangle = \langle h_0 , \R (g^1) \rangle = 0$ gives
\begin{align}
	\langle L \, h_0 , g^1 \rangle = 0 {\text{ and }} (L\, h_0)(V , \cdot ) = 0 {\text{ if $V$ is Euclidean.}}
\end{align}
Thus, using that $L$ preserves this orthogonal decomposition of $h$, we get
\begin{align}
	\left| L \, h \right|^2 &= \left| L \, h_0 \right|^2 + \ell \, (u+ \cL\, u)^2 + \left| \cL\, h_2 \right|^2 \, .  	\label{e:para0}
\end{align}

\vskip1mm
The strong rigidity  will hold when
$\Sigma$ satisfies the condition:
\begin{itemize}
\item[($\star$)] There exists $C_{N,n}$ so that if  $h_0 , L \, h_0 \in L^{2} (\Sigma)$, then  
\begin{align}
	\| h_0 \|_{W^{1,2}}^2 \leq C_{N,n} \, \| L \, h_0 \|_{L^2}^2 \, .	\label{e:estar}
\end{align}  
\end{itemize}

  If $v \in \cK$, then $v\, g^1$ is a Jacobi field.
  Conversely,
  if $N$ satisfies ($\star$),  $h = u \, g^1 + h_0 + h_2$ with $L \, h = 0$, $\dv_f\, h = 0$ and $h \in L^2$, then
  Theorem \ref{t:approx} gives that $h_0 = h_2 = 0$ and $u$ is in $\cK$.

 We will see next that ($\star$) holds for the sphere;   it also holds for the families of symmetric spaces in \cite{CaH} that are   linearly, but not neutrally, stable for 
  Perelman's $\nu$-entropy.

\begin{Lem}	\label{l:sphere}
If $N = \SS^{\ell}_{\sqrt{2\ell}}$, or is  any quotient of $ \SS^{\ell}_{\sqrt{2\ell}}$, then $N$ satisfies ($\star$) with $C_{N,n}=(\ell-1)^2$.
In fact, if $N$ has positive sectional curvature and both $\Delta + 2\, \R$ and $\Delta + 2 \, \R - \frac{1}{2}$ are injective on the space of trace-free symmetric $2$-tensors on $N$, then \eqr{e:estar} holds.
\end{Lem}

\begin{proof}
We have
 $\R_{ikjn}^N = \frac{1}{2(\ell -1)} \,  \left( g^1_{ij}\, g^1_{kn} - g^1_{in}\, g^1_{kj} \right)$ and, thus, 
 $\R (h_0)_{ij} =    - \frac{1}{2(\ell -1)} \, (h_0)_{ij}$.
Since  $L = \cL + 2 \, \R$ and $2\, \R (h_0)  =    - \frac{1}{\ell -1} \, h_0$, we get that
\begin{align}
	\|\nabla h_0\|_{L^2}^2  &= - \int \langle h_0 , \cL\, h_0 \rangle \, \e^{-f} =  - \frac{1}{(\ell -1)} \int |h_0|^2 \, \e^{-f}  +
	 \int \langle h_0 , \left(\frac{1}{(\ell -1)} -  \cL  \right)h_0 \rangle \, \e^{-f}  \notag \\
	&\leq - \frac{1}{2(\ell -1)} \| h_0 \|_{L^2}^2   + \frac{\ell - 1}{2}  \, \| L \, h_0 \|_{L^2}^2  \, ,  
\end{align}
where the last inequality used the absorbing inequality $a\,b\leq  \frac{a^2}{2\,(\ell -1)}  +  \frac{\ell -1}{2}\, b^2$.  This gives \eqr{e:estar}.  

We turn to the second claim.   Since $L$ preserves the decomposition and $\R$ is bounded, it has a spectral decomposition and it suffices
to show that if $L \, h_0 = 0$ and $\| h_0 \|_{L^2} < \infty$, then $h_0$ vanishes.
 Since $K_N > 0$, proposition $4.9$ in \cite{BK} gives that the largest eigenvalue of $\R$ acting on $h_0$ (at each point) is at most
 $\frac{1}{2} - \ell \, \min K_N  < \frac{1}{2}$.
Thus, $2\, \R - 1$ is a negative operator on the trace-free symmetric $2$-tensors and, thus, $L - 1$ has trivial $L^2$ kernel.
Let $\partial_i$ and $\partial_j$ be  $\RR^{n-\ell}$ derivatives.  Since $L \, h_0 = 0$, we have
\begin{align}
	0 = \nabla_{\partial_i} \, (L \, h_0) = L \, (\nabla_{\partial_i}h_0) - \frac{1}{2} \,  (\nabla_{\partial_i} h_0)  
	{\text{ and }}    (L - 1) \, (\nabla_{\partial_j} \, \nabla_{\partial_i}h_0) = 0 \, .
\end{align}
Consequently,   $\nabla_{\partial_j} \, \nabla_{\partial_i}h_0 = 0$ and, thus, 
	$h_0 = h_0^0 + \sum a_i \, x_i \, h_0^i$, where
  $h_0^0$ and the $h_0^i$'s are symmetric $2$-tensors on $N$.
It follows that $(\Delta + 2\, \R ) (h_0^0) = 0$ and   $(\Delta + 2\, \R  - \frac{1}{2}) (h_0^i) = 0$
\end{proof}

We will use the next Poincar\'e inequality for vector fields  tangent to $N$:

\begin{Lem}	\label{l:poinS}
There is a constant $C=C(N)$ so that if $V$ is tangent to $N$ and $\| V \|_{W^{1,2}} < \infty$, then $\| V (1+|\nabla f|) \|_{L^2} \leq C \, \| \nabla V \|_{L^2}$.
\end{Lem}

\begin{proof}
Let $\lambda$ be the smallest eigenvalue of $\Delta$ acting on vector fields on $N$.  
Since $\Ric_N>0$, $N$ has no nontrivial harmonic one forms and thus no parallel vector fields, we have therefore $\lambda > 0$.
Given $x \in \RR^{n-\ell}$, let $V_x$ be the restriction of $V$ to $N_x=N\times \{x\}$.  We get that
\begin{align}	\label{e:poina}
	\int |V|^2 \, \e^{-f} = \int_{\RR^{n-\ell}} \int_{N_x} |V_x|^2 \, \e^{-f} \leq \frac{1}{\lambda} \, \int_{\RR^{n-\ell}} \int_{N_x}  |\nabla_{N_x} V_x|^2 \, \e^{-f} \leq 
	\frac{1}{\lambda} \, \int   |\nabla  V|^2 \, \e^{-f}  \, .
\end{align}
 The lemma follows from this and Lemma \ref{l:L2RPa}.
\end{proof}

We see that $h$ is well-approximated by a Jacobi field when $L \, h$  and $\dv_f \, h$ are  small.
   
\begin{Thm}	\label{t:approx}
There exists $C$ so that 
if $ h,\dv_f \, h  \in W^{2,2}$ and $v \, g^1$ is the $L^2$ projection of $h$ to $\cK \, g^1$, then 
   \begin{align}
   	 \| h - v g^1 \|_{W^{2,2}} \leq C \, \left\{ \| L \, h \|_{L^2}  + \| \dv_f \, h \|_{L^2} \right\}   \, .
   \end{align}
   \end{Thm}

\begin{proof}
The tensor $h$ has a (pointwise) orthogonal decomposition $h = u\,g^1 + h_0 + h_2$.      We will bound  $h_0$, $h_2$ and $(u-v)$ in  $W^{1,2}$ to control  $\| h-v\,g^1 \|_{W^{1,2}}$  in terms of $\| L\,h \|_{L^2} + \| \dv_f \, h \|_{L^2}$.  Since $\cL\, (h-v\,g^1) $ can be bounded by these and the   curvature of the cylinder and $\Hess_f$  are bounded, we then also get the  desired bound on  $\| h-vg^1 \|_{W^{2,2}}$  (cf. ($3.19$) in \cite{CM2}).

Using \eqr{e:para0}, we have that
\begin{align}	\label{e:para0A}
	\ell \, \| (\cL + 1) \, u \|_{L^2}^2 + \| L \, h_0 \|_{L^2}^2 + \| \cL \, h_2 \|^2_{L^2} = \| L \, h \|_{L^2}^2 \, .
\end{align}
The bound $\| h_0 \|_{W^{1,2}}$ follows immediately from this and ($\star$).
To control $h_2$, start by noting that there is a further orthogonal decomposition $h_2 = h_2^N + h_2^{\perp}$ that is also preserved by $\cL$, where $h_2^{\perp}$ is the purely Euclidean part.  In a block decomposition for a frame, $h_2^N$ consists of two off-diagonal parts that are transposes of each other.
Applying the part of $h_2^N$ that maps the Euclidean factor to $N$ to each Euclidean derivative $\partial_i$ to get a vector field tangent to $N$ and applying Lemma \ref{l:poinS} gives  
\begin{align}	\label{e:h2Nbd}
	\| h_2^N \|_{L^2}^2 \leq C \, \| \nabla h_2^N \|_{L^2}^2  = C \, \int \langle - h_2^N , \cL \, h_2^N \rangle \, \e^{-f} \leq
	C \, \| h_2^N \|_{L^2} \, \| \cL \, h_2 \|_{L^2} \, ,
\end{align}	
where the last inequality used the Cauchy-Schwarz inequality and the fact that $\cL$ preserves the decomposition of $h_2$.  It follows from this 
and \eqr{e:para0A} that $\| h_2^N \|_{W^{1,2}} \leq C \, \| L \, h \|_{L^2}$.

 It remains to bound $h_2^{\perp}$ and $(u-v)$; this is where we will also need the bound on $\dv_f \, h$.  Let $a_{ij}$ be a symmetric constant $(n-\ell)$ matrix so that $\int (a - h_2^{\perp}) \, \e^{-f} = 0$.  The Poincar\'e inequality on $\Sigma$ gives that
$\| a- h_2^{\perp} \|_{W^{1,2}} \leq  C \, \| \cL \, h_2^{\perp} \| \leq C \, \| L \, h \|$.   Next, let $\zeta$ be the projection of $u$ onto the $1$-eigenspace of $N$
(if this is empty, set $\zeta = 0$).  The spectral gap on $\Sigma$ gives that
$\| u - v - \zeta \|_{ W^{1,2}} \leq C \, \| (\cL +1) \, u \|_{L^2}$.
The desired bounds  on $h_2^{\perp}$ and $(u-v)$ will follow once
we bound $|a|$ and $\zeta$.   The triangle inequality and the bounds thus far give
\begin{align}	
	\| \dv_f (a + \zeta \, g^1) \|_{L^2} \leq \| \dv_f \, h \|_{L^2} + \| \dv_f ( h - a - \zeta \, g^1) \|_{L^2} \leq \| \dv_f \, h \|_{L^2} + C\, \| L \, h \|_{L^2} \, .
	\notag
\end{align}
Since $a$ is purely Euclidean  and $\dv_f (a + \zeta \, g^1) = \nabla \zeta - a(\nabla f)$ , we have
\begin{align}
	\| \nabla \zeta \|_{L^2}^2 + \| a (\nabla f) \|_{L^2}^2 = \| \dv_f (a + \zeta \, g^1) \|_{L^2}^2 \leq 2\,  \| \dv_f \, h \|_{L^2}^2 + C\, \| L \, h \|_{L^2}^2 \, . 
\end{align}
Since $\zeta$ has eigenvalue one, this gives the desired bound $W^{1,2}$ bound on $\zeta$.  
It also gives the bound on $|a|$ since $\int f_j \, f_k \, \e^{-f} = \frac{\delta_{jk}}{2} \, \int e^{-f} $, so that
\begin{align}
	\| a(\nabla f)\|_{L^2}^2  = \sum_i \left\| \sum_j  a_{ij} \, f_j \right\|_{L^2}^2 =\sum_{i,j,k}  \int a_{ij} \, a_{ik} \, f_j \, f_k \, \e^{-f} = \frac{1}{2} \, |a|^2 \, \int \e^{-f} \, .    \, \, \, \, \, \, \, \, \, \, \, \, \, \, \qedhere
\end{align}
     \end{proof}

\begin{Lem}	\label{l:polygrow}
There exists $C_n , C > 0$ depending on $n$ so if $v\in \cK$ on $\RR^n$, then
\begin{align}
	 |v| +  (1+|x|^2)  \,  |\Hess_v| &\leq  C_n \, (1+|x|^2)  \,  \| v \|_{L^2} {\text{ and }}
	  |\nabla v| \leq C_n \, |x| \, \| v \|_{L^2} \, . \label{e:immgives} 
\end{align}
Furthermore, $2\, \| \Hess_{v} \|_{L^2}^2 =   \| \nabla v \|_{L^2}^2 =    \|  v \|_{L^2}^2$
and $u =|\nabla v|^2-\Delta\,|\nabla v|^2$ is in $\cK$ with
 \begin{align}   \label{e:secondorder}
	 \int u \, |\nabla v|^2 \, \e^{-f}  = \| u \|_{L^2}^2 \geq C \, \| v \|_{L^2}^4\, .
 \end{align}
\end{Lem}

\begin{proof}
By lemma $3.26$ in \cite{CM2}, $v = a_{ij} x_i x_j - 2\,a_{ii}$ for a constant matrix $a_{ij}$. 
This gives \eqr{e:immgives} and also that
$|\nabla v|^2$ is a homogeneous quadratic polynomial.  
 Since $\cL \, v = - v$,  \eqr{e:e326} (using the drift Bochner formula) gives  	$\| \Hess_{v} \|_{L^2}^2 =  \frac{1}{2} \, \| \nabla v \|_{L^2}^2 =   \frac{1}{2} \, \|  v \|_{L^2}^2$.
 
 Let $\cQ$ be the space of homogeneous quadratic polynomials and define the linear map 
$\Psi(w)=w- \Delta \, w$.  We will show that $\Psi$ maps to $\cK$.  For each $w \in \cQ$, 
 there is a constant $c$ so that  $w- c \in \cK$; since $\cK$ is orthogonal to constants,   $\int (w - c) \, \e^{-f} = 0$.  Using homogeneity again, we have $r\, \partial_r \, w = 2 \, w$ so that
	$\cL \, w= \Delta \, w - \frac{r}{2} \, \partial_r \, w =  \Delta \, w - w$.
 It follows that $c=\Delta \, w$ and, thus,  $\Psi (w) \in \cK$.
 Since $\Delta \, w$ is a constant while $w$ is second order,  $\Psi$ is one to one.  Thus, since $\cQ$ is finite dimensional,
 there is a constant $C_0$ such that for any $w\in \cQ$
\begin{align}    \label{e:so2}
	\|w\|_{L^1}\leq C_0\,  \| \Psi (w) \|_{L^1} = C_0 \, \left\|w- \Delta \, w \right\|_{L^1}\, .
\end{align}
 The equality in \eqr{e:secondorder} uses that $\cK$ is orthogonal to constants.  Applying \eqr{e:so2} with $w= |\nabla v|^2$
 and $\Psi (w) = u$ gives $\| v \|_{L^2}^2 = \| |\nabla v|^2 \|_{L^1} \leq C_0 \, \| u \|_{L^1} $.  The Cauchy-Schwarz inequality then gives
 the inequality in \eqr{e:secondorder}.
\end{proof}

Finally, we will need a   simpler estimate for metrics on $N$:

\begin{Lem}	\label{l:Ntriv}
Suppose that ($\star$) holds. There exists $C_N$ so that if $h$ is a symmetric $2$-tensor on $N$ and $\dv_N \, h = 0$, then 
\begin{align}
	\| h \|_{W^{2,2}} \leq C_N \, \| L_N \, h \|_{L^2} \, .
\end{align}
\end{Lem}

\begin{proof}
Since $L_N$ is elliptic and $N$ is compact, it suffices to show that $L_N \, h = 0$ and $\dv_N \, h = 0$ implies that $h=0$.  The trace-free part of $h$ vanishes by ($\star$), so we can assume that $h = w \, g^1$.  Then being in the kernel of $L_N$ forces $w$ to be a $1$-eigenfunction on $N$, but being divergence-free forces $w$ to be constant -- so $w\equiv 0$.
\end{proof}

\begin{Cor}	\label{r:specialcase}
Suppose that ($\star$) holds.  There exists $\delta_0 > 0$ so that if $f$ is a function and 
  $h$ is a symmetric $2$-tensor on $N$ with
$\dv_N \, h = 0$ and $\| h \|_{C^{2}} + \| \nabla f \|_{C^1} \leq \delta_0$, then  
\begin{align}
	\| h \|_{W^{2,2}} + \| \nabla f \|_{W^{1,2}}  \leq C \, \| \phi(h,f) \|_{L^2} \, , 
\end{align} where $\phi (h,f)$ is the shrinker quantity for $(N,g^1+h , f)$.
\end{Cor}

\begin{proof}
The linearization of the $\phi(h,f)$ is given by $\frac{1}{2} \, L_N \, h + \Hess_{ \frac{1}{2} \, \Tr \, h - f} + \dv_N^* \, \dv_N \, h$.  The last term vanishes here since $\dv_N \, h = 0$.  It follows that
\begin{align}	\label{e:Taaa}
	\left| \frac{1}{2} \, L_N \, h + \Hess_{ \frac{1}{2} \, \Tr \, h - f} \right| \leq |\phi (h,f)| + C \, ([h]_2^2 + [h]_2 \, [\nabla f]_1 ) \, , 
\end{align}
where $[h]_2$ is the pointwise $C^2$ norm of $h$.  Since $\dv_N \, h = 0$ and $N$ is Einstein, Theorem \ref{c:BHa} gives that $L_N \, h$ is $L^2$-orthogonal to any Hessian. 
Combining this with \eqr{e:Taaa} gives that
\begin{align}
	\frac{1}{4} \, \| L_N \, h \|_{L^2}^2 + \| \Hess_{ \frac{1}{2} \, \Tr \, h - f} \|_{L^2}^2 \leq 2 \, \| \phi (h,f) \|_{L^2}^2 +  C\, \delta_0 \, \left( \| h \|_{W^{2,2}}^2 + \| \nabla f \|_{W^{1,2}}^2 \right) 
\end{align}
Combining this with Lemma \ref{l:Ntriv} gives the desired bound on $h$ when $\delta_0 > 0$ is small.   
\end{proof}

 \section{Eigenfunctions and almost parallel vector fields}	\label{s:Eigen}
 
 The main result of this section is an extension theorem which shows that if a shrinker is close to a model shrinker on some large scale, then it remains close on a larger scale with a loss in the estimates.  To explain this,
 let $\Sigma = (N^{\ell} \times \RR^{n-\ell}, \bar{g} , \bar{f})$ be the model gradient shrinking soliton, where $N$ is closed and Einstein, $x_i$ are Euclidean coordinates, and $\bar{f} = \frac{|x|^2}{4} + \frac{\ell}{2}$.  
 
Let $\delta_0 , a_0 > 0$ be   constants that are sufficiently small (to be chosen) and fix $\alpha \in (0,1)$. We consider two   notions of closeness for a
  gradient shrinking    $(M,g,f)$:
\begin{itemize}
\item[($\star_R$)]	  There is a diffeomorphism $\Phi_R$ from a subset of $\Sigma$ to   $M$ onto $\{ b < R \}$ so that
$   \| \bar{g} - \Phi_R^* \, g \|_{W^{2,2}}^2 + \| \bar{f} - f \circ \Phi_R \|^2_{W^{2,2}}     \leq  \e^{ - \frac{ R^2}{4} } $ and $\|  \bar{g} - \Phi_R^* \, g \|_{C^{4,\alpha}} + \|  \bar{f} - f \circ \Phi_R \|_{C^{4,\alpha}} \leq \delta_0$.
\item[($\dagger_R$)]
There is a diffeomorphism $\Psi_R$ from a subset of $\Sigma$ to  $M$ that is onto $\{ b < R \}$ so that
	$\left\| \bar{g} - \Psi_R^* \, g \right\|_{C^{4,\alpha}}  + \left\| \bar{f} - f \circ \Psi_R \right\|_{C^{4,\alpha}}    
	 \leq   \e^{ -a_0 \, R^2 } $.
 \end{itemize}
 Note that 
  ($\star_R$) gives   stronger bounds on the region where $\bar{f}$ is small.

 \begin{Thm}	\label{t:extend}
 There exist $a_0 , R_0 , \beta > 0$ so that  if ($\star_R$) holds and $R \geq R_0$, then ($\dagger_{(1+\beta) \, R}$) holds.
 \end{Thm}
 
 The next proposition, which relies on the growth bounds, uses 
 ($\star_R$) to get almost linear functions on the larger scale $(1+\beta) \, R$.
 This is the key ingredient in Theorem \ref{t:extend}.  
 
  \begin{Pro}	\label{c:extend}
 There exist $C, m , R_1 , \beta_1 > 0$ so that  if ($\star_R$) holds and $R \geq R_1$, then we get $n-\ell$ functions $u_i$ 
 so that $\int u_i \, \e^{-f} = 0$ and on $\{ b < (1+\beta_1) \, R \}$ 
 \begin{align}	 \label{e:gooduis}
	\left| \delta_{ij} - \langle \nabla u_i , \nabla u_j \rangle \right|  + \| \Hess_{u_i} \|_{C^1}  + \left| 2\, \langle \nabla u_i , \nabla f \rangle -u_i 
	\right|  &\leq C\, R^m \, \e^{ - \frac{ R^2}{16} } \,  .
\end{align}
Furthermore, for each $m$, there exists $c_m$ so that $\| u_i \|_{C^{m}} \leq c_{m} \, R^{m}$ on $\{ b < (1+\beta_1) \, R \}$.
 \end{Pro}

Roughly, this proposition shows the propagation (outward in space) of almost splitting for a shrinker.  For flows, this corresponds to the propagation forward in time for an almost splitting; see \cite{CM13,CM14}.

    \subsection{Pseudo-locality}
   
 Applying
  pseudo-locality    to the flow generated by $(M, g,f)$ gives  estimates  forward in time for the flow and, thus, estimates on the shrinker on a larger scale.  
        Let $c_n$  be the Euclidean isoperimetric constant and define
    $(M,g,f)$ to be $(\delta , r_0)$-Euclidean  to scale $R$ if 
    $|\partial \Omega|^n \geq (1-\delta) \, c_n \, |\Omega|^{n-1}$  for every $\Omega \subset \{ b \leq R \}$ with $\diam \, \Omega \leq r_0$.

   \begin{Pro}	\label{t:brakke}
  There exist $\delta_0 > 0$,  $R_0$ and $C_0$ so that for any $r_0 \in (0,1)$, we get $\alpha_0 = \alpha_0 (r_0 , n) > 0$ so that if $(M,g,f)$ is $(\delta , r_0)$-Euclidean out to scale $R \geq R_0$, then 
  \begin{align}
  	\sup_{ \{ b \leq (1+\alpha_0) \, R \} } \, \, |\R| \leq C_0 \, r_0^{-2} \, .
  \end{align}
   \end{Pro}

 Once we have the $\R$ bound, then the Shi estimates, \cite{S},  give corresponding bounds on  $\nabla \R$, etc.  There are also versions of this estimate for expanding solitons, where the estimate forward in time for the flow corresponds to estimates for the expander on smaller scales.

 We will use  Li-Wang's version for shrinkers (theorem $25$ in \cite{LW1}) of Perelman's pseudo-locality (cf. $10.1$  in \cite{P} or $30.1$ on page 2658 of \cite{KL1}):
 There exist $\delta , \epsilon > 0$ with the following
property. Suppose that   $(M, (x_0,-1)), g(t))$ is  a smooth pointed Ricci flow 
for $t \in [-1, \epsilon \, r_0^2- 1]$.  If $S \geq - r_0^{-2}$ 
 on $ B_{r_0} (x_0)\times \{ t = -1 \}$ and $B_{r_0}(x_0)$ is $(\delta , r_0)$-Euclidean at time $-1$, then 
\begin{align}
	| \R |(x, t) <  (t+1)^{-1} + (\epsilon \, r_0)^{-2}	\label{e:pseudoL}
\end{align}
 whenever $-1 < t< (\epsilon \, r_0)^2-1$ and $\dist_t (x,x_0) \leq \epsilon \, r_0$.

\begin{proof}[Proof of Proposition \ref{t:brakke}]
For each $x\in M$, let $\gamma_x$ be the integral curve given by 
 $\gamma_x(-1)=x$ and 
$\gamma_x'(t) =-\frac{1}{t}\,\nabla f \circ \gamma_x (t) $.
Here, the gradient is computed with respect to the fixed metric $g$.
Define  $\Phi (x,t)=\gamma_x(t)$ so that $\Phi (x,-1)=x$ and   $\partial_t \,  \Phi= \frac{1}{-t} \, \nabla f \circ \Phi$.
 Working in the background metric $g$, we have
\begin{align}	\label{e:dtfphi}
	\partial_t f(\Phi (x,t)) = \langle \nabla f (\Phi (x,t)) , \frac{1}{-t} \, \nabla f (\Phi (x,t))  \rangle = \frac{1}{-t} \, \left| \nabla f \right|^2 \circ \Phi (x,t) \, .
\end{align}
We will show that there exists $C$ depending on $B_1 \subset M$ so that if  $\alpha > 0$ and $y$ is a point with $S \leq C_S$ on $\{f \leq (1+\alpha) \, f(y)\}$, then for $t \in (-1, 0)$
\begin{align}	\label{e:goody}
	  f(\Phi (y,t)) \geq  \, \min \, \left\{ (1+ \alpha)\,  f(y) , \frac{f(y)  - C_S}{-t}  \right\} \, .
\end{align}
 By \eqr{e:dtfphi} and \eqr{e:fromccz}, 
	$\partial_t \, f(\Phi (y,t))  = \frac{1}{-t} \, \left(f -S \right) \circ  \Phi (y,t) $.
Rewriting this as $\partial_t \, (-t\, f) =  - S$ and 
integrating  from $-1$ to $t$ gives
	$ -(\sup S) \, (1+t)  \leq -t \, f(\Phi(y,t)) - f(y) \leq 0 $, 
where the supremum of $S$ is taken over the curve $\Phi(y,s)$ for $-1 \leq s \leq t$. Combining this and monotonicity of $f$ along the flow line gives \eqr{e:goody}.

 It is well-known that 
   $g(x,t) = -t \, \Phi^*  \, g (x)$ is a 
   Ricci flow.  
    By assumption, the set $\{ b \leq R \}$ is $(\delta , r_0)$-Euclidean at time $-1$ and has $S \geq 0$ by \cite{Cn}.  Thus, if $x \in M$ is any point with $B_{r_0} (x) \subset \{ b \leq R \}$, then 
pseudo-locality \eqr{e:pseudoL} gives
 	$\left| \R_{g(x,-1+r_0^{-2})} \right| \leq C \, r_0^{-2}$.
It follows from \eqr{e:goody} that this curvature bound  for the evolving metric is equivalent to a curvature bound for $\R_g$ some fixed factor further out.  Here there is an additive loss because of the last term   that can be absorbed as long as $R$ is large enough.
\end{proof}

 \subsection{Spectral estimates}  \label{s:spec}
 
We show that if ($\star_R$) holds, then $M$ has $(n-\ell)$ $L^2$ eigenfunctions with eigenvalues exponentially close to $\frac{1}{2}$.   
  By \cite{HN} and    \cite{CxZh2} (cf. \cite{AN,BE,FLL}), 
  $\cL$ has discrete spectrum   $\mu_0 = 0<   \frac{1}{2}\leq \mu_1  \leq \dots \to \infty$ on $M$  and
 the eigenfunctions   are in $W^{1,2}$.

\begin{Lem}	\label{p:thefunctions}
There exists $C$ so that 
if ($\star_R$) holds, then 
 there are $(n-\ell)$   $L^2$-orthonormal functions $v_i$ on $M$   with
$\cL\,v_i+\mu_i\,v_i=0$, $\mu_i \geq \frac{1}{2}$, and  
\begin{align}	\label{e:muibound}
\| \Hess_{v_i} \|_{L^2}^2 + \left(\mu_i-\frac{1}{2}\right)  &\leq  C\, R^2 \, \e^{- \frac{R^2}{8}}  \, .
\end{align}
\end{Lem}

We will use the low eigenfunctions on $\Sigma$ as test functions to get an upper bounds for the low eigenvalues on $M$.  
The next lemma recalls the properties of the low eigenvalues on $\Sigma$.

\begin{Lem}	\label{l:modelR}
There exist $\bar{c} , C$ so that  $\bar{u}_0 = \bar{c}$ and $\bar{u}_i = \frac{\bar{c}}{\sqrt{2}} \, x_i$, for 
$1\leq i \leq {n-\ell}$ satisfy
\begin{align}
	\left| \delta_{ij} - \int_{\bar{b} < r} \bar{u}_i \, \bar{u}_j \, \e^{-\bar{f}} \right| &\leq C \, r^{n-\ell} \, \e^{ - \frac{r^2}{4}} {\text{ for all }} i, j 
	 \, , \\
	\left| \frac{1}{2} \, \delta_{ij} - \int_{\bar{b} < r} \langle \nabla \bar{u}_i , \nabla \bar{u}_j \rangle 
	\, \e^{-\bar{f}} \right| &\leq C\, r^{n-\ell} \, \e^{ - \frac{r^2}{4}}  {\text{ for }} i, j \geq 1 \, .
\end{align}
\end{Lem}

\begin{proof}
Choose $\bar{c}$ depending on $\ell , n , \Vol (N)$ so that $\int_{\Sigma} \bar{c}^2 \, \e^{ - \bar{f}} = 1$.  Since $\cL \, (x_i \, x_j) =
2\, \delta_{ij} - x_i \, x_j$, it follows that $\bar{u}_0 = \bar{c}$ and $\bar{u}_i = \frac{\bar{c}}{\sqrt{2}} \, x_i$, for 
$1\leq i \leq {n-\ell}$ are $L^2$-orthonormal.
To estimate the ``tails'' of these integrals, observe that 
\begin{align}
	\int_{\bar{b} \geq s} \, (1+ |x|^2) \, \e^{-\bar{f}} \leq \Vol (N) \, \Vol (\SS^{n-\ell -1}) \, \e^{ - \frac{\ell}{2}} \, 
	\int_{\sqrt{s^2 - 2\, \ell}}^{\infty} r^{n-\ell -1} (1+r^2) \, \e^{ - \frac{r^2}{4}} \, .  \, \, \, \, \, \, \, \, \, \, \, \, \, \qedhere
\end{align}
\end{proof}

\begin{proof}[Proof of Lemma \ref{p:thefunctions}]
Define a cutoff function on $M$ to be zero for $\{  b > R\}$, one for $\{ b< R - 1 \}$, and with
$\eta = R - b$ in between.  Note that $|\nabla \eta| \leq 1$.    Let $\bar{u}_i$ be as in Lemma \ref{l:modelR} and 
set   $u_i =  \eta \, \bar{u}_i \circ \Phi_R^{-1}$. Define symmetric matrices $a_{ij}$ and $b_{ij}$ by
\begin{align}	\label{e:defab}
	a_{ij} = \int  u_i \, u_j\, \e^{-f} {\text{ and }} b_{ij} = \int \langle \nabla u_i , \nabla u_j \rangle\,\e^{-f}\, .
\end{align}
   Lemma \ref{l:modelR} and the change of variable formula give
   \begin{align}	\label{e:aijstart}
   	\left| \delta_{ij} - a_{ij} \right| &\leq C \, R^n \, \e^{ - \frac{(R-2)^2}{4}} + \int_{ \{ b < R-1\} }
	\left| u_i \, u_j \right| \, \left| \e^{ -f} \frac{dv_g}{dv_{\bar{g}}} - \e^{ - \bar{f}}
		\right| \, dv_{\bar{g}}
	\, ,
   \end{align}
   where we use the shorthand ${f} \equiv {f} \circ \Phi_R^{-1}$ and similarly for $dv_{{g}}$.  The triangle inequality gives
    \begin{align}
   	 \left| \e^{ -f} \frac{dv_g}{dv_{\bar{g}}} - \e^{ - \bar{f}}
		\right| &\leq   \left| 1- \sqrt{ \det \left( \bar{g}^{-1} \, g \right)}
		\right| \, \e^{-f} +  \left| \e^{\bar{f} -f}   - 1
		\right| \, \e^{- \bar{f}}
   \end{align}
     Therefore,  since $|f -\bar{f}|$ is small by  
   ($\star_R$),  we can Taylor expand to get that
     \begin{align}
   	 \left| \e^{ -f} \frac{dv_g}{dv_{\bar{g}}} - \e^{ - \bar{f}}
		\right| &\leq   C \, \left( |f - \bar{f}| + |\bar{g} - g| \right)  \, \e^{- \bar{f}} \, .
   \end{align}
   Combining this with \eqr{e:aijstart} and using the $L^2$ bound from  ($\star_R$)
   gives
\begin{align}
	\left| \delta_{ij} - a_{ij} \right| \leq C\, R^2 \, \e^{ - \frac{R^2}{8}} \, .
\end{align}
    Thus, $a_{ij}$ is invertible, the inverse $a^{ij}$
    has the same estimate, and
    we get $c_{ij}$ so that the $ \sum_j c_{ij} \, u_j $'s are $L^2(M)$ orthonormal.  It follows that $\delta_{ij} = c_{ik} \, a_{km} \, c_{jm}$ and, thus, $c^T \, c = a^{-1}$.  
    A similar argument shows that $\left| \int u_i \, \e^{-f} \right| \leq C \, R \, \e^{ - \frac{R^2}{8}}$.  The variational characterization of eigenvalues gives
\begin{align}	\label{e:myvari}
	\sum_{i=0}^m \mu_i \leq \sum_{i} \| \nabla v_i \|_{L^2(M)}^2 = \sum_{i,j,k} c_{ij} \, b_{jk} \, c_{ik} = \sum_{i,j} b_{ij} \, (c^T\, c)_{ij} 
	= \sum_{i,j} b_{ij} \, a^{ij}  \, .
\end{align}
Set $\bar{b}_{00} = 0$ and $\bar{b}_{ij} = \frac{1}{2} \, \delta_{ij}$ for $1\leq i +j$.
 Note that $\sum_{i,j \leq n-\ell} \delta_{ij} \, \bar{b}_{ij} = \frac{n-\ell}{2}$.  Arguing as above gives
 	$\left| \bar{b}_{ij} - b_{ij} \right| \leq C \, \e^{ - \frac{R^2}{8}}$.
	 Using this and $\left| \delta_{ij} - a^{ij} \right| \leq C\, R^2  \, \e^{ - \frac{R^2}{8}}$ in
   \eqr{e:myvari} gives 
 \begin{align}	\label{e:upperBo}
	\sum_{i=0}^{n-\ell} \mu_i \leq \sum_{i,j} a^{ij} \, b_{ij} \leq  \frac{n-\ell}{2} + C\, R^2 \,  \e^{ - \frac{R^2}{8}}  \, .
\end{align}
Now let $v_0 , \dots , v_{n-\ell}$ be the first $n-\ell +1$ eigenfunctions, i.e., an $L^2$-orthonormal set corresponding to $\mu_0=0$ up to $\mu_{n-\ell}$.  
  By \eqr{e:e326}, the drift Bochner formula gives  	$\| \Hess_{v} \|_{L^2}^2 = \left( \mu - \frac{1}{2} \right) \, \| \nabla v \|_{L^2}^2$.
  It follows that $\mu_0 = 0$ and $\mu_i \geq \frac{1}{2}$ for $i  \geq 1$.  Combining this with \eqr{e:upperBo} gives
 \eqr{e:muibound}.  
\end{proof}

 \begin{proof}[Proof of Proposition \ref{c:extend}]
 Let $v_1 ,\dots , v_{n-\ell}$ from Lemma \ref{p:thefunctions} and set $I_{i} = r^{1-n} \, \int_{b=r}
 v_i^2 \, |\nabla b|$.  Since $\mu_i < 1$, Theorem \ref{et:main} gives $r_0 = r_0 (n)$ so that
\begin{align}	\label{e:CM11a}
	I_i(r_2) \leq 2 \, \left( \frac{r_2}{r_1} \right)^4 \, I_i (r_1) {\text{ if }}  r_0 < r_1 < r_2 \, .
\end{align}
The bound \eqr{e:CM11a}
used the complete shrinker $M$.  The rest of the argument focuses on the region where we have a priori bounds.   Proposition \ref{t:brakke} gives $\alpha_0 > 0$ and $C_0$ so that $|\R| + |\nabla \R|\leq C_0$ on $\{ b < (1+\alpha_0)\, R\}$.
This bound $\Ric$ and $S$ and, thus, gives a positive lower bound for $|\nabla b|$.  Therefore, \eqr{e:CM11a}
gives polynomial bounds for the {\it{ordinary}} $L^2$ norm (i.e., without $|\nabla b|$ or $\e^{-f}$) on $\{ b < (1+\alpha_0)\, R \}$.  This and  elliptic estimates for the eigenvalue equation bound  $\Hess_{v_i}$ and $\nabla \Hess_{v_i}$ 
on $\{ b < (1+\alpha_0)\, R- R^{-1} \}$.  Thus, 
\begin{align}
	I_{\Hess_{v_i}} (r) = r^{1-n} \, \int_{b=r} \left| \Hess_{v_i} \right|^2 \, |\nabla b| {\text{ and }}
	D_{\Hess_{v_i}} (r) = \frac{r^{2-n}}{2} \, \int_{b=r} \langle \nabla \left| \Hess_{v_i} \right|^2  , \frac{\nabla b}{|\nabla b|} \rangle
\end{align}
are polynomially bounded for $r \leq (1+\alpha_0)\, R- R^{-1}$.  Furthermore,  Corollary
\ref{t:bochHess} and the $|\R|$ bound on this region
give that
\begin{align}
	\langle \cL \, \Hess_{v_i} , \Hess_{v_i} \rangle &= \langle L \, \Hess_{v_i} , \Hess_{v_i} \rangle - 2 \, \langle \R (\Hess_{v_i} ), \Hess_{v_i} \rangle \geq - C  \, \left| \Hess_{v_i} \right|^2 \, .
\end{align}
Moreover, local elliptic estimates and the $L^2$ Hessian bound in \eqr{e:muibound} give that for each fixed $r << R$ we have
$I_{\Hess_{v_i}} (r) \leq C_r \, \e^{ - \frac{R^2}{8}}$.
Therefore, we can apply   proposition  \ref{ep:effective} to get $m$ and $C$ so that 
$I_{\Hess_{v_i}} (r) \leq C \, R^m \, \e^{ - \frac{R^2}{8}}	$
for $ r \leq (1+\alpha_0)\, R- R^{-1}-1$. 
Using  elliptic estimates on scale $R^{-1}$ again, we conclude that on $\{ b < (1+\alpha_0) \, R - 2 \}$
\begin{align}	\label{e:ptHess}
	\left| \Hess_{v_i} \right|^2 + R^{-2} \, \left| \nabla \, \Hess_{v_i} \right|^2 \leq C \, R^{m+n} \,  \e^{ - \frac{R^2}{8}}	 \, .
\end{align}
Since $\mu_i < 1$, Lemma \ref{c:localize}
gives for each $s$ that
\begin{align}
	\frac{s^2}{4} \, \int_{ b \geq s} \, \left\{ v_i^2 + |\nabla v_i|^2 \right\} \, \e^{-f} \leq 4\, \mu_i^2 + (n+2)\,  \mu_i + n < 2\, n + 6  	\label{e:conco}
	\, .
\end{align}
It follows from \eqr{e:conco} that there is a fixed $s$ and constant $q_0 > 0$ (independent of $R$) so that the matrix $Q_{ij} = \int_{\{ b < s\}} \langle \nabla v_i , \nabla v_j \rangle \, \e^{-f}$ 
is invertible with $|Q| + |Q^{-1}| < q_0$.  Note that \eqr{e:ptHess} and the fundamental theorem of calculus  imply that 
$\langle \nabla v_i , \nabla v_j \rangle$ is exponentially close to being constant on $\{ b < (1+\alpha_0) \, R - 2 \}$.  Therefore, we can choose a bounded linear transformation $\tilde{Q}_{ij}$ so that $u_i = \tilde{Q}_{ij} \, v_j$ satisfy $\int u_i \, \e^{-f} = 0$ and
\begin{align}
	\sup_{   \{ b < (1+\alpha_0) \, R - 2 \} }
	\, \left| \delta_{ij} - \langle \nabla u_i , \nabla u_j \rangle \right| \leq C \, R^{m'} \, \e^{ - \frac{R^2}{16}} \, .
\end{align}
This gives the first bound in \eqr{e:gooduis} and the next two bounds follow from \eqr{e:ptHess}
since $\tilde{Q}$ is bounded.  The last bound in \eqr{e:gooduis} follows similarly, using also that the $\mu_i$ close to $\frac{1}{2}$.

Finally, the uniform (but not exponentially small) higher derivative bounds on the $u_i$'s follow from the uniform higher order bounds on the curvature from pseudo-locality and  the Shi estimates together with elliptic estimates on the scale of $R^{-1}$ (each $u_i$ solves a drift eigenvalue equation). 
 \end{proof}

\begin{proof}[Proof of Theorem \ref{t:extend}]
We will use freely below that $|\R| + |\nabla \R|\leq C_0$ on $\{ b < (1+\alpha_0)\, R\}$
by
Proposition \ref{t:brakke}.
Proposition \ref{c:extend}
gives $n-\ell$ ``almost linear'' functions $u_i$ 
 satisfying \eqr{e:gooduis}.   
 Using the bounds on $ \nabla^2 \, u_i$ and $\nabla^3  \, u_i$ in \eqr{e:gooduis},
 the definition of   $\R$ gives
  \begin{align}	\label{e:Ruismall}
 	\left| \R (\nabla u_i , \cdot , \cdot , \cdot ) \right| \leq C \, R^m \, \e^{ - \frac{R^2}{16}} \, .
\end{align}
Tracing this gives $ 	\left| \Ric (\nabla u_i , \cdot  ) \right| \leq C \, R^m \, \e^{ - \frac{R^2}{16}}$.  
Since $M$ is fixed close to the model $\Sigma$ on a fixed central ball, $\Ric$ has a block decomposition with an $\ell \times \ell$ block close to $\frac{1}{2} \, g^1$ and a complementary block that almost vanishes.  Thus, by \eqr{e:Ruismall}, the span of the $\nabla u_i$'s is almost orthogonal to $\Phi_R(N)$.
It follows that the projection of $\nabla f$ perpendicular to the span of the $\nabla u_i$'s is also fixed small.
Thus, if we set $\tilde{f} = \frac{\ell}{2} + \frac{1}{4} \, \sum u_i^2$, then \eqr{e:gooduis} gives that $\tilde{f} - \frac{\ell}{2} - |\nabla f|^2$ is bounded.
Therefore, since $|\nabla f|^2 +S = f$ (by \eqr{e:normalizef}) and $S$ is bounded, we see that $|\tilde{f}-f| \leq C$.

Let $N_0 = \{ u_1 = \dots = u_{n-\ell} = 0 \}$ be the intersection of the zero sets and $f_0$ the restriction of $f$ to $N_0$.
Since $|\tilde{f}-f| \leq C$,  $f_0\leq \frac{\ell}{4} +C$ and, thus, $\tilde{f}\leq C'$ on $N_0$.  It follows that $N_0$ is a smooth $\ell$-dimensional submanifold, the $\nabla u_i$'s span its normal space,   and $|\nabla^{\perp} f|$ is exponentially small on $N_0$ by \eqr{e:gooduis}.
Moreover, the level sets of the map $(u_1 , \dots , u_{n-\ell})$  foliate $b< R$, so 
$N_0$ is connected and diffeomorphic to $N$.
Moreover, 
     $N_0$ is locally a graph with small gradient over   $\Phi_R(N)$. 
     Using the slice theorem, fix a diffeomorphism $\Psi_0 : N \to N_0$ with $\dv_N \, (\Psi_0^* \, g_0 - g^1) = 0$ 
   and with $(\Psi_0^* \, g_0 - g^1)$ fixed $C^{4,\alpha}$ small (cf. theorem $3.6$ in \cite{V},   $3.1$ in \cite{ChT}).   
   Let $\tilde{\theta} = \Psi_0^* \, g_0 - g^1$ denote the metric perturbation on $N$.

   Let $e_j$ be an orthonormal tangent frame for $N_0$.
Using \eqr{e:gooduis}, we see that the second fundamental form $A $ of $N_0$ satisfies
\begin{align}	\label{e:Ajksmall}
	\left| \langle A(e_j , e_k) , \nabla u_i \rangle \right| = 	\left| \langle  e_k , \nabla_{e_j} \nabla u_i \rangle \right| \leq C \, R^m \, \e^{ - \frac{R^2}{16}} \, .
\end{align}
Combining this, the Gauss equation, and \eqr{e:Ruismall}, we see that the Ricci curvature $\Ric^0$ of $N_0$ and the Hessian $\Hess^0_{f_0}$ satisfy
\begin{align}	\label{e:RicHess}
	\left| \Ric^0 (e_j , e_k) - \Ric (e_j , e_k) \right|  + \left| \Hess_{f_0}^0 (e_j , e_k) - \Hess_f (e_j , e_k) \right|
	\leq  C \, R^m \, \e^{ - \frac{R^2}{16}} \, .
\end{align}
Since $\Ric$ vanishes exponentially in the normal directions by  \eqr{e:Ruismall}, we get the same bound for the difference $|S^0 - S|$ of the scalar curvatures.  It follows that the shrinker quantity $\phi_0$ on $N_0$ satisfies
\begin{align}	\label{e:aronson0}
	\left|\phi_0 \right| + \left| |\nabla f_0|^2 + S^0 - f_0 \right| \leq C \, R^m \, \e^{ - \frac{R^2}{16} } \, ,
\end{align}
with similar estimates for the $C^{2,\alpha}$ norm of $\phi_0$.
 Corollary \ref{r:specialcase} now gives that $\| \tilde{\theta} ||_{W^{2,2}}$ and $\| \nabla f_0 \|_{W^{1,2}}$ are exponentially small.  
 Furthermore, since $\dv_N \, \tilde{\theta} = 0$, the equation for $\Ric_{N_0}$ is elliptic{\footnote{By, e.g., $1.174$ in \cite{Bs}, given $x \in TN$ the principal symbol maps a symmetric matrix $B$ to $\frac{1}{2} \, |x|^2 \, B + \frac{1}{2} \, (\Tr \, B) \, x \otimes x$. Suppose that $x\ne 0$ and $B$ is in the kernel of this map.  Taking the inner product with the identity gives $|x|^2 \, \Tr \, B$, so $\Tr \, B = 0$ 
 and, thus, $B=0$ since $|x|^2 \, B + (\Tr \, B) \, x\otimes x =0$.  This gives ellipticity.}}
   in $\tilde{\theta}$; elliptic estimates give exponentially small bounds for the $C^{4,\alpha}$ norms.   

We must now extend the estimates off of $N$ and $N_0$.    Note that $|f - \tilde{f}| \leq C \, R^m \, \e^{ - \frac{R^2}{16} }$ on $N_0$.  Differentiating gives
\begin{align}
	\nabla_{\nabla u_j} \left( f - \tilde{f} \right) &= \langle \nabla u_j , \nabla f \rangle - \frac{1}{2} u_j + \frac{1}{2} \, u_i \, \left( \langle \nabla u_j , \nabla u_i \rangle - \delta_{ij} \right)  \, .
\end{align}
This is exponentially small by \eqr{e:gooduis} and, by integrating up, so is $f-\tilde{f}$.
 
 Define the map $H: N_0 \times  \RR^{n-\ell} \to M$ by letting
$H (q, x_1 , \dots , x_{n-\ell})$ be the time one flow starting at $q$ along the vector field $\sum x_i \, \nabla u_i$.  
Now, set
\begin{align}
	\Psi (p, x_1 , \dots , x_{n-\ell}) = H(\Psi_0 (p), x_1 , \dots , x_{n-\ell}) \, .
\end{align}
Write $x = r \, y$ where $y \in \SS^{n-1}$ and observe that
 	$H_r = \langle  y , \nabla u \circ H \rangle$
and this is exponentially parallel.  It follows that $H$ is exponentially close to a local isometry and, thus, also a local diffeomorphism.
 Similarly,   $u_i \circ H - x_i$ is exponentially small and, thus, $H$ is proper.
 Since $H$ is a proper local diffeomorphism between complete connected spaces and has pull-back metric bounded from below, \cite{GW} implies that $H$ has the path-lifting property.  Using that the image and target are topologically $N$ times a Euclidean space, we see that $H$ is a diffeomorphism from a subset of $N \times \RR^{n-\ell}$ onto $\{ b < (1+\alpha_0) \, R \}$.

 Since pseudo-locality gives uniform curvature bounds, the drift eigenfunction equation has uniform bounds on scale $\frac{1}{1+ b}$.  Elliptic estimates on this scale give higher derivative bounds on the 
eigenfunctions and, thus, on $\Psi$.
\end{proof}

 \section{Variations of geometric quantities}	\label{s:variations}
 
  The main result of this section is  the formula \eqr{e:SVJa}  for  $\phi''$ in the direction of a Jacobi field.   
Let $g(t) =g + t\, h $ and $f(t)= f + t \, k$  be   families of metrics and functions.  We will work in a frame $\{ e_i \}$ of coordinate vector fields  independent of $t$. 

Cao-Hamilton-Ilmanen, \cite{CaHI} (cf. \cite{CaZ}), computed the first variation of $\phi$ 
  \begin{align}	\label{c:phip}
  	\phi' (0) &=  \frac{1}{2} \,  L \, h   + \Hess_{ \left( \frac{1}{2} \, \Tr \, (h) - k \right) } + 
	  \dv_f^{*}\,  \dv_f \, h \, .
  \end{align}
  Thus, the Jacobi fields on $\Sigma = N \times \RR^{n-\ell}$ consist of  $h=u\,g^1$ and $k= \frac{\ell}{2} \, u$  with $u\in \cK$.

 \begin{Pro}	\label{p:twovar}
If  $h=u\,g^1$ and $k= \frac{\ell}{2} \, u$ on $\Sigma$ where $u$ depends only on $\RR^{n-\ell}$, then 
 \begin{align}	\label{e:SVJa}
  	2\, \phi_{ij}''(0) &=   -   2\,  |\nabla u|^2  \, g^1     - 2\, \ell\, u\, u_{ij} -  \ell\, u_i\, u_j      \, .
  \end{align} 
 \end{Pro}
 
 Formally, Proposition \ref{p:twovar} and Lemma \ref{l:polygrow} say that the Jacobi fields are not integrable since
  \begin{align}
  	\int \langle \phi'' ,  (|\nabla u|^2 - \Delta |\nabla u|^2) \, g^1 \rangle \, \e^{-f} = - \ell \, \int (|\nabla u|^2 - \Delta |\nabla u|^2) \, |\nabla u|^2 \, \e^{-f} \leq - C \, \| u \|_{L^2}^4
	\end{align}
  is strictly negative, but $\phi$, and thus  $\phi''$,  vanish on a one-parameter family of shrinkers.

\subsection{First variations}

    In this subsection, we collect  well-known first variation formulas   (see, e.g., \cite{T} or \cite{CaZ})   for reference;
    these results do not use the product structure on $\Sigma$.

\begin{Pro}		\label{p:varyR}
The variations of $S, \Ric$ and $\R$ are given by
\begin{align}
	S' &= - \langle h , \Ric \rangle + \dv^2 \, h - \Delta \, \Tr \, h \, , \\
	2\, \Ric_{ij}' &= - \Delta \, h + h_{ik} \, \Ric^k_j + \Ric^i_k \, h_{kj} -2 \, \R (h) - \Hess_{\Tr \, h} + \nabla \dv \, h +( \nabla \dv \, h  )^T \, , \\
	2\, \R_{ijkn}' &= \R_{ijk\ell} \, h^{\ell}_{n} -  \R_{ij n\ell } \, h^{\ell}_{ k} +  h_{in,kj} - h_{jn,ki} + h_{jk,ni} - h_{ik,nj} \, .
\end{align}
Here $\dv \, h$ is the divergence{\footnote{Note the different sign convention from \cite{T}.
}} of $h$ given by $(\dv \, h)_i = h_{ij,j}$ and $( \nabla \dv \, h  )^T $ is the transpose.
\end{Pro}

By definition, 
$
	\nabla_{e_j} e_i = \Gamma_{ji}^k \, e_k$,
where
	\begin{align}
		\Gamma_{ji}^k =\frac{1}{2} \,g^{km} \,  \left(e_i (g_{mj}) + e_j(g_{mi}) - e_m(g_{ji})  \right)
	\end{align}
	is the Christoffel symbol.
Since these are coordinate vector fields, we have $\Gamma_{ji}^k = \Gamma_{ij}^k$.
Even though  $\nabla_{\cdot}  (\cdot) $  is not a tensor (it is not tensorial in the upper slot), the derivative is a tensor.

\begin{Lem}	\label{l:H4}
At a point where $g_{ab} = \delta_{ab}$ and $e_c(g_{ab}) =0$ at $t=0$,  we have  
\begin{align}
	\left( \nabla_{e_i} e_j \right)' &= \sum_k \cC_{ij}^k \, e_k {\text{ where }}
	\cC_{ij}^k = \frac{1}{2} \,  \left(h_{kj,i} + h_{ki,j} - h_{ji,k}  \right)  \, , \label{e:csone} \\
	\left( \nabla_{e_n} \nabla_{e_j} e_i \right)'  &= \frac{1}{2} \,  \left\{  h_{kj,in} + h_{ki,jn} - h_{ji,kn}    \right\} \, e_k \, .
\end{align}
\end{Lem}

\begin{Lem}	\label{l:hessf}
If $u$ is a one-parameter family of functions, then 
at $t=0$ at a point where $g_{ij} = \delta_{ij}$ and $e_c (g_{ab})=0$ we have
	$\left( \Hess_u \right)' = \Hess_{u'} - \cC_{ij}^k \,  u_k$.
\end{Lem}

\begin{Lem}	\label{l:hesstenp}
At $t=0$ at a point where $g_{ij} = \delta_{ij}$ and $e_c (g_{ab})=0$,   we have
\begin{align}	 \label{e:hessbij}
	e_m \left( b_{ij}' \right)  & = [b_{ij,m}]'  +  b_{nj} \cC_{mi}^n   + b_{in} \cC_{mj}^n
	        \, , \\
	  e_n  \left[  {e_m} \left( [b_{ij}]' \right) \right]   &= \left( b_{ij,mn} \right)' + (\nabla b) ( \cC_{ni}^p e_p, e_j , e_m) + (\nabla b) (e_i , \cC_{nj}^p e_p, e_m)  
	  + (\nabla b) (  e_i , e_j ,\cC_{nm}^p e_p) \notag  \\ &
	  + (\nabla b)(\cC_{mi}^p e_p,  e_j,e_n) + b([\nabla_{e_n} \nabla_{e_m} e_i ]', e_j)    + b'(\nabla_{e_n} \nabla_{e_m} e_i , e_j)    \\
	  &+ (\nabla b)( e_i , \cC_{mj}^p e_p ,e_n)    + b(   e_i , [\nabla_{e_n} \nabla_{e_m} e_j]' )     + b'(   e_i , \nabla_{e_n} \nabla_{e_m} e_j)  \, . \notag 
\end{align}
\end{Lem}

 \begin{Lem}	\label{l:phip}
The derivative of $\phi = \kappa\, g - \Ric - \Hess_f$ is 
  \begin{align}
  	\phi_{ij}' &= \kappa \, h_{ij}  + \frac{1}{2} \,( L \, h)_{ij}    + \left( \frac{1}{2} \, \Tr \, (h) - k \right)_{ij}  \\
	&- \frac{1}{2} \left( g^{kn} h_{jk,ni} + g^{kn} h_{ik,nj} + \Ric_{i}^{k} h_{jk} + \Ric_{j}^{k} h_{ik}    \right)   + \frac{1}{2} \, \left( h_{jn,i} + h_{in,j}   \right) g^{nm} f_m \notag \, .
  \end{align}
  \end{Lem}

 \begin{Lem}	\label{l:adj}
If $(M,g,f)$ is a gradient shrinking soliton, then
  \begin{align}
 	 ( \dv_f^{*}\,  \dv_f \, h) (e_i , e_j) & =\kappa \, h_{ij}  - \frac{g^{kn}}{2} \left\{ h_{jk,ni}   - f_n h_{kj,i}  +  h_{ik,nj}      - f_n h_{ik,j}
	         + \Ric_{ni} h_{kj}   +
	     \Ric_{nj} h_{ik} \right\} \, . \notag
 \end{align}
 \end{Lem}

 Combining Lemma \ref{l:phip} and Lemma \ref{l:adj} recovers the first variation \eqr{c:phip} for $\phi$.

\subsection{Computing $\phi''(0)$: Proof of Proposition \ref{p:twovar}}

\begin{Lem}	\label{l:lapei}
At a point where $g_{ij} = \delta_{ij}$ and $e_c(g_{ab})=0$ at $t=0$, 
\begin{align}
	2\, \left[ \left( \nabla_{e_n} \nabla_{e_n} e_i \right)' \right]_k &=    h_{ki,nn}  +  h_{kn,i}\, f_n   -h_{in,k}\, f_n 
		+  (\dv_f \,h)_{k,i} -  (\dv_f \,h)_{i,k}   \, .  
\end{align}
\end{Lem}

\begin{proof}
Taking the trace in the second claim in Lemma \ref{l:H4}  at $t=0$ gives
\begin{align}	\label{e:fromH4}
	2\, \left[ \left( \nabla_{e_n} \nabla_{e_n} e_i \right)' \right]_k &=     h_{kn,in} + h_{ki,nn} - h_{ni,kn}      \, .  
\end{align}
The Ricci identity and $(\dv_f \, h)_{k,i} = h_{kn,ni} - f_k \, h_{kn,i} - f_{ki} \, h_{kn}$ gives 
\begin{align}
	h_{kn,in} &= h_{kn,ni} + \R_{nikm}\, h_{mn} + \R_{ninm}\, h_{mk} = h_{kn,ni} - [\R(h)]_{ki} + \Ric_{im}\, h_{mk} \notag \\
	&= (\dv_f \,h)_{k,i} + h_{kn,i}\, f_n + h_{kn}\, f_{ni} - [\R(h)]_{ki} + \Ric_{in}\, h_{nk} \, , 
\end{align}
Using the shrinker equation,  this becomes
 $	h_{kn,in}  =  (\dv_f \,h)_{k,i} + h_{kn,i}\, f_n - [\R(h)]_{ki} + \frac{1}{2} \,  h_{ik}$.
 The last two terms are symmetric in $i$ and $k$, so we get
\begin{align}
	 h_{kn,in}   - h_{ni,kn} = (\dv_f \,h)_{k,i} -  (\dv_f \,h)_{i,k} + h_{kn,i}\, f_n - h_{in,k}\,f_n \, .
\end{align}
Substituting this into \eqr{e:fromH4} gives the lemma.
\end{proof}

 In the remainder of this section, all results will be stated at a point where $g_{ij} = \delta_{ij}$ so that there is no difference between upper and lower indices.

\begin{Cor}	\label{c:cLpr}
If $h' =0$ at $t=0$, then we have at $t=0$ that
\begin{align}
	-\left[ \cL\, h_{ij} \right]'  
	 &= h_{mn} \, h_{ij,mn}
	 + 2\, h_{pj,m}\,  \cC_{mi}^p   + 2\,h_{ip,m}\,  \cC_{mj}^p    +h_{ij,p}\,  \left( k - \frac{1}{2}\,  (\Tr\, h) \right)_p \notag \\
	  &
	   +  \frac{1}{2} \, h_{mj} \, \cL\, h_{mi}     +  \frac{1}{2} \, h_{mi} \, \cL\, h_{mj}   + h_{ij,p} \, (\dv_f h)_p \notag \\
		 &+ \frac{1}{2} \left(   h_{mi}  [ (\dv_f \,h)_{m,j} -  (\dv_f \,h)_{j,m} ]     +  h_{mj} [ (\dv_f \,h)_{m,i} -  (\dv_f \,h)_{i,m}      ]
		 \right)
   \, .   \notag
\end{align}
If, in addition, $h=u\, g^1$ and $k = \frac{\ell}{2} \, u$ where $u$ depends only on $\RR^{n-\ell}$, then
\begin{align}
	\left[ \cL\, h  \right]'  
	 &=  -     [ 2\,|\nabla u|^2 + u \, \cL\, u    ] \, g^1  \, .
\end{align}
\end{Cor}

\begin{proof} 
We will work at a point where $g_{ij} = \delta_{ij}$ and $e_c (g_{ab})=0$ at $t=0$.
By definition, we have $[\Delta h](e_i , e_j) = g^{mn} \, h_{ij,mn}$ and, thus, that at $t=0$ at this point
\begin{align}
		[[\Delta h]']_{ij} &= -h_{mn} \, h_{ij,mn} +   (h_{ij,mm} )' 
		   \, .
\end{align}
Since $h'=0$, Lemma \ref{l:hesstenp}
gives that
\begin{align}	 \label{e:1114}
	  - \left( h_{ij,mm} \right)' &= h_{pj,m}\,  \cC_{mi}^p   + h_{ip,m}\,  \cC_{mj}^p    +h_{ij,p}\,  \cC_{mm}^p  
	  + \cC_{mi}^p\, h_{pj,m}  + h_{nj} \,  ([\nabla_{e_m} \nabla_{e_m} e_i ]')_n    \notag    \\
	  &+h_{ip,m}\,  \cC_{mj}^p     + h_{in}\, ([\nabla_{e_m} \nabla_{e_m} e_j]' )_n    \notag \\
	  &= 2\, h_{pj,m}\,  \cC_{mi}^p   + 2\,h_{ip,m}\,  \cC_{mj}^p    +h_{ij,p}\,  \cC_{mm}^p  
	   + h_{nj}\,   ([\nabla_{e_m} \nabla_{e_m} e_i ]')_n          + h_{in}\, ([\nabla_{e_m} \nabla_{e_m} e_j]' )_n \, . 
\end{align}
Lemma \ref{l:lapei} gives that
\begin{align}
	2\, h_{nj}  [ \left( \nabla_{e_m} \nabla_{e_m} e_i \right)' ]_n &=  h_{nj}\, \left\{h_{nm,i} f_m   + h_{ni,mm} -h_{im,n}\, f_m 
	+  (\dv_f \,h)_{n,i} -  (\dv_f \,h)_{i,n}  \right\}  \, .  \label{e:e614}
\end{align}
Using this
 and $\cC_{mm}^p =( \dv_f h)_p + h_{pm}\, f_m - \frac{1}{2}\,  (\Tr\, h)_p$ (by \eqr{e:csone})
 in \eqr{e:1114}  gives
\begin{align}	 
	  - \left( h_{ij,mm} \right)' &= 2\, h_{pj,m}  \cC_{mi}^p   + 2\,h_{ip,m}  \cC_{mj}^p    +h_{ij,p} \,  
	   \left( ( \dv_f h)_p + h_{pm}\, f_m - \frac{1}{2}\,  (\Tr\, h)_p \right)
	    \notag \\
	  &
	   +  \frac{1}{2} \, h_{nj} \left\{h_{nm,i} f_m   + h_{ni,mm} -h_{im,n}\, f_m   +  (\dv_f \,h)_{n,i} -  (\dv_f \,h)_{i,n}    \right\}  \\
	   &+  \frac{1}{2} \, h_{ni} \left\{h_{nm,j}\, f_m   + h_{nj,mm} -h_{jm,n}\, f_m   +  (\dv_f \,h)_{n,j} -  (\dv_f \,h)_{j,n}    \right\}  \, . \notag
\end{align}
For the drift term, we have
\begin{align}
	\left( \nabla_{\nabla f} h_{ij} \right)' &= \left( g^{nm}\, f_n\, h_{ij,m} \right)' = -h_{nm}\, f_n\, h_{ij,m} +   k_m\, h_{ij,m} +   f_n\, (h_{ij,n})'  \, .
\end{align}
Since $h'=0$, Lemma \ref{l:hesstenp}
gives that $(h_{ij,n})' = -h_{mj}\, \cC_{ni}^m  -h_{mi}\, \cC_{nj}^m$, so we get
\begin{align}	 
	\left( \nabla_{\nabla f} h_{ij} \right)' &= -h_{nm}\, f_n\, h_{ij,m} +   k_m\, h_{ij,m} -   f_n\,  \left( h_{mj}\, \cC_{ni}^m  +h_{mi}\, \cC_{nj}^m \right) \, .
\end{align}
Combining, canceling terms and using that $h_{mn,i} - h_{in,m} = 2\, \cC_{in}^m - h_{mi,n}$  gives the first claim.

To get the second claim, we plug in $h = u \, g^1$ and $k = \frac{\ell}{2} \, u = \frac{1}{2} \, \Tr \, h$ into the first claim.  With these choices, $\dv_f \, h = 0$ so the last term on the second line and the entire third line
drop out immediately.  Using \eqr{e:csone},  $\nabla \, g^1 = 0$, and the fact that $g^1$ is nonzero only on the first factor $N$, while $u$ depends only on $\RR^{n-\ell}$, the second claim follows.
\end{proof}

\begin{Cor}	\label{c:dvprh}
If $h' =0$ at $t=0$, then we have at $t=0$ that
\begin{align}
	-\left[ g^{kn}\, h_{jk,ni}   \right]'  
	 &=             h_{kn} \,  h_{jk,ni}
    + h_{pm,m}\, \cC_{ij}^p   + h_{jp,m}\, \cC_{im}^p  +  h_{jm,p}\, \cC_{im}^p   
	  +  h_{pm,i}\, \cC_{mj}^p   \notag \\
	  &+ h_{jp,i}\,  \left( ( \dv_f h)_p + h_{pm}\, f_m - \frac{1}{2} \, (\Tr \, h)_p   \right)  
	       + \frac{h_{pm}}{2} \,  \left\{  h_{pm,ji} + h_{pj,mi} - h_{mj,pi}    \right\} \\
	     &  +  \left\{   ( \dv_f \, h)_{p,i} +h_{pm,i}\, f_m +   h_{pm}\, f_{mi}     - \frac{1}{2} \, (\Tr \, h)_{pi}    \right\} \, h_{jp} \, .  \notag
\end{align}
If, in addition, $h=u\, g^1$ where $u$ depends only on $\RR^{n-\ell}$, then
$
	\left[ g^{kn}\, h_{jk,ni}   \right]'  
	  =   - \frac{\ell}{2}\, \left( u_i\, u_j + u\, u_{ij} \right)$.
\end{Cor}

\begin{proof}
Working at a point as before, 
we have 
	 $ \left[   g^{kn}\,   h_{jk,ni}  \right]'    =   -      h_{kn}\,   h_{jk,ni}  + \left[   h_{jk,ki}  \right]' $, so we must compute
	 $ \left[   h_{jk,ki}  \right]'$.
Since $h'=0$, Lemma \ref{l:hesstenp}
gives that
\begin{align}
	 -\left( h_{jm,mi} \right)'   &=  (\nabla h) ( \cC_{ij}^p\, e_p, e_m , e_m) + (\nabla h) (e_j , \cC_{im}^p\, e_p, e_m) + (\nabla h) (  e_j, e_m,\cC_{im}^p\, e_p)    \notag \\
	  &
	  + (\nabla h)(\cC_{mj}^p\, e_p,  e_m,e_i) + h([\nabla_{e_i} \nabla_{e_m} e_j ]', e_m)        
	  + (\nabla h)( e_j , \cC_{mm}^p\, e_p ,e_i)    + h(   e_j, [\nabla_{e_i} \nabla_{e_m} e_m]' )    \, . 
\end{align}
Lemma \ref{l:H4}
gives that $\left( \nabla_{e_n} \nabla_{e_j} e_i \right)'  = \frac{1}{2} \,  \left\{  h_{pj,in} + h_{pi,jn} - h_{ji,pn}    \right\} \, e_p$, so we have
\begin{align}
		 h([\nabla_{e_i} \nabla_{e_m} e_j ]', e_m)        &=  \frac{h_{pm}}{2} \,  \left\{  h_{pm,ji} + h_{pj,mi} - h_{mj,pi}    \right\}  \, , \\
		h(   e_j, [\nabla_{e_i} \nabla_{e_m} e_m]' )&=   \left\{  ( \dv_f \, h)_{p,i} +  h_{pm,i}\, f_m +   h_{pm}\, f_{mi}     - \frac{1}{2} \, (\Tr \, h)_{pi}    \right\} \, h_{jp} \, . \notag
\end{align}
Using these and the formula for $\cC^p_{mm}$ gives
 the first claim.
If $h= u\,g^1$, then most terms drop out immediately  and the definition \eqr{e:csone}
	for $\cC_{ij}^k$ gives
\begin{align}
	2\, \left[ g^{kn} h_{jk,ni} \right]'  
	 &=         - 2\,(h_{jp,m}+ h_{jm,p}) \cC_{im}^p  
	  - 2\, h_{pm,i}\, \cC_{mj}^p     	       -   h_{pm,ji}    h_{pm}  	        =           
	  -  \ell \, u_i\, u_j    
	     	       -     \ell \, u\, \, u_{ij}
	        \, .  \notag \, \, \, \, \, \, \, \, \, \, \, \, \qedhere
\end{align}
\end{proof}

 \begin{proof}[Proof of Proposition \ref{p:twovar}]
   We will work at a point using coordinates where $g_{ij} = \delta_{ij}$ at $t=0$.  
 Using that $h'=0$ and differentiating
   Lemma \ref{l:phip} at $t=0$ gives
  \begin{align}
  	2\, \phi_{ij}''(0) &= \left[ ( \cL \, h)_{ij} \right]' + 2\, [\R(h)]'   +\left[ \left(   \Tr \, (h) - 2\,k \right)_{ij} \right]' \notag \\
	&-  \left( g^{kn}\, h_{jk,ni} + g^{kn}\, h_{ik,nj} \right)' - \left( [ \Ric_{i}^{k}]' h_{jk} + [\Ric_{j}^{k}]'\, h_{ik}    \right) \label{e:e6p27} \\
	&  +   \left( h_{jn,i} + h_{in,j}   \right)'   f_n  -   \left( h_{jn,i} + h_{in,j}   \right)\, h_{nm}\, f_m  +   \left( h_{jn,i} + h_{in,j}   \right)\, k_n \notag \, .
  \end{align}
We will compute each term next. The third claim in Proposition \ref{p:varyR}
gives at $t=0$ that
\begin{align}
	 2\, [\R(h)]'  &=   2\, [\R_{ikjn} g^{kp}\, g^{nm} \, h_{pm}]' =     2\, \R_{ikjn}'\,  h_{kn} -    2\, \R_{ikjn}\, h_{kp}\,  h_{pn} -   2\, \R_{ikjn}\,  h_{nm} \, h_{km}  \notag \\
	  &= \left\{ \R_{ikj\ell} \, h_{\ell n} -  \R_{ik n\ell } \, h_{\ell j} +  h_{in,jk} - h_{kn,ji} + h_{jk,ni} - h_{ij,nk} \right\} h_{kn} - 4\,  \R_{ikjn} h_{kp}\,  h_{pn} \, .
\end{align}
Since $h= u\,g^1$ where $u$ depends only on $\RR^{n-\ell}$ and 
$\R_{ikjn} = \frac{1}{2(\ell -1)} \,  \left( g^1_{ij} g^1_{kn} - g^1_{in} g^1_{kj} \right)$, we have
\begin{align}
	 2\, [\R(h)]'  &
	 =  \left\{ 2\,u \, \R_{ikjn}    - g^1_{kn}\, u_{ij}    \right\} u\, g^1_{kn} - 4\,  u^2\, \R_{ikjn}\,  g^1_{kn}  = - u^2 \, g^1_{ij} - \ell\, u \, u_{ij}  \, . \notag
\end{align}
The second claim in Proposition \ref{p:varyR}
gives at $t=0$ that
\begin{align}
	2\, \left( \Ric_i^k \right)' &= 2\, \left( \Ric_{ip} g^{kp} \right)' = 2 \, \left( \Ric_{ip}   \right)'\, g^{kp} -   g^1_{ip}\, h_{kp} \notag \\
	  &= - \Delta \, h + \frac{1}{2} \,  h_{in} \, g^1_{kn} + \frac{1}{2}  g^1_{in} \, h_{kn} -2 \, \R (h) - \Hess_{\Tr \, h} + \nabla \dv \, h + ( \nabla \dv\, h )^T -   g^1_{ip}\, h_{kp} \, .
\end{align}
Therefore, since $h=u\,g^1$, we get that
   \begin{align}
   	2\, \left( \Ric_i^k \right)' \, h_{jk} =  - u\,(\Delta \, u) \, g^1   -  \, u^2 \, g^1 \, .
\end{align}
Since $[\Tr\, (h)]' = [ g^{ij}\, h_{ij}]' = - |h|^2$ and $k'=0$, Lemma \ref{l:hessf}
gives that
\begin{align}
	\left( \Hess_{\left(   \Tr \, (h) - 2\,k \right)} \right)' = \Hess_{-|h|^2} - \frac{1}{2} \,   \left(h_{nj,i} + h_{ni,j} - h_{ji,n}  \right)\,  ( \Tr \, (h) - 2\,k)_n \, .
\end{align}
Since  $h= u\,g^1$ and $k= \frac{\ell}{2} \, u$, this becomes
 	$\left( \Hess_{\left(   \Tr \, (h) - 2\,k \right)} \right)' = \Hess_{-|h|^2}  = -  \ell \, \Hess_{u^2}$.
	The  first claim in Lemma \ref{l:hesstenp} and the definition \eqr{e:csone} of
	$\cC_{ij}^k $
give
\begin{align}	 	 \label{e:hijp2}
	 [h_{in,j}]' =-  \frac{1}{2} \,  h_{ip} \,  \left(h_{pj,n} + h_{pn,j} - h_{jn,p}  \right)
	    -  \frac{1}{2} \,  h_{np} \,  \left(h_{pj,i} + h_{pi,j} - h_{ji,p}  \right) \, .
\end{align}
Using that $h=ug^1$ gives
	 $[h_{in,j}]' \, f_n =-  \frac{1}{2} \,  h_{ip} \,  h_{pj,n}    f_n = - \frac{1}{2} \, u u_n f_n \, g^1_{ij} $.
We now use these calculations in \eqr{e:e6p27}, together with Corollary \ref{c:cLpr}
for the $[\cL \, h]'$ term and 
Corollary \ref{c:dvprh} for the first terms in the middle line.
This gives
\begin{align}
  	2\, \phi_{ij}''(0) &= -  [ 2\,|\nabla u|^2 + u\, \cL\, u] \, g^1 - \left( u^2 \, g^1 + \ell \, u \, u_{ij} \right) -   \ell\, \Hess_{u^2} + \ell\, (u_i\, u_j + u\, u_{ij})  \\
	&\quad + (u\, \Delta\, u + u^2) \, g^1 - u\, u_n\,f_n \, g^1  = -   2\,  |\nabla u|^2  \, g^1     - 2\, \ell\, u\, u_{ij} -  \ell\, u_i\, u_j       \, . \notag \, \, \, \, \, \, \, \, \, \, \, \, \qedhere
  \end{align}

\end{proof}

  \section{Second order stability of  $N\times \RR^{n-\ell}$}    \label{s:global}

         In this section $\Sigma=N^{\ell}\times\RR^{n-\ell}$ and $g^1$ is an Einstein metric on $N$ with $\Ric=\frac{1}{2}\,g^1$ and  satisfying ($\star$).   
         Given a nearby metric $\bar{g}+h$ and potential $\bar{f} +k$, let
$u\,g^1$ be the orthogonal projection of $h$ onto $\cK\,g^1$  and write  
\begin{align}
	h= u \, g^1 + \bb {\text{ and }} k=  \frac{\ell}{2} \, u + \psi  \, .
\end{align}
 Bars denote quantities relative to $\bar{g}$; e.g.,   $\overline{\Ric}$ is the Ricci tensor for $\bar{g}$.

The main result of this section shows that $\Sigma$ has a local rigidity: If $(h,k)$ is small, then it can be bounded in terms of the failure $\phi$ to be a gradient shrinking soliton, with two caveats.  First, we need to bound $\dv_{\bar{f}}$ to control the gauge.
Second, 
even if $h=0$, $k$  could be   linear, corresponding to a translation along the axis of $\Sigma$.  To mod out for this, we must bound the ``center of mass'' vector
 \begin{align}	\label{e:centermass}
 	\cB_i (h,k) =  \int x_i \, \left( k - \frac{1}{2} \, \Tr_{\bg} \,  h \right) \, \e^{ - \bar{f}} = 2 \,  \int \langle \partial_{x_i}, 
	\bar{\nabla} \left( k - \frac{1}{2} \, \Tr_{\bg} \,  h \right) \rangle \, \e^{ - \bar{f}} \, .
 \end{align}
 The next theorem uses a first  order Taylor expansion to show that the Jacobi field $u\, g^1$  dominates the error terms $\bb , \psi$ and then uses the second order expansion to estimate $\| u \|_{L^2}$.   
     
     \begin{Thm}	\label{p:firstorder}
There exist $C, \delta > 0$  so that if  $\| h \|_{C^2} + \| \bar{\nabla} k \|_{C^1} \leq \delta$, then   for any $\epsilon>0$
\begin{align}
     	&\|  \bb  \|_{W^{2,2}}^2  + \|  \bar{\nabla} \psi  \|_{W^{1,2}}^2  \leq C \, \left\{ \| \phi  \|_{L^2}^2 + \| \dv_{\bar{f}} \, h \|_{W^{1,2}}^2 +     | \cB (h,k)|^2 +  \| u \|_{ L^2}^{ 4}  \right\}  \, ,	 \notag \\
	& \|u\|^2_{L^2}\leq   C\,\left\{\| u \|_{L^2}^{ 3} +\|\phi \,(1+|x|^2)\|_{L^1} \right\} + C_{\epsilon} \, \left\{ \| \phi  \|_{L^2}^{2-\epsilon}  + | \cB(h,k)|^{2-\epsilon} +\| \dv_{\bar{f}} \, h \|_{W^{1,2}}^{2-\epsilon} \right\}\, . \notag
\end{align}
  \end{Thm}

\vskip1mm
When  $\phi$, 
  $\dv_{\bar{f}} \, h$ and $\cB (h,k)$ vanish globally, we get:

   \begin{Cor}		\label{t:rigidA}
   There exists $\delta > 0$ so that if  $\phi = 0$,  $\dv_{\bar{f}} \, h = 0$, $\cB(h,k)=0$ and $\| h \|_{C^2} + \| \overline{\nabla} {k} \|_{C^1} \leq \delta$, then
   $h=0$ and $k=0$.  
   \end{Cor}
   
   \vskip1mm
In a formal sense, the corollary says that   $(\bar{g} , \bar{f})$ is ``isolated'' as a shrinker once we mod out by the diffeomorphism group (to make $\dv_{\bar{f}} h = 0$) and translations (to make $\cB = 0$).  If we had a similar statement for a compact shrinker, then one could carry this out directly.  

          \begin{proof}[Proof of Corollary \ref{t:rigidA}]
 The second claim in Theorem \ref{p:firstorder} gives that $\| u \|_{L^2}^2 \leq C \, \| u \|_{L^2}^{ 3 }$.  Since $\| u \|_{L^2} \leq \| h \|_{L^2}$, $u$ vanishes if $\| h \|_{C^2}$ is small.  Once $u \equiv 0$,   then the first claim in
 the theorem  gives that $\bb=0$ and $\bar{\nabla} \psi = 0$.  It follows that $\bar{\nabla} k = 0$.  Combining this with
 the normalizations $S + |\nabla f|^2 = f$ and $\bar{S} + |\bar{\nabla} \bar{f}|^2 = \bar{f}$, we conclude that $k=0$.
  \end{proof}
   
   \subsection{Pointwise Taylor expansion of $\phi$} The estimates in this subsection Taylor expand near $\Sigma$ and, as such, assume that $h$, $k$ and $v$ are small at the point where we compute.

      \begin{Lem}		\label{l:append}
 There is a smooth map  $\Psi$    so that  $ {\Ric} = \Psi (h , \bar{\nabla} h , \bar{\nabla} \bar{\nabla} h)$.  Furthermore, 
   $\Hess_{\bar{f}+k} =  \overline{\Hess}_{\bar{f}} + \overline{\Hess}_k  - \left( \Gamma_{ij}^n - \bG_{ij}^n \right) e_n ({\bar{f}}+k) $.
      \end{Lem}
      
      \begin{proof}
     The Christoffel symbols   of $g = \bar{g}+h$ are given by
      \begin{align}
      	\Gamma_{ji}^p &=  \frac{1}{2} \, (\bar{g}+h)^{pm} \left( e_i (\bar{g}+h)_{jm} + e_j (\bar{g}+h)_{mi} - e_m (\bar{g}+h)_{ij} \right) \, .
      \end{align}
    Note that $e_i h_{pm} = h_{pm,i} + \bG_{ip}^n h_{nm} + \bG_{im}^n h_{pn}$ where $h_{pm,i}$ is the covariant derivative of $h$ (with respect to $\bar{g}$).
    Thus,  $\Gamma$ is a smooth function of $h$ and $\bar{\nabla} h$.  
   The curvature tensor ${\R}^i_{jpn}$ of $\bar{g}+h$ is the sum of linear terms   in the derivative of $\Gamma$ and    quadratic terms  in $\Gamma$, giving    the first  claim.
  The last claim follows from $\Hess_v (e_i , e_j) = e_i (e_j (v)) - \Gamma^n_{ij} \, v_n$.
      \end{proof}

Define the one-parameter families of $2$-tensors $H(t)= \Hess_{\bar{f}+t\,k}(\bar{g}+th)$  to be the Hessian of $\bar{f}+t\,k$ computed with respect to the metric $g(t) = \bar{g}+t\,h$; define $\phi (t)$ similarly.

\begin{Lem}	\label{p:Hap}
There exists $C$ so that
\begin{align}
	\left| H(1) - H(0) - H'(0)   \right| &\leq C \,  |\bar{\nabla} h| \left(  |h|\, |\bar{\nabla} \bar{f}|  +  |\bar{\nabla} k| \right) \, , \\
	\left| H(1) - H(0) - H'(0) - \frac{1}{2} \, H''(0) \right| &\leq C \,  |h| \, |\bar{\nabla} h| \left(  |h|\, |\bar{\nabla} \bar{f}|  +  |\bar{\nabla} k| \right) \, .
\end{align}
\end{Lem}

\begin{proof}
Let $(\Gamma^{t} )^{k}_{ij}$ be the Christoffel symbols for the metric $\bar{g}+ t \, h$.  We will bound the $t$ derivatives of $\Gamma^t$ for $t \in [0,1]$. Since the difference of Christoffel symbols is a tensor, we can do this at a point using coordinates where $e_c (g_{ab} )= 0$ and $\bar{\Gamma}=0$, so that
\begin{align}
	2\, (\Gamma^t)^k_{ij} = t\, (\bar{g}+t\,h)^{km} \left( e_i (h_{jm} ) + e_j (h_{mi}) - e_m (h_{ij}) \right) \, .
\end{align}
Differentiating this expression, we see that
\begin{align}
		\left|  \partial_t \, \Gamma^t \right|  \leq C \, |\bar{\nabla} h| \, ,  \, 
	\left|  \partial_t^2 \, \Gamma^t \right|  \leq C \, |h| \, |\bar{\nabla} h| {\text{ and }}	\left| \partial_t^3 \, \Gamma^t \right|  \leq C \, |h|^2 \, |\bar{\nabla} h| \, .  \label{e:e714} 
\end{align}
The last claim in Lemma \ref{l:append} gives  
 $H_{ij} (t) - H_{ij} (0)  =    t\, \overline{\Hess}_k    + \left(  \bar{\Gamma} - \Gamma^t \right)_{ij}^p ( \bar{f}_p + t \, k_p ) $.
     Differentiating gives 
   that $H_{ij}' =  \overline{\Hess}_k    + \left(  \bar{\Gamma} - \Gamma^t \right)_{ij}^p \,  k_p    - \left( \partial_t \Gamma^t \right)_{ij}^p ( \bar{f}_p + t \, k_p )$,
   \begin{align}
   	H_{ij}'' &= - 2\,  \left(  \partial_t \Gamma^t \right)_{ij}^p \,  k_p    - \left( \partial_t^2 \Gamma^t \right)_{ij}^p ( \bar{f}_p + t \, k_p ) \, ,  \label{e:e717} \\
	 H_{ij}''' &= - 3\,  \left(  \partial_t^2 \Gamma^t \right)_{ij}^p \,  k_p    - \left( \partial_t^3 \Gamma^t \right)_{ij}^p ( \bar{f}_p + t \, k_p ) \, .
   \end{align}
Thus, we get 
    	$|H''| \leq C\, |\bar{\nabla} h| (| \bar{\nabla} k| + |h| \, | \bar{\nabla} \bar{f}|) {\text{ and }}
	|H'''| \leq C\,  |h| \,  | \bar{\nabla} h| (|\bar{\nabla} k| + |h| \, |\bar{\nabla} \bar{f}|)$.   
\end{proof}

To keep notation short, set  $[h]_1 = |h| + |\bar{\nabla} h|$ and $[h]_2 = |h| + |\bar{\nabla} h| + |\bar{\nabla}^2 h|$.  

 \begin{Cor}	\label{l:644}
 We have
$	\left|  \phi'(0) + \frac{1}{2} \, \phi''(0) \right| \leq |\phi (1)| + C  \,  [h]_2^3 +  C\,  |h| \, |\bar{\nabla} h| \left(  |h|\, |\bar{\nabla} \bar{f}|  +  |\bar{\nabla} k| \right) $.
 \end{Cor}

 \begin{proof}
   Lemma \ref{l:append} and  the chain rule
give  $\left|  \Ric(1)  -  \Ric(0) -  \Ric'(0)  - \frac{1}{2}\,  \Ric''(0)  \right|
	\leq C \, [h]_2^3$.
	 Combining this with the second bound in Lemma \ref{p:Hap} gives the claim.
 \end{proof}

\vskip1mm
The next proposition writes $\phi''(0)$ as a term that is quadratic in $u$ (and its derivatives) and an error term that is higher order
($\bb$ and $\bar{\nabla}\psi$ will be shown to be smaller than $u$).

\begin{Pro}	\label{p:2varphi}
There exists $C$ so that 
if $h= u g^1 +\bb$ and $  k= \frac{\ell}{2} \, u + \psi$ where $u$ depends only on $\RR^{n-\ell}$, then 
\begin{align}
  	\left| 2\, \phi_{ij}''(0) +  2\,  |\bar{\nabla} u|^2 \, g^1     + 2\, \ell \, u \, u_{ij} +  \ell \, u_i \, u_j  \right| &\leq C \, [u]_2 [\bb]_2 + C \, [\bb]_2^2 +
	C \,  |\bar{\nabla} \bar{f}| \, ([u]_1[\bb]_1 + [\bb]_1^2) \notag \\
	&\qquad + C \, |\bar{\nabla} \psi| ( |\bar{\nabla} u| + |\bar{\nabla} \bb|)  \, .
  \end{align} 
\end{Pro}

\begin{proof}
We divide $\phi (t)$ into two pieces, $\phi_0 (t) = \frac{1}{2} \, (g +t\,h)- \Ric_{g+t\,h}$ and the Hessian part $H(t)$.  Similarly, let $\phi_u (t) = \phi_{u,0} (t) - H_{u} (t)$ be the variation of $\phi$ in the direction $(u\,g^1 , \frac{\ell}{2} \, u)$.  
Proposition \ref{p:twovar}
gives 
 \begin{align}	\label{e:twovr}
  	2\, (\phi_u)_{ij}''(0) &=   -   2\,  |\bar{\nabla} u|^2  \, g^1     - 2\, \ell \, u  \, u_{ij} -  \ell\,  u_i \, u_j      \, .
  \end{align} 
By  Lemma \ref{l:append} and the chain rule, $\phi_0'' (0)$ is quadratic  in $(h, \bar{\nabla} h , \bar{\nabla}^2 h)$ and, thus, 
  \begin{align}	\label{e:ggp}
  	\left| \phi_0'' (0) - \phi_{u,0}'' (0) \right| &\leq C \, [u]_2 \, [\bb]_2 + C \, [\bb]_2^2  \, .
  \end{align}
  On the other hand, \eqr{e:e717} plus  \eqr{e:e714}   imply that
    \begin{align} 
  	\left| H'' (0) - H_{u}'' (0)  \right| &\leq C \, \left[ |u|\,  |\bar{\nabla} \bb|  + |\bb|    \left(  |\bar{\nabla} u| +  |\bar{\nabla} \bb| \right) \right] |\bar{\nabla} \bar{f}|
	+ C \,  |\bar{\nabla} \psi| (|\bar{\nabla} u| + |\bar{\nabla} \bb|) + C \, |\bar{\nabla } u| \, |\bar{\nabla} \bb| \, . \notag
  \end{align}
  The proposition follows by combining this with \eqr{e:twovr} and \eqr{e:ggp}.
  \end{proof}

 \subsection{Integral estimates} We turn now to integral estimates.  Suppose that $h$, $k$ and $u$ are as in Theorem \ref{p:firstorder}.  Even though $h$ and $k$ are small,  $u$ grows quadratically so the Taylor expansion is not valid for $x$ large.   The next lemma gives an integral bound for $u$ in terms of $\phi$
 and ``error terms'' that are higher order.
   
  \begin{Lem}	\label{l:riesz}
We have  		$ \| u \|_{L^2}^2 \leq    C \, \left\|   \left\{ |\phi(1)| + [h]_2^3     + [\bb]_2^2 + |\bar{\nabla} \psi |^2 \right\} 
	  (1+|x|^2) \right\|_{L^1} + C \, \| u \|_{L^2}^4$.
\end{Lem} 
 
 \begin{proof}
 We can assume that $\| h \|_{L^2}$ and $\| u \|_{L^2}$  are fixed small; we will use this freely below.
 Lemma \ref{l:polygrow} gives that 
  $[u]_2 \leq C\, (1+|x|^2) \, \| u\|_{L^2}$, so $u$ remains small as long as $|x|^2 \leq \frac{c}{\| u \|_{L^2}}$.
 Let $\eta \geq 0$ be a cutoff function that is supported on the set  
 $|x|^2 \leq \frac{c}{\| u \|_{L^2}}$ and that depends only on $\RR^{n - \ell}$.

  \vskip1mm
  \noindent
  {\bf{Step 1: Setting it up}}. 
     By \eqr{c:phip}, the first variation of $\phi$ in a direction $(h,k)$ is given by
  \begin{align}	\label{e:fvphihkbra}
  	\phi' (0) &=   \frac{1}{2} \, L \, h   + \overline{\Hess}_{ \left( \frac{1}{2} \, \Tr \, (h) - k \right) } + 
	  \dv_{\bar{f}}^{*}\,  \dv_{\bar{f}} \, h \, .
  \end{align}
  To simplify the equations, let $\cE$ denote the point-wise error function
  \begin{align}
  	\cE \equiv  [u]_2 [\bb]_2 + [\bb]_2^2 +
	 |x| \, ( [u]_1[\bb]_1  +[ \bb]_1^2) +  |\bar{\nabla} \psi| ( |\bar{\nabla} u| + |\bar{\nabla} \bb|) \, .	\label{e:subce}
  \end{align}
  With this notation,  Proposition \ref{p:2varphi} gives
 $C$ so that on the support of $\eta$ 
 \begin{align}	\label{e:myfavprop}
  	\left| 2\, \phi''(0) +  2\,  |\bar{\nabla} u|^2 \, g^1     + 2\, \ell \, u \, u_{ij} +  \ell \, u_i \, u_j  \right| &\leq C\, \cE \, .  
  \end{align} 
   By Lemma \ref{l:polygrow}, $v=|\bar{\nabla} u|^2-\bar{\Delta}\,|\bar{\nabla} u|^2 \in \cK$, 
 $ |v|  \leq  C \, (1+ |x|^2) \, \| u \|_{L^2}^2$ and $ |\bar{\nabla}v|  \leq  C \, (1+ |x|) \, \| u \|_{L^2}^2$; we will use these freely below.

       Note that $\eta \, v \, g^1$ is point-wise orthogonal to $u \, u_{ij}$ and $u_i \, u_j$.  Since $\dv_{\bar{f}} \, (\eta \, v \, g^1) = 0$, it is $L^2$-orthogonal to $ \overline{\Hess}_{ \left( \frac{1}{2} \, \Tr \, (h) - k\right) }$ and 
	  $\dv_{\bar{f}}^{*}\,  \dv_{\bar{f}} \, h$.     
	  Thus, 
	  taking the $L^2$ inner product of $\phi'(0)+\frac{1}{2}\,\phi''(0)$ with $\eta \, v \,g^1$ and using
	  \eqr{e:fvphihkbra} and \eqr{e:myfavprop} gives
 \begin{align}   \label{e:almostthere}
 	\left| \int \langle \phi'(0)+\frac{1}{2}\,\phi''(0), \eta \, v\,g^1\rangle\,\e^{-{\bar{f}}} + \frac{\ell}{2}\int \eta \,  |\bar{\nabla} u|^2\,v\, \e^{-\bar{f}} \right| & \leq 
	 C\,  \int   \cE \, |v|  \, \e^{-\bar{f}} \notag \\
	&\quad  + \frac{1}{2} \, \left| \int  \langle L \, h , \eta \, v \, g^1 \rangle \, \e^{-\bar{f}} \right| \, .
 \end{align}
  Since $L$ is symmetric,  $L \, (v \, g^1) =0$ and  $L \, h = L \, \bb$, we see that
  \begin{align}	
  	\left| \int  \langle L \, h , \eta \, v \, g^1 \rangle \, \e^{-\bar{f}} \right|  &= \left| \int  \langle  \bb , L \, (\eta \, v \, g^1) \rangle \, \e^{-\bar{f}} \right| 
	 = \left| \int  \langle  \bb ,  ((\cL \, \eta) \, v + 2\, \langle \bar{\nabla} \eta , \bar{\nabla} v \rangle ) \, g^1  \rangle \, \e^{-\bar{f}} \right| \notag \\
	 &\leq C\, \| u \|_{L^2}^2 \, \int  |\bb| \, \left(  |\cL \, \eta| \, (1+ |x|^2) + |\bar{\nabla} \eta| \, (1+|x|)
	 \right) \, \e^{ - \bar{f}}  \, .
  \end{align}
  Using this in \eqr{e:almostthere}, we see that
   \begin{align}  \label{e:needthistoo}
 	 \frac{\ell}{2} \, \left| \int    \eta \,  |\bar{\nabla} u|^2\,v\, \e^{-\bar{f}} \right| & \leq 
	C\, \| u \|_{L^2}^2 \, \left| \int| \phi'(0)+\frac{1}{2}\,\phi''(0)| \,  \eta \,  (1+|x|^2) \rangle\,\e^{-{\bar{f}}}   \right|  + 
	 C\, \| u \|_{L^2}^2 \int   \cE \,  (1+|x|^2)  \, \e^{-\bar{f}} \notag \\
	&\quad  +  C\, \| u \|_{L^2}^2 \, \int  |\bb| \, \left(  |\cL \, \eta| \, (1+ |x|^2) + |\bar{\nabla} \eta| \, (1+|x|)
	 \right) \, \e^{ - \bar{f}} \, .
 \end{align}

 \vskip1mm
 \noindent
 {\bf{Step 2: An upper bound for $\| u \|_{L^2}^4$}}.  We will show  first that there  is a  constant  $c_1  > 0$ so that
 \begin{align}
 	\int \eta \, v \, |\bar{\nabla} u|^2 \, \e^{ - \bar{f}} \geq \frac{1}{2} \, \int v^2 - c_1 \, \| u \|_{L^2}^4 \, 
	 \int | \bar{\nabla} \eta | \, (1+|x|) \e^{ - \bar{f}}   \, . \label{e:old517}
 \end{align}
 Set $a = \bar{\Delta} |\bar{\nabla} u|^2$ (this is constant since $u$ is quadratic) and note that $|a| \leq c_2 \, \| u \|_{L^2}^2$ by 
   Lemma \ref{l:polygrow}.  Using that $|\bar{\nabla} u|^2 = v  + a$, we have that
   \begin{align}
   	\int \eta \, v \, |\bar{\nabla} u|^2 \, \e^{ - \bar{f}}  = \int \eta \, v^2 \, \e^{ - \bar{f}} + a \, \int \eta \, v \,   \e^{ - \bar{f}} \geq 
	\frac{1}{2} \, \int v^2 \, \e^{ - \bar{f}} + a \, \int \eta \, v \,   \e^{ - \bar{f}} \, , 
   \end{align}
 where the inequality used the concentration inequality from 
 Lemma \ref{l:L2RPa}. Using the bound for $a$, the equation $\cL \, v = - v$, and integration by parts, we see that
 \begin{align}
 	\left| a \, \int \eta \, v \,   \e^{ - \bar{f}}  \right| &\leq c_2 \,  \| u \|_{L^2}^2 \, \left| \int \langle \bar{\nabla} \eta , \bar{\nabla} v \rangle\e^{ - \bar{f}}  \right|  \leq c_2^2 \,  \| u \|_{L^2}^4 \,   \int | \bar{\nabla} \eta | \, (1+|x|) \e^{ - \bar{f}}  
	\, .
 \end{align}
 where the last inequality  used that $|\bar{\nabla} v| \leq c_2 \, (1+|x|) \, \| u \|_{L^2}^2$ by Lemma \ref{l:polygrow}.  This gives the claim 
 \eqr{e:old517}.  The last claim in  Lemma \ref{l:polygrow} gives that $\| v \|_{L^2}^2 \geq c_3 \, \| u \|_{L^2}^4$ for $c_3 > 0$.  Combining this with 
  \eqr{e:old517} gives that
   \begin{align}
 	\int \eta \, v \, |\bar{\nabla} u|^2 \, \e^{ - \bar{f}} \geq \frac{c_3}{2} \,   \| u \|_{L^2}^4  -   c_1 \, \| u \|_{L^2}^4 \, 
	 \int | \bar{\nabla} \eta | \, (1+|x|) \e^{ - \bar{f}}    \, . 
 \end{align}
 As long as $\| u \|_{L^2}$ is sufficiently small, we can cut off $\eta$ far enough out to arrange that
 \begin{align}
 \frac{c_3}{2}    -   c_1 \,  
	 \int | \bar{\nabla} \eta | \, (1+|x|) \e^{ - \bar{f}}  \geq \frac{c_3}{4} \, ,
 \end{align}
 so we conclude that
   \begin{align}	\label{e:cccT}
 	\int \eta \, v \, |\bar{\nabla} u|^2 \, \e^{ - \bar{f}} \geq \frac{c_3}{4} \,   \| u \|_{L^2}^4     \, . 
 \end{align}
 
  \vskip1mm
 \noindent
 {\bf{Step 3: Completing the argument}}. Combining \eqr{e:needthistoo},  \eqr{e:cccT} and Corollary \ref{l:644} gives
 \begin{align}
 	 \| u \|_{L^2}^4  & \leq C\, \| u \|_{L^2}^2 \,  \|  |\phi (1)| +   [h]_2^3 +    [h]_1^2 \left(  |h|\, |x|  +  |\bar{\nabla} k| \right) \,  (1+|x|^2) \|_{L^1} \notag \\
	 &+ 
	 C\, \| u \|_{L^2}^2\, \| \cE \,  (1+|x|^2) \|_{L^1}   +  C\, \| u \|_{L^2}^2 \,  \|   |\bb| \, \left(  |\cL \, \eta| \, (1+ |x|^2) + |\bar{\nabla} \eta| \, (1+|x|)
	 \right) \|_{L^1} \, .
 \end{align}
 Dividing through by $\| u \|_{L^2}^2$, using that $|\bar{\nabla} k| \leq |\bar{\nabla} \psi | + \frac{\ell}{2} \, |\bar{\nabla} u|$, and dividing up the term on the first line gives
  \begin{align}
 	 \| u \|_{L^2}^2  & \leq C\,   \|  |\phi (1)| +   [h]_2^3 +    [h]_1^2  \,  |\bar{\nabla} \psi |  \,  (1+|x|^2) \|_{L^1}
	  + C\,   \|     [h]_1^3\, |x|^3    \|_{L^1}
	 + C\,   \|     [h]_1^2 \,   |\bar{\nabla} u|  \,  (1+|x|^2) \|_{L^1} \notag \\
	 &+ 
	 C\,  \| \cE \,  (1+|x|^2) \|_{L^1}   +  C\,    \|   |\bb| \, \left(  |\cL \, \eta| \, (1+ |x|^2) + |\bar{\nabla} \eta| \, (1+|x|)
	 \right) \|_{L^1} \, . \notag
 \end{align}
 To complete the proof, we will explain why each of the five terms on the right is bounded by  $ C \, \left\|   \left\{ |\phi(1)| + [h]_2^3     + [\bb]_2^2 + |\bar{\nabla} \psi |^2 \right\} 
	  (1+|x|^2) \right\|_{L^1} + C \, \| u \|_{L^2}^4$.  This is clear for the first term (use a Cauchy inequality on the $ |\bar{\nabla} \psi | $ term).  The second term has a $|x|^3$ in it (and we want at most $|x|^2$), but we use the gaussian weighted Poincar\'e inequality (Lemma \ref{l:L2RPa}) to reduce the power of $|x|$ at the cost of an additional derivative to get it in the right form. The third term follows by using an absorbing inequality  (and   Lemma \ref{l:polygrow}).  The fourth term follows in the same way{\footnote{Many of the terms in $\cE$ are already of the right form; the term with an extra $|x|$ is dealt with as above.}}.  For the last term, we use a Cauchy inequality to get an $\| |\bb|^2 \|_{L^1}$ term (which is of the first form) plus a weighted integral where the integrand vanishes where $\eta$ is constant.  Since the support of $\nabla \eta$ is on the scale of $\| u \|_{L^2}^{ - \frac{1}{2}}$, this integral is bounded by a constant time $\| u \|_{L^2}^4$ (in fact, we could have taken any power here since exponentials dominate polynomials).
This completes the proof.
   \end{proof}

   We will use the following Poincar\'e inequality:
   
   \begin{Lem}	\label{l:laterP}
There exists $C$ so if $V \in W^{1,2}$ is a vector field on $\Sigma$, then 
 $\| V (1+|\bar{\nabla} \bar{f}|) \|_{L^2} \leq C \,  \| \bar{\nabla} V \|_{L^2} + C\, \sum_{i=1}^{n-\ell} \left| \int \langle \partial_{x_i} , V \rangle \, \e^{ - \bar{f}}
 \right| $.
\end{Lem}

\begin{proof} 
Let $T = \sum a_i \, \partial_{x_i}$ be the constant $\RR^{n-\ell}$ vector field with $\int \langle \partial_{x_i} , V-T \rangle \, \e^{ - \bar{f}} = 0$.  Using Lemma \ref{l:poinS}  to control the projection to
 $N$ and the  Poincar\'e inequality on $N\times \RR^{n-\ell}$ to control the Euclidean part of $V$, we get that
 $\| V-T \|_{L^2} \leq C \, \| \nabla V \|_{L^2}$.  
 Combining this with Lemma \ref{l:L2RPa} gives the claim.
 \end{proof}

  \begin{Lem}	\label{l:ply}
  Given $m$, $\epsilon \in (0,1/2)$ and $p,q > 0$, there exists $c = c(m,p,q,\epsilon)$ so that if $\eta$ is any function on $\RR^m$ with $|\eta | \leq 1 + |x|^q$, then
$	\int \eta^2 \, |x|^p \,  \e^{ - \frac{|x|^2}{4} } \leq c \, \| \eta \|_{L^2}^{2-\epsilon}$.
  \end{Lem}

\begin{proof}
For $\epsilon \in (0,1/2)$, we have $\eta^{\epsilon} \leq 1+ |x|^q$ and, thus the H\"older inequality gives
 \begin{align}
 	\|  \eta^2 \, |x|^p \|_{L^1} \leq \| \eta^{2-\epsilon} \, (1+ |x|^q) \, |x|^p\|_{L^1} \leq \| \eta^{2-\epsilon} \|_{L^{ \frac{2}{2-\epsilon}}} \, 
	\| (1+ |x|^q) \, |x|^p  \|_{L^{ \frac{2}{\epsilon}}} = c_{m,p,q,\epsilon} \, \| \eta \|_{L^2}^{2-\epsilon}
	\, .  \notag \, \, \, \, \, \, \, \, \, \, \, \, \qedhere
\end{align}
\end{proof}

\begin{proof}[Proof of Theorem \ref{p:firstorder}]
 By Lemma \ref{l:append} and  the chain rule, 
 	$\left|  \Ric(1)  -  \Ric(0) -  \Ric'(0)    \right|
	\leq C \, [h]_2^2 $.
Combining this with the first bound in Lemma \ref{p:Hap}, we get 
\begin{align}
	\left| \phi (1) - \phi (0) - \phi'(0) \right| \leq C  \left( [h]_2^2 + |\bar{\nabla} h|\, |h|\, |\bar{\nabla} \bar{f}| + |\bar{\nabla} h|\, |\bar{\nabla} k|  \right) \, .
\end{align}
Using that $\phi (0) = 0$ and $\phi'(0)$ is given by  
\eqr{c:phip}, we get that
\begin{align}	\label{e:Lhbound0}
	\left|   \frac{1}{2} \, L \, h+ \overline{\Hess}_w \right| \leq |\phi(1)| + \left|  \dv_{\bar{f}}^{*}\,  \dv_{\bar{f}} \, h \right| +  C  \left([h]_2^2 + |\bar{\nabla} h|\, |h|\, |\bar{\nabla} f| + |\bar{\nabla} h|\, |\bar{\nabla} k|  \right) \, , 
\end{align}
where $w = \frac{1}{2} \, \Tr \, h - k$.  Subtract a  linear function from $w$ to get $\bar{w}$ with $\int  \bar{w} \, \e^{-\bar{f}} = \int x_i \,  \bar{w} \, \e^{-\bar{f}} =0$. Obviously, $\overline{\Hess}_w = \overline{\Hess}_{\bar{w}}$.  Self-adjointness of $L$ and Corollary  \ref{t:bochHess} give
\begin{align}
	\int \langle L \, h ,  \overline{\Hess}_w \rangle \, \e^{-\bar{f}} &= \int \langle  h , L \, \overline{\Hess}_{\bar{w}} \rangle \, \e^{-\bar{f}}
	= \int \langle   h ,  \overline{\Hess}_{(\cL +1)\, \bar{w}} \rangle \, \e^{-\bar{f}} \notag \\
	&=  - \int \langle  \dv_{\bar{f}} \,  h ,   \bar{\nabla} (\cL +1)\, \bar{w}  \rangle \, \e^{-\bar{f}} = \int \dv_{\bar{f}}\, (  \dv_{\bar{f}} \,  h ) \,      (\cL +1)\, \bar{w}    \, \e^{-\bar{f}} \, .
\end{align}
   Putting the last two equations together, we get that
\begin{align}	\label{e:Lhbounda}
	  \frac{1}{4}\, \|L\, h\|_{L^2}^2 +\| \overline{\Hess}_w \|_{L^2}^2 &\leq 2\, \|\phi (1)\|_{L^2}^2   + 2\, \| \bar{\nabla} \,  \dv_{\bar{f}}\, h \|_{L^2}^2 + \|  (\cL+1)\, \bar{w} \|_{L^2} \, \| \dv_{\bar{f}}\, (  \dv_{\bar{f }} \,  h ) \|_{L^2}  \notag \\
	  &+
	C \int \left(  [h]_2^4 + |\bar{\nabla} h|^2\, |h|^2\, |\bar{\nabla} f|^2 + |\bar{\nabla} h|^2\, |\bar{\nabla} k|^2  \right) \, \e^{-\bar{f}} \, .
\end{align}
The second term on the last line is bounded by the first by Lemma \ref{l:L2RPa}, while the third term is bounded the first and $ |\bar{\nabla} h|^2 \, |\bar{\nabla}w|^2$.
    Lemma \ref{l:laterP} gives
\begin{align}	\label{e:fromlaterP}
	\|  \bar{w} \|_{W{2,2}}     \leq C   \,    \| \overline{\Hess}_w \|_{L^2}  {\text{ and }}  \|   \bar{\nabla} {w} \|_{W^{1,2}}     \leq C   \, \left(  \| \overline{\Hess}_w \|_{L^2} + |\cB (h,k)| \right)  \, .
\end{align}
We use \eqr{e:fromlaterP}
 and  the absorbing inequality to bound the $ \|  ( \cL + 1)\, \bar{w} \|_{L^2} \, \| \dv_{\bar{f}}\, (  \dv_{\bar{f }} \,  h ) \|_{L^2}$ term by a small constant times $\| \overline{\Hess}_w \|_{L^2}^2$ plus a multiple of
$ \|      \dv_{\bar{f}} \,  h   \|_{ W^{1,2}}^2$.    Thus, we get
\begin{align}
	 \|L\, h\|_{L^2}^2 +\| \overline{\nabla} w \|_{W^{1,2}}^2  & \leq C \,\left\{ \|\phi (1)\|_{L^2}^2  +  \| \dv_{\bar{f}}\, h \|_{W^{1,2}}^2 +
	 |\cB (h,k)|^2 \right\} +
	C \int    ([h]_2^4 + |\bar{\nabla} h|^2 \, |\bar{\nabla} w|^2 ) \, \e^{-\bar{f} }  \, .   \notag
\end{align}
As long as $\sup |\bar{\nabla} h| \leq \delta_0$ for some $\delta_0 > 0$ small enough (depending on $n$), \eqr{e:fromlaterP} allows us to  absorb the
$\, \sup |\bar{\nabla}  h|^2 \, \int |\bar{\nabla} w|^2  \, \e^{-\bar{f} }$ term on the left to get 
\begin{align}	\label{e:Lhbound}
	 \|L\, h\|_{L^2}^2 +\| \overline{\nabla} w \|_{W^{1,2}}^2 & \leq C \,\|\phi (1)\|_{L^2}^2  + C \, \| \dv_{\bar{f}}\, h \|_{W^{1,2}}^2 + C\, | \cB (h,k)|^2 +
	C \int    [h]_2^4   \, \e^{-\bar{f} } \, .  
\end{align}
Since $N$ satisfies ($\star$), Theorem \ref{t:approx}
gives $C$   so that
   \begin{align}	\label{e:fromv}
   	\| \bb \|_{W^{2,2}}^2    = \|h - u\, g^1 \|_{W^{2,2}}^2      \leq C \, \| L \, h \|_{L^2}^2 + C \, \| \dv_{\bar{f}}\, h \|_{L^2}^2  \, .
   \end{align}
Combining this with \eqr{e:Lhbound} and using that $\psi = \frac{1}{2} \, \Tr \, \bb - w$ gives
     \begin{align}	\label{e:bb2a}
    	\| \bb \|_{W^{2,2}}^2 +  \| {\bar{\nabla}} \, \psi \|_{W^{1,2}}^2\leq  C \, \| \dv_f\, h \|_{W^{1,2}}^2 + C \,\|\phi (1)\|_{L^2}^2  + C\, | \cB (h,k)|^2 +
	 C \, \int   [h]_2^4   \, \e^{-\bar{f}}  \, .
    \end{align}
     We still need to get better bounds on the $[h]_2^4$   term. 
      Lemma \ref{l:polygrow} gives a constant $C_n$ so that
   \begin{align}	\label{e:vbounds}
   	  [u]_2 \leq C_n  \,  \| u \|_{L^2}  \, (1+|x|^2) \leq C_n  \,  \| h \|_{L^2}  \, (1+|x|^2)  \, .
   \end{align}
  The triangle inequality $[h]_2 \leq [\bb]_2 + [u\, g^1]_2$, the absorbing inequality, and 
      \eqr{e:vbounds} give
  \begin{align}
\int [h]_2^4 \, \e^{-\bar{f} }  &\leq  C\,    \| [h]_2 \, [u]_2 \|_{L^2}^2  + C \, \| [h]_2 \, [\bb]_2 \|_{L^2}^2  \leq \frac{1}{2} \,\int [h]_2^4 \, \e^{-\bar{f}} + C \, \| u \|_{L^2}^4 + C \, (\sup [h]_2^2) \, \| \bb \|_{W^{2,2}}^2   \, .
		\notag
  \end{align}
  As long as $[h]_2$ is small, we can use this in \eqr{e:bb2a} and absorb the last term on the right to replace $\int [h]_2^4 \, \e^{-\bar{f} } $ with $\| u \|_{L^2}^4$.   This completes the proof of the first claim.

    We turn now to the second claim.  For this, we will use an elementary inequality using that $[h]_2 \leq 1$
    and $[h]_2 \leq [\bb]_2 +[u]_2$
       \begin{align}
    	[h]_2^3 \leq 2\, [h]_2 ([\bb]_2^2 +  [u]_2^2) \leq 2 \, [\bb]_2^2 + 2 [\bb]_2 \, [u]_2^2 + [u]_2^3 \leq 3\, [\bb]_2^2 +  [u]_2^3 + [u]_2^4  \, .
\end{align}
 Using this in 
  Lemma \ref{l:riesz}    gives
\begin{align}	
  		 \| u \|_{L^2}^2 \leq      C \, \|\phi (1)\,(1+|x|^2)\|_{L^1}+C\,\left\|   \left\{ [u]_2^4 + [u]_2^3  + [\bb]_2^2 + | \bar{\nabla} \psi |^2 \right\} 
	  (1+|x|^2) \right\|_{L^1} + C \, \| u \|_{L^2}^4
  \, . \notag
\end{align}
 Lemma \ref{l:polygrow} gives
 that    $\|( [u]_2^4 + [u]_2^3)\,(1+|x|^2)\|_{L^1}\leq C\,(\|u\|_{L^2}^3 + \| u\|_{L^2}^4)$.   From this,  the first claim
 and  Lemma \ref{l:ply},
 we get  $C$ and $C_{\epsilon}$  
  \begin{align}
  \| u \|_{L^2}^2  \leq C\, ( \| u \|_{L^2}^{ 3} + \| u \|_{L^2}^4) + C_{\epsilon} \, \left\{\| \phi (1) \|_{L^2}^{2-\epsilon} + |\cB (h,k)|^{2-\epsilon} + \| \dv_{\bar{f}} \, h \|_{W^{1,2}}^{2-\epsilon} \right\}\, .   \, \, \, \, \, \, \, \, \, \, \, \qedhere
  \end{align}
   \end{proof}

 \section{The action of the diffeomorphism group} \label{s:section7}
 
The main result of this section is the following ``improvement'' estimate, proving that a shrinker which is close to a model on some large scale is even closer on smaller scales:

 \begin{Thm}	\label{t:improve}
  Given $\theta < 1$, there exists $R_1$ so that  if ($\dagger_R$) and $R> R_1$, then ($\star_{\theta \, R}$) holds.
 \end{Thm}
 
 Theorem \ref{t:improve} is the last ingredient needed to prove the strong rigidity Theorem \ref{t:main}.  Before doing so, we will state
 a more general result (note that $\SS^{\ell}$ satisfies ($\star$) in Section \ref{s:jacobi}):
 
 \begin{Thm}	\label{t:rigidP}
Let $N^{\ell}$ satisfy ($\star$) in Section \ref{s:jacobi} and let $\Sigma = N  \times \RR^{n-\ell}$ be a shrinker with potential $f_{\Sigma}= \frac{|x|^2}{4} + \frac{\ell}{2}$.  There exists an $R=R(n)$ such that if $(M^n,g,f)$ is another shrinker and $\{f_{\Sigma}\leq R\}\cap \Sigma$ is close to $\{f\leq R\}\subset M$ in the smooth topology and $f_{\Sigma}$ and $f$ are close on this set, then $(M,g,f)$ is 
identical to $\Sigma$ after a diffeomorphism.
\end{Thm}

  \begin{proof}[Proof of Theorems \ref{t:main}, \ref{t:rigidP}]
Repeatedly applying Theorems \ref{t:extend} and \ref{t:improve} gives 
 maps $\Psi_{R_i}$ satisfying ($\dagger_{R_i}$)
  with $R_i \to \infty$.  The maps are uniformly Lipschitz on compact subsets since the $\Psi_{R_i}$'s are almost isometries and, 
  since $f$ and $\bar{f}$ are proper, the Arzela-Ascoli theorem gives a uniformly convergent subsequence and a limiting proper map $\Psi$.  
  As $R_i \to \infty$, the Lipschitz constants go to one and we conclude that $\Psi$ preserves both the metric and the potential.
  \end{proof}
  
  The challenge for proving Theorem \ref{t:improve} is that Theorem \ref{p:firstorder} requires  bounds on $\dv_{\bar{f}} \, h$ and $\cB (h,k)$ that are stronger than what comes out of
  ($\dagger_R$).  This is a gauge problem: these quantities are only small in the right coordinates, and this is true even if the shrinker is isometric to the model $\Sigma$.  
   We will use  $\cP$  to find the right coordinates in Proposition \ref{p:solvegauge}.
  
  \subsection{The gauge problem}

Given a vector field $V$ with compact support, define a diffeomorphism $\Phi (x) = \Phi_V (x) $ by 
\begin{align}
	\Phi  (x) =   \gamma (x , V(x) ) \, ,
\end{align}
where  $ \gamma (x , \zeta ) = \exp_x  \zeta $ is the exponential map for $x \in M$ and  $\zeta \in T_xM$.{\footnote{The corresponding map on Euclidean space is just $x \to x + V(x)$.}} This is well-defined as long as $|V| \leq \delta_0$ where $\delta_0 > 0$ depends on the closed manifold $N$.
We will often use $y=y_V$ as coordinates for $\Phi (x)$.
We assume that $h,k$ satisfy on $b<R$
\begin{align}	\label{e:3p2}
	\| h \|_{C^{4,\alpha}} + \| k \|_{C^{4,\alpha}} \leq \e^{ - a_0 \, R^2 } \, , 
\end{align}
where  $ a_0 > 0$ is given by Theorem \ref{t:extend}.

The map $\Phi$  gives a new metric  $ \Phi_V^* (\bar{g} + h)$ and, thus, a new ``metric perturbation'' $\bar{h}$ and 
 ``measure perturbation'' $\bar{k}$ 
\begin{align}
	\bar{h} =  \Phi_V^* (\bar{g} + h) - \bar{g} {\text{ and }}  \bar{k} = (\bar{f} + k) \circ \Phi - \bar{f} \, .
\end{align}
Define a mapping $J = J(h,V)$   by
\begin{align}
	J (V) = \frac{1}{2} \, \dv_{\bar{f}} \left(   \bar{h} + 2 \, \dv_{\bar{f}}^* V \right) \, ,
\end{align}
where we added $2 \, \dv_{\bar{f}}^* V$ to cancel  the linearization in $V$ at $h=0$.   We will also need to track the ``center of mass''   $\cB (V) = \cB(h,k,V) \in \RR^{n-\ell}$ given by
\begin{align}
	\cB^i (V) = \int x_i \, \left( \bar{k} - \frac{1}{2} \, \Tr \, \bar{h} \right) \, \e^{ - \bar{f}} \, .
\end{align}
 
 The next proposition constructs a vector field $V$  giving a diffeomorphism that makes $\dv_{\bar{f}} \, \bar{h}$ 
 and $\cB$  small  relative to the scale that we are working on.
 
 \begin{Pro}	\label{p:solvegauge}
 There exists $R_0 > 0$ so that if $R \geq R_0$ and $h,k$ have support in $b \leq R$ and satisfy \eqr{e:3p2},
then there exists $V$ with support in $b \leq R$ so that
  $\| V \|_{C^{5,\alpha}} \leq \e^{ - \frac{3}{4}\, a_0 \, R^2}$, \, $|\cB (V) | \leq 2 \, \e^{ - \frac{(R-1)^2}{4}} $ and so that:
  \begin{itemize}
  \item $\dv_{\bar{f}} \, \bar{h}$ vanishes unless $b \in [R-1 , R]$ and satisfies $\|  \dv_{\bar{f}} \, \bar{h}  \|_{C^{3,\alpha}} \leq \e^{ - \frac{1}{2} \, a_0 \, R^2} $.
  \item   $\bar{h}$ and $\bar{k}$ are supported in $b \leq R$ and  satisfy $\| \bar{h} \|_{C^{4,\alpha}} + \| \bar{k} \|_{C^{4,\alpha}} \leq \e^{ - \frac{1}{2} \, a_0 \, R^2}$.
\end{itemize}
\end{Pro}

\vskip2mm
If $J(V) = \cP \, V$,  then
  $\dv_{\bar{f}} \, \bar{h} = 0$, so we would like to solve the nonlinear equations 
\begin{align}	\label{e:pairss}
	 J(V) = \cP \, V {\text{ and }} \cB(V) = 0 \, .
\end{align}
 We could do this if $M$ was closed.  To deal with the non-compactness, we will instead solve this up to error terms from a cutoff function $\eta$.
    We will reformulate this version of \eqr{e:pairss} as finding a fixed point 
  for a nonlinear mapping ($\zeta$ in \eqr{e:zetaDEF}) that we will show   is a type of  contraction mapping.  There will be a two additional subtleties. 
The   first
  is that $\cP$ has a kernel, so we have to work orthogonally to this.  The second
  issue  is that  $\zeta$ will be contracting only in a weak  sense explained below.
  
  \vskip1mm
     The starting point is to 
   understand how $\bar{h}$ and $\bar{k}$ depend on $h, k$ and $V$.

\begin{Lem}	\label{l:barhP}
We have that
\begin{align}
	\bar{h}_{ij} (x) & =  \left( \bar{g}_{mn} (y) \,  
	  \gamma^m_i \, \gamma^n_j - \bar{g}_{ij} \right) 
	  +   h_{mn} (y) \,  
	 \gamma^m_i \, \gamma^n_j +  \bar{g}_{mn} (y) \,  
	 \left(  \gamma^n_j \, \gamma^m_{\eta_p}  \, V^p_i  +
	 \gamma^m_i \,   \gamma^n_{\eta_q}  \, V^q_j   
	 \right) \notag \\
	 &+ [\bar{g} , \gamma_{\eta} , \gamma_{\eta} , \nabla V , \nabla V] +   [h , \gamma_{\eta} , \gamma_{\eta} , \nabla V , \nabla V] 
	 + [ h , \gamma_x , \gamma_{\eta} , \nabla V ] \, , 
\end{align}
where the terms involving $h$ are evaluated at $y$ and the terms on the last row are multi-linear combinations of the listed quantities.
\end{Lem}

\begin{proof}
The chain rule gives that the differential of the map $y$ is 
\begin{align}
	y^i_j   & = \gamma^i_j (x,V(x)) + \gamma^i_{\eta_m} (x,V(x)) \, V^m_j   \, ,
\end{align}
where terms are evaluated at $x$ unless specified otherwise.
Thus, we see that
\begin{align}
	[ y^*(\bar{g} +h)]_{ij}  &=  (\bar{g} +h)_{mn} (y) \, y^m_i \, y^n_j = 
	 (\bar{g} +h)_{mn} (y) \,  
	 \left( \gamma^m_i + \gamma^m_{\eta_p}  \, V^p_i  
	 \right) \,  \left( \gamma^n_j + \gamma^n_{\eta_q}  \, V^q_j  
	 \right) \, , 
\end{align}
where $\gamma$'s are always evaluated at $(x,V(x))$. Expanding this out gives
the claim.
\end{proof}

The next proposition shows that $J$ is bounded from $C^{5,\alpha}$ to $C^{3,\alpha}$ and is Lipschitz from  $C^{4,\alpha}$ to $C^{2,\alpha}$.  This loss of a derivative in the Lipschitz property will result in $\zeta$ (defined in \eqr{e:zetaDEF}))
only being contracting on $C^{4,\alpha}$, which complicates the fixed point argument.

\begin{Pro}	\label{p:JA}
If $V, W \in C^{5,\alpha}$  have    $\| V \|_{C^0} , \| W \|_{C^0} \leq \delta_0$,  $\| V \|_{C^5} , \| W \|_{C^5} \leq 1$, and vanish on $b>R$, then
\begin{align}
	\| J (V) \|_{C^{3,\alpha}} &\leq C\, R \, \|   h \|_{C^{4,\alpha}} +  C \, R \,  \left\{ \| V \|_{C^{5,\alpha}}^2 + \| h \|_{C^{4,\alpha}} \, \| V \|_{C^{5,\alpha}}
	\right\} 
		   \, , \label{e:J1} \\
	\| J (V) - J (W)  \|_{C^{2,\alpha}} &\leq C\, R  \, \left(  \| h \|_{C^{4, \alpha}} + \| V \|_{C^{4,\alpha}}  
	+ \| V - W\|_{C^{4,\alpha}}   \right)  \, \| V - W \|_{C^{4,\alpha}} \, . \label{e:J2}
	\end{align}
\end{Pro}

\begin{proof}
The proof is elementary, but involved.  For simplicity, we will explain it in the case where  $\gamma (x,\zeta) =x + \zeta$;
the general case follows similarly with  additional terms from the differential of the exponential map and its derivatives.
In this case, we have
\begin{align}
	\bar{h}_{ij} &=  h_{ij}  + V^j_i + V^i_j +
	[\nabla V , h] + [\nabla V , \nabla V]  + [\nabla V , \nabla V , h] \, ,
\end{align}
where the $h$ terms on the right are evaluated at $y= y_V = x + V(x)$ and $[\nabla V , h]$ denotes a term that is linear in both $\nabla V$ and $h$, etc.
Using this, we get that
  \begin{align}
 	J(V) &=   {\textcolor{blue}{ (\dv \, h)   - h_{ij} \, \bar{f}_j }} +  
	{\textcolor{red}{ [\nabla^2 V , h] + [\nabla V , \nabla h] + [\nabla V , h , \nabla \bar{f} ] }} 	\notag \\
	&\qquad   +
	[\nabla V , \nabla V , \nabla h] 
 + [\nabla^2 V , \nabla V , h]  
	  + [\nabla V , \nabla V , \nabla V , \nabla h] + [\nabla V , \nabla V , h , \nabla \bar{f}]  \\
	  &  \qquad  + [\nabla^2 V , \nabla V]+ [\nabla V , \nabla V , \nabla \bar{f} ]  
	\, . \notag
 \end{align}
 The terms on the right appear in three groups.  The blue terms are linear in $h$ and have no $\nabla V$'s; the red terms are bilinear in $h$ (or $\nabla h$) and $\nabla V$ (or $\nabla^2 V$); the black terms are at least quadratic in $\nabla V$ (or $\nabla^2 V$).
 
 \vskip1mm
 {\bf{Proving \eqr{e:J1}}}.  The first blue term is bounded in $C^{0}$ by $\| \dv \, h\|_{C^0}$, while the second one is bounded by
 $R \, \| h \|_{C^0}$ (the $R$ comes from a bound for $\nabla \bar{f}$ on $b< R$).  The three red terms 
 are bounded in $C^0$ by $\| V \|_{C^2} \, \| h \|_{C^0}$, $\| V \|_{C^1} \, \| h \|_{C^1}$ and $R \, \| h \|_{C^0} \, \| V \|_{C^1}$, respectively. 
  The black terms are all at least quadratic 
 in $V$ and depend on at most one derivative of $h$ and two derivatives of $V$.  Differentiating and arguing similarly gives the  $C^{3,\alpha}$ bound for $J$.
 
 \vskip1mm
 {\bf{Proving \eqr{e:J2}}}.  Using the fundamental theorem of calculus, the difference in the first blue terms for $J(V)$ and $J(W)$ is bounded by{\footnote{We need one more derivative on $h$ than in \eqr{e:J1}; this is why we use  the $C^{1,\alpha}$ norm of the difference.}}
 \begin{align}
 	| (\dv \, h)(y_V) - (\dv \, h)(y_W)| &\leq \|\nabla \dv \, h \|_{C^0} \, |y_V - y_W| = \|\nabla \dv \, h \|_{C^0} \, |V - W| \notag \\
	&\leq C \, \| h \|_{C^2} \, \| V-W \|_{C^0} \, .
 \end{align}
 The second blue term is similar, but has a factor of $R$ because of the $\nabla \bar{f}$ term.  The three red terms and the black terms on the second line all follow similarly using the triangle inequality, with each also giving a bound with a factor of $h$ in it.
 The two black terms on the last line are slightly different since there is no $h$ appearing.  To handle the first term on the last line, we use the triangle inequality to get
 \begin{align}
 \left| [\nabla^2 V , \nabla V] - [\nabla^2 W , \nabla W] \right| &\leq  \left| [\nabla^2 V , \nabla V] - [\nabla^2 V , \nabla W] \right|
 +  \left| [\nabla^2 V , \nabla W] - [\nabla^2 W , \nabla W] \right| \notag \\
 &\leq \| V \|_{C^2} \, \| V -W \|_{C^1} + \| W \|_{C^1} \, \| V -W \|_{C^2} \\
 &\leq \| V \|_{C^2} \, \| V - W \|_{C^2} + \| V - W \|_{C^2}^2  \, . \notag 
 \end{align}
 The other term is similar, but with an extra $R$ factor.  The $C^{2,\alpha}$ estimates follow similarly.
\end{proof}

We turn next to $\cB$.  It is useful to 
let $\langle V , W \rangle_{L^2} = \int \langle V , W \rangle \, \e^{ - \bar{f}}$ be the weighted $L^2$ inner product for vector fields.
The lemma shows that $\cB$ is bounded and Lipschitz.

\begin{Lem}	\label{l:Bmap}
The map $\cB$ satisfies
\begin{align}
	\left| \cB^i (V) - \cB^i (0) - \langle \partial_i , V \rangle_{L^2} \right| &\leq  C \, R \, \left( \| h \|_{C^1} + \| k \|_{C^1} + \| V \|_{C^1}
	\right) \, \| V \|_{C^1}
	\, ,  \label{e:BB1} \\
	\left| \cB^i (V) - \cB^i (W) - \langle \partial_i , (V-W) \rangle_{L^2} \right| &\leq
	C \, R \, \left( \| h \|_{C^1} + \| k \|_{C^1} + \| V \|_{C^1}
	\right) \, \| V - W \|_{C^1} \notag \\
	&\quad + C \, R \, \| V - W \|_{C^1}^2
	\, . \label{e:BB2}
\end{align}
\end{Lem}

\begin{proof}
As in the proof of Proposition \ref{p:JA}, we will suppress the error terms involving the differential of the exponential map, so that
\begin{align}
	\bar{h}_{ij} &=  h_{ij}  + V^j_i + V^i_j +
	[\nabla V , h] + [\nabla V , \nabla V]  + [\nabla V , \nabla V , h] \, , \\
	\bar{k} &=  \bar{f} (y) - \bar{f} + k  \, , 
\end{align}
where $h$ and $k$ are evaluated at $y= y_V = x + V(x)$. To simplify notation, set $w = k - \frac{1}{2} \, \Tr \, h$ and 
$\bar{w} = \bar{k} - \frac{1}{2} \, \Tr \, \bar{h}$. 
In particular, we see that 
\begin{align}
	\bar{w}   - w &=  (w (y)- w)  + (\bar{f}(y) - \bar{f} ) - \dv \, V + [\nabla V , h] + [\nabla V , \nabla V]  + [\nabla V , \nabla V , h]
	\, . \notag
\end{align}
Since $\dv_{\bar{f}} \, (x_i \, V) = \langle \partial_i , V \rangle + x_i \, \dv \, V - x_i \, \langle V , \nabla \bar{f} \rangle$, integration by parts gives that
\begin{align}
	-\int x_i \, \dv (V) \, \e^{ - \bar{f}} =  \langle \partial_i , V \rangle_{L^2} - \int x_i \langle V , \nabla \bar{f} \rangle \, \e^{ - \bar{f}} \, .
\end{align}
Thus, we see that
\begin{align}
	\cB^i (V) &- \cB^i (0) - \langle \partial_i , V \rangle_{L^2} =
	{\textcolor{blue}{ \int 
	x_i \, \left( \bar{f}(x+V(x)) - \bar{f} - \langle V , \nabla \bar{f} \rangle \right)
	\, \e^{ - \bar{f}} }} \notag \\
	&+ {\textcolor{red}{\int x_i \, \left( w(y) - w \right) 
	\, \e^{ - \bar{f}} }}  + \int x_i \, 
	\left\{   [\nabla V , h] + [\nabla V , \nabla V]  + [\nabla V , \nabla V , h]
	\right\} \, \e^{ - \bar{f}} \, . \notag
\end{align}
Since the hessian of $\bar{f}$ is bounded, the blue term is bounded by $C \, R \, \| V \|_{C^0}^2$.  The red term is bounded by
$R\, \| w \|_{C^1} \, \| V \|_{C^0} \leq C \, R \, (\| h \|_{C^1} + \| k \|_{C^1} ) \, \| V \|_{C^0}$.  Finally, the three black terms are bounded $R \, \| V \|_{C^1} \, \| h\|_{C^0}$, $R \, \| V \|_{C^1}^2$ and $R\, \| V \|_{C^1}^2 \, \| h \|_{C^0}$, respectively.
This gives \eqr{e:BB1}.

We turn next to \eqr{e:BB2}.  The red term contributes $R \, (\| h \|_{C^1} + \| k \|_{C^1}) \, \| V - W \|_{C^0}$.  To bound the first black term, observe that
\begin{align}
	 |[\nabla V , h(y_V)] - [\nabla W , h(y_W)]| &\leq  |[\nabla V , h(y_V)] - [\nabla V , h(y_W)]| +  |[\nabla V , h(y_W)] - [\nabla W , h(y_W)]| \notag \\
	 &\leq  C \, \| V \|_{C^1} \, \| h \|_{C^1} \, \| V - W \|_{C^0} + 
	 C\,  \| h \|_{C^0} \, \| V - W \|_{C^1} \, .
\end{align}
The other black terms and the blue term are similar.
\end{proof}

We will construct $V$ iteratively, using a type of contraction mapping argument.  The sequence of vector fields will stay bounded in $C^{5,\alpha}$, but will converge in $C^{4,\alpha}$.  The next lemma will be used for boundedness.

\begin{Lem}	\label{l:seqA}
Given $C$, there exists $\epsilon_0 > 0$ so that if $ \epsilon < \epsilon_0$ and $c_i \geq 0$ is a sequence with $c_0 \leq 2 \, C \, \epsilon$ and 
$c_{i+1} \leq C \, ( \epsilon + c_i^2 + \epsilon \, c_i)$, then 
\begin{align}
	c_i \leq 2 \, C \, \epsilon {\text{ for every }} i \, .
\end{align}
\end{Lem}

\begin{proof}
We will prove this inductively.  It is true for $i=0$ by assumption.  Suppose it is true now for $i$.  We get  that
\begin{align}
	c_{i+1} &\leq C \, ( \epsilon + c_i^2 + \epsilon \, c_i) 
	\leq C \, \epsilon + C \, \left\{ (2 \, C \, \epsilon)^2  + \epsilon \,  (2 \, C \, \epsilon) \right\} \notag \\
	&=  C \, \epsilon + C \, \epsilon \, 
	\left\{ 4 \, C^2 \, \epsilon  + 2 \, C \, \epsilon
	\right\} 
	\, .
\end{align}
To ensure that this is at most $2\, C \,\epsilon$, we need that
\begin{align}
	 4\, C^2 \, \epsilon  + 2\, C \, \epsilon \leq  1 \, .
\end{align}
This holds for $\epsilon_0 > 0$ sufficiently small and the  lemma follows.
\end{proof}

\begin{proof}[Proof of Proposition \ref{p:solvegauge}]
Define the constant $\omega > 0$ to be the weighted volume $\omega = \int \e^{ - \bar{f}}$.
Fix a smooth cutoff function $\eta$ that depends only on $\RR^{n-\ell}$, has   support in $b < R$,   is identically one on $b < R -1$, and has
$\| \eta \|_{C^{5,\alpha}} \leq C$ where $C$ is independent of $R$. Given a vector field with support in $b\leq R$, define a new vector field $\zeta (V)$ supported in $b\leq R$ by
   \begin{align}	\label{e:zetaDEF}
   	\zeta (  V)  = \eta \, \left( \cP^{-1} (J(  V)) - \frac{1}{\omega} \, \sum_i \, \left( \cB^i ( V) - \langle \partial_i , ( V) \rangle_{L^2} \right) \, \partial_i 
	\right)  \, .
   \end{align}
   Note that the compact support of $V$ and the definition of $J$ (it has a $\dv_{\bar{f}}$ in it) ensure that $J(V)$ is in $\cK^{\perp}$ and, thus, $\cP^{-1} (J(V))$ is defined by Proposition \ref{p:cPi2A}.
Using the Leibniz rule for H\"older norms and the triangle inequality, we see that
\begin{align}
	\| \zeta (V) \|_{C^{5,\alpha}} \leq C \, \sum_ i   \left| \cB^i (V) - \langle \partial_i , V \rangle_{L^2} \right| 
	+ C \,   \| \cP^{-1} (J(V)) \|_{C^{5,\alpha}_R}  
	\, .
\end{align}
Using the first claim in Lemma \ref{l:Bmap} on the first term and 
Proposition \ref{p:cPi2A} and the first claim in Proposition
\eqr{p:JA} on the second term gives that 
\begin{align}	\label{e:zeta1}
	\| \zeta (V) \|_{C^{5,\alpha}} & \leq C \,   R \, \left( \| h \|_{C^0} + \| k \|_{C^0} \right)  +
	   C \, R \, \left( \| h \|_{C^1} + \| k \|_{C^1} + \| V \|_{C^1}
	\right) \, \| V \|_{C^1}
 \notag \\
	&+ C \,   R^{m+1} \, 
 \left( 
  \|   h \|_{C^{4,\alpha}} +   \left\{ \| V \|_{C^{5,\alpha}}^2 + \| h \|_{C^{4,\alpha}} \, \| V \|_{C^{5,\alpha}}
	\right\}  \right)
		   \, .
\end{align}
We will need a Lipshitz property for $\zeta$.  Using again the Leibniz rule for H\"older norms and the triangle inequality, we get that
\begin{align}
	\| \zeta(V) - \zeta(W) \|_{C^{4,\alpha}} \leq C \, \sum_ i   \left| \cB^i (V)- \cB^i (W)  - \langle \partial_i , V -W\rangle_{L^2} \right| 
	+ C \,   \| \cP^{-1} (J(V) - J(W)) \|_{C^{4,\alpha}_R}  \notag \, .
\end{align}
Using the second claim in  Lemma \ref{l:Bmap} on the first term and 
Proposition \ref{p:cPi2A} and the second claim in Proposition
\eqr{p:JA} on the second term gives that
\begin{align}
	 \| \zeta(V) - \zeta(W) \|_{C^{4,\alpha}} &\leq 
	C \, R \, \left( \| h \|_{C^1} + \| k \|_{C^1} + \| V \|_{C^1}
	\right) \, \| V - W \|_{C^1}   + C \, R \, \| V - W \|_{C^1}^2 \notag \\
	&  + C \, R^{m+1} \,    \left(  \| h \|_{C^{4, \alpha}} + \| V \|_{C^{4,\alpha}}  
	+ \| V - W\|_{C^{4,\alpha}}   \right)  \, \| V - W \|_{C^{4,\alpha}} 
	 \, .  	\label{e:zeta2}
\end{align}

Define a sequence of vector fields by setting $V_0 = 0$ and
\begin{align}
	V_{i+1} =  \zeta (V_i) \, .
\end{align}
Set $c_i = \| V_i \|_{C^{5,\alpha}}$ and $d_i = \| V_i - V_{i-1} \|_{C^{4,\alpha}}$. 
The estimate \eqr{e:zeta1} allows us to apply 
Lemma \ref{l:seqA} to get that 
\begin{align}
	c_i \leq C \, R^{m+1} \, \| h \|_{C^{4,\alpha}} + C \, R \, \| k \|_{C^1} \, .
\end{align}
Using this in \eqr{e:zeta2}, we see that the $d_i$'s decay  geometrically and, thus, are summable.  This gives that the sequence $V_i$ converges in $C^{4,\alpha}$ to a limiting vector field $V \in C^{4,\alpha}$.  It follows from this and the continuity of $\zeta$ that $\zeta (V) = V$ since
\begin{align}
	\| \zeta (V) - V \|_{C^{4,\alpha}} \leq \| \zeta (V) - \zeta (V_i) \|_{C^{4,\alpha}} + \| \zeta (V_i) - V_i \|_{C^{4,\alpha}} + \| V_i - V \|_{C^{4,\alpha}} \, .
\end{align}
Since the sequence is uniformly bounded in $C^{5,\alpha}$, the limit $V$ satisfies the same $C^{5,\alpha}$ bound.  This gives the first claim.

We will use that $\zeta (V) = V$ to show that $V$ has the required bounds on $\dv_{\bar{f}}$ and $\cB$.    Taking the inner product of $V$ with $\partial_i$ gives
\begin{align}
	\langle V , \partial_i \rangle_{L^2} &= \langle \zeta (V) , \partial_i \rangle_{L^2} = 
	\langle \eta \cP^{-1} (J(V)) , \partial_i \rangle_{L^2} - \frac{\cB^i (V) - \langle \partial_i , V \rangle_{L^2}}{\omega} \,  \int \eta \, \e^{- \bar{f}}\, .
\end{align}
To bound the first term on the right, use that $\partial_i \in \cK$ and $\cP^{-1}$ maps into $\cK^{\perp}$ to get 
\begin{align}
	\left| \langle \eta \cP^{-1} (J(V)) , \partial_i \rangle_{L^2} \right| &= \left|  \langle (1-\eta) \,  \cP^{-1} (J(V)) , \partial_i \rangle_{L^2} \right| \leq \int_{b> R-1} | \cP^{-1} (J(V))| \, \e^{- \bar{f}} \leq  \e^{ - \frac{R^2}{4}} \, ,
	\notag 
\end{align}
where the last inequality used the second claim in Proposition \ref{p:cPi2A}. Using this and the fact that $\int \eta \, \e^{-\bar{f}}$ is exponentially close to $\omega$, we see that
\begin{align}
	|\cB (V) | \leq 2 \, \e^{ - \frac{(R-1)^2}{4}} \, .
\end{align}
This gives the second claim.  

 For the last claim, we consider three different regions depending on $b$.  When $b> R$, then $\eta \equiv 0 $ and $h \equiv 0$, so $\dv_{\bar{f}} \, \bar{h} \equiv 0$.
 When $b < R -1$, then $\eta \equiv 1$, so that applying $\cP$ to $V = \zeta (V)$ gives that
 \begin{align}
 	\cP \, V = J(V) = \frac{1}{2} \, \dv_{\bar{f}} (\bar{h}) + \cP \, V \, , 
 \end{align}
 so we see again that $\dv_{\bar{f}} \, \bar{h} = 0$ here.  Finally, we turn to the intermediate region where $R_1 < b < R$.  
 On this region, we simply use the $C^{5,\alpha}$ bound on $V$ and the first claim in Proposition
\eqr{p:JA} to get that
\begin{align}
	\| J(V) \|_{C^{3,\alpha}} \leq \e^{ - \frac{3}{4} \, a_0 \, R^2} \, .
\end{align}
Since $J(V) = \frac{1}{2} \, \dv_{\bar{f}} (\bar{h}) + \cP \, V$, combining this with the bound on $\| V \|_{C^{5,\alpha}}$ again gives that
\begin{align}
	\|  \dv_{\bar{f}} \, \bar{h}  \|_{C^{3,\alpha}} \leq \e^{ - \frac{1}{2} \, a_0 \, R^2} \, .
\end{align}
Finally, the bounds on $\bar{h}$ and $\bar{k}$ follow from the initial bounds on $h$ and $k$ together with the bounds on the vector field $V$ (notice that we need one more derivative on $V$ because the pull-back depends on its differential).
 \end{proof}

  \subsection{The improvement}

  \begin{proof}[Proof of Theorem \ref{t:improve}]
 Fix a smooth cutoff function $\eta$ with support in $\bar{b} \leq R$ and that is  one on $\bar{b} \leq {R-1}$.
Set $h_0 =  \eta (\Psi_R^* \, g - \bar{g})$ and $k_0 =
\eta (  f \circ \Psi_R - \bar{f})$, so that
 	\begin{align}	\label{e:hzerosmall}
		\left| h_0 \right|_{C^{4,\alpha}}^2 + \left| k_0 \right|_{C^{4,\alpha}}^2   
	 \leq   C\, \e^{ -  a_0\,  R^2 } 
	\end{align}
	 and 
	 $\left| h_0 \right|_{C^{\ell}} + \left| k_0 \right|_{C^{\ell}} 
	 \leq C_{\ell} \, R^{\ell}$.	  
	 Proposition \ref{p:solvegauge} gives a diffeomorphism $\Phi$ so that 
 $|\cB (V) | \leq 2 \, \e^{ - \frac{(R-1)^2}{4}} $ and  
  \begin{itemize}
  \item $\dv_{\bar{f}} \, \bar{h}$ vanishes unless $b \in [R-1 , R]$ and satisfies $\|  \dv_{\bar{f}} \, \bar{h}  \|_{C^{3,\alpha}} \leq \e^{ - \frac{1}{2} \, a_0 \, R^2} $.
  \item   $\bar{h}$ and $\bar{k}$ are supported in $b \leq R$ and  satisfy $\| \bar{h} \|_{C^{4,\alpha}} + \| \bar{k} \|_{C^{4,\alpha}} \leq \e^{ - \frac{1}{2} \, a_0 \, R^2}$.
\end{itemize}
	 Let
$u\,g^1$ be the orthogonal projection of $  \bar{h}$ onto $\cK\,g^1$  and write  $  \bar{h}= u \, g^1 + \bb$ and $ \bar{k}=  \frac{\ell}{2} \, u + \psi $.
  Given $\epsilon > 0$, Theorem \ref{p:firstorder} gives  constants $C$ and $C_{\epsilon}$ so that 
 \begin{align}
     	&\|  \bb  \|_{W^{2,2}}^2  + \|  \bar{\nabla} \psi  \|_{W^{1,2}}^2  \leq C \, \left\{ \| \phi (1) \|_{L^2}^2 + \| \dv_{\bar{f}} \, \bar{h} \|_{W^{1,2}}^2 +     | \cB  |^2 +  \| u \|_{ L^2}^{ 4}  \right\}  \, ,	 \notag \\
	& \|u\|^2_{L^2}\leq   C\,\left\{\| u \|_{L^2}^{ 3} +\|\phi (1)\,(1+|x|^2)\|_{L^1} \right\} + C_{\epsilon} \, \left\{ \| \phi (1) \|_{L^2}^{2-\epsilon}  + | \cB |^{2-\epsilon} +\| \dv_{\bar{f}} \, \bar{h} \|_{W^{1,2}}^{2-\epsilon} \right\}\, . \notag
\end{align}
Since projection cannot increase the norm, we have $\| u\, g^1 \|_{L^2} \leq \|  \bar{h} \|_{L^2}^2$ and, thus, we can absorb the $\| u \|_{L^2}^3$ term in the second equation into the left-hand side.
Using the bounds from Proposition \ref{p:solvegauge}, we see that the remaining terms on the right-hand side are all of the order 
$\e^{ - \frac{(R-1)^2}{4} \, (1 - \epsilon/2)}$, so we see that 
  \begin{align}
   &\|u\|^2_{L^2}\leq C \,  \e^{ - \frac{(R-1)^2}{4} \, (1 - \epsilon/2)}   \, .
   \end{align}
 Using this in the first equation, we see that   $\|  \bb  \|_{W^{2,2}}^2$ and $ \|  \bar{\nabla} \psi  \|_{W^{1,2}}^2$ are of the order $\e^{ - \frac{(R-2)^2}{4}}$.  To see that $\bar{k}$ itself, and not just $\bar{\nabla}\bar{k}$, is small, we use the previous bounds and the normalizations $S + |\nabla f|^2 = f$ and $\bar{S} + |\bar{\nabla} \bar{f}|^2 = \bar{f}$. 
 Finally, choosing $\epsilon > 0$ small, this gives the desired $W^{2,2}$ bounds on the scale $\theta \, R$ (where we also guaranteed that introducing the cutoff has not changed the metric), completing the proof.
 \end{proof}

We will next use the strong rigidity of Theorem \ref{t:main} to prove that if one tangent flow is a cylinder, then every tangent flow is.  To make this precise, 
let   $ \tilde{g}(t))$  a Ricci flow on $M \times [T,0)$ that has a singularity at $t=0$  where the conclusions of theorem $1.4$ in \cite{MM} hold; this includes closed manifolds with type-I singularities.

   \begin{Thm}		\label{t:uniquetype}
 If $M, \tilde{g}$ is a Ricci flow as above and one tangent flow at a point is a cylinder, then every tangent flow at that point is a cylinder (with the same $\ell$).
   \end{Thm}

\begin{proof}
As in \cite{MM}, by solving the conjugate heat equation, continuously rescaling and reparameterizing the Ricci flow gives a solution
$(M,g(t),f(t))$ of the 
rescaled Ricci flow equation where a sequence of times converges to a cylinder $\Sigma$.  The curvature bound assumed in \cite{MM} (see ($1.2$) there) and the Shi estimates, \cite{S},  bound all derivatives of the flow.

We will argue by contradiction.  Suppose instead that  $t_j , t_j'$ are sequences going to infinity with $t_j < t_j' < t_{j+1} < t_{j+1}' \dots$ and so that
\begin{enumerate}
\item $(M, g(t_j) , f(t_j))$ converges to   $\Sigma$.
\item  $(M, g(t_j') , f(t_j'))$ converges to a different shrinker.
\end{enumerate}
Theorem \ref{t:rigidP} gives an $R$   so that if ($\dagger_R$) (relative to $\Sigma$) holds for a shrinker, then the shrinker agrees identically with $\Sigma$ (up to a diffeomorphism).  

By (1), we have that ($\dagger_{2R}$) holds for every $t_j$ sufficiently large.
On the other hand, (2) implies that ($\dagger_{2R}$) must fail for $t_j'$ sufficiently large.  Since $g$ and $f$ vary continuously in $t$, there must be a first $s_j \in (t_j , t_j')$ where ($\dagger_{2R}$) fails.  In particular, using also that we have uniform higher derivative bounds, we see that ($\dagger_R$) holds at $s_j$.
Theorem $1.4$ in \cite{MM} gives 
that a subsequence of the $s_j$'s gives a limiting shrinker $(M,\bar{g} , \bar{f})$, where the convergence is smooth on compact subsets.  On the one hand, this limit must be different from $\Sigma$ since
($\dagger_{2R}$)  fails at every $s_j$.  On the other hand, ($\dagger_R$) holds for the limiting shrinker, so Theorem \ref{t:main} implies that it agrees with $\Sigma$ giving the desired contradiction.
\end{proof}

\end{document}